\documentclass[11pt]{article}
\usepackage[a4paper,
bindingoffset=0.2in,
left=0.5in,
right=0.5in,
top=0.4in,
bottom=0.6in,
footskip=.4in]{geometry}
\usepackage{amsthm,amsmath,amssymb,amsfonts,amscd}
\usepackage{graphicx}
\usepackage{enumerate}
\usepackage[all]{xy}
\usepackage{booktabs}
\usepackage{mathrsfs}
\usepackage{hyperref}
\usepackage[nocompress]{cite}
\usepackage{caption}
\usepackage{subcaption}
\usepackage{verbatim}
\usepackage{xcolor}
\usepackage[normalem]{ulem}
\usepackage{mathtools}
\usepackage{tikz}
\usetikzlibrary{arrows}

\mathtoolsset{showonlyrefs,showmanualtags}

\newtheorem{theorem}{Theorem}[section]

\newtheorem{prop}[theorem]{Proposition}
\newtheorem{lemma}[theorem]{Lemma}
\newtheorem{cor}[theorem]{Corollary}

\theoremstyle{definition}
\newtheorem{defi}[theorem]{Definition}

\theoremstyle{remark}
\newtheorem{remark}[theorem]{Remark}

\DeclareMathOperator*{\esssup}{ess\,sup}
\DeclareMathOperator*{\essinf}{ess\,inf}
\DeclareMathOperator*{\esssupp}{ess\,supp}

\numberwithin{equation}{section}

\title{Nonlocal BV and nonlocal Sobolev spaces\\ induced by nonfractional weight functions}

\author{Francesc Alcover\thanks{DMI \& IAC3, University of the Balearic Islands, Cra. de Valldemossa, km. 7.5, E-07122 Palma, Illes Balears, Spain. E-mail: \texttt{francesc.alcover@uib.cat}, \texttt{joan.duran@uib.es}, \texttt{catalina.sbert@uib.es}} \and Joan Duran\footnotemark[1] \and Ramon Oliver-Bonafoux\thanks{Dipartimento di Informatica, Università di Verona. Strada le Grazie 15, 37134 Verona, Italy. E-mail: \texttt{ramon.oliverbonafoux@univr.it}} \and Catalina Sbert\footnotemark[1]}

\date{\today}

\begin{document}

\maketitle

\begin{abstract}
In this paper, we expand upon the theory of the space of functions with nonlocal weighted bounded variation, first introduced by Kindermann et~al.~in 2005 and later generalized by Wang et~al.~in 2014. We consider nonfractional $\mathcal{C}^1$ weights and, using an analogous formulation to the aforementioned works, we also introduce a (to our knowledge) new class of nonlocal weighted Sobolev spaces. After establishing some fundamental properties and results regarding the structure of these spaces, we study their relationship with the classical $\mathrm{BV}$ and Sobolev spaces, as well as with the space of test functions. We handle both the case of domains with finite measure and that of domains of infinite measure, and show that these two situations lead to quite different scenarios. As an application, we also show that these function spaces are suitable for establishing existence and uniqueness results of global minimizers for several classes of functionals. Some of these functionals were introduced in the above-mentioned references for the study of image deblurring problems.
\end{abstract}



\section{Introduction} \label{IntroductionSection}
In recent years, fractional calculus has established itself at the forefront of nonlocal analysis, due in part to its effectiveness at modeling various phenomena in science and engineering. However, the study of nonlocal operators induced by another class of not necessarily radial nor singular weights has also seen a significant rise, specially in an image processing context $-$a historical summary of this rise will be given later in this section. Among this wave of new nonlocal operators, Kindermann, Osher and Jones introduced in \cite{Kindermann} a notion of \emph{nonlocal total variation} to be used in image deblurring problems, and studied its corresponding \emph{space of functions with nonlocal bounded variation}. A weighted variant was later proposed by Wang and Ng in \cite{Wang-Ng} with the scope of addressing problems in image decomposition.

As introduced in \cite{Wang-Ng}, we define the \emph{space of functions with nonlocal weighted bounded variation} as follows: Let $\Omega$ be a nonempty open subset of $\mathbb{R}^N$. We will say that $\omega:\Omega\times\Omega \to [0,+\infty)$ is an admissible weight function in $\Omega$ if $\omega$ is nonnegative, symmetric and belongs to $\mathcal{C}^1$ in $\Omega \times \Omega$\footnote{Positivity and symmetry were the only hypotheses in \cite{Wang-Ng}, with the regularity of $\omega$ unmentioned to our knowledge. In this article we impose that $\omega$ is $\mathcal{C}^1$.}. Notice, however, that no assumption is made on the behavior of $\omega$ near the boundary of $\Omega \times \Omega$, so that one could in principle have that $\omega(x,y) \to +\infty$ as $(x,y)$ approaches some point at the boundary of $\Omega \times \Omega$. Given a fixed admissible weight function $\omega$, we first consider the \textit{nonlocal divergence operator} (associated to $\omega$), defined for $\phi=(\phi_1,\ldots,\phi_N) \in \mathcal{C}^1(\Omega \times \Omega;\mathbb{R}^N)$ as
\begin{equation}
\mathrm{div}_\omega(\mathbf{\phi})(x):= \sum_{i=1}^N \partial^*_{\omega,i}(\phi_i)(x) \hspace{2mm} \mbox{ for all } x \in \Omega,
\end{equation}
where, for all $i = 1,\ldots,N$ and $\psi \in \mathcal{C}^1(\Omega \times \Omega)$, $\partial^*_{\omega,i}(\psi)$ stands for the \textit{nonlocal partial derivative}, defined as
\begin{equation} \label{NlStrongDerIntro}
\partial_{\omega, i}^\ast (\psi) (x) := \int_\Omega \frac{\partial}{\partial x_i} \left[ \omega(x,y) (\psi(y,x)-\psi(x,y)) \right] \mathrm{d}y \hspace{2mm} \mbox{ for all } x \in \Omega.
\end{equation}

This divergence operator then naturally leads to defining the \textit{nonlocal total variation} (induced by $\omega$) of an $L^1(\Omega)$ function $u$ as follows:
\begin{equation}\label{NLTVDefIntro}
\mathrm{NLTV}_\omega(u):= \sup\left\{ \int_\Omega u(x) \mathrm{div}_\omega(\phi)(x)\mathrm{d}x: \phi \in \mathcal{C}^1_c(\Omega \times \Omega;\mathbb{R}^N), \|\phi\|_{\infty} \le 1 \right\} \in [0,+\infty].
\end{equation}

In the definition above and throughout this paper, $\| \cdot\|_\infty$ denotes the uniform norm on $\mathcal{C}_0(\Omega\times\Omega; \mathbb{R}^N)$ given by

\begin{equation}
    \|\phi\|_\infty := \sup_{x\in \Omega} |\phi(x)| \quad \mbox{ for all } \phi \in \mathcal{C}_0(\Omega\times\Omega; \mathbb{R}^N),
\end{equation}
where $|\cdot|$ denotes the Euclidean norm in $\mathbb{R}^N$. An integration by parts shows that
\begin{equation}
\mathrm{NLTV}_\omega(u)=\int_{\Omega \times \Omega} \omega(x,y)\lvert \nabla u(y) -\nabla u(x) \rvert \mathrm{d}x \mathrm{d}y,
\end{equation}
whenever $u \in \mathcal{C}^1(\Omega)$. Subsequently, we define the set
\begin{equation}
\mathrm{NLBV}_\omega(\Omega):= \{ u \in L^1(\Omega):\mathrm{NLTV}_\omega(u)<+\infty\},
\end{equation}
which is a Banach space when endowed with the norm
\begin{equation}
\lVert u \rVert_{\mathrm{NLBV}_\omega(\Omega)}:= \lVert u \rVert_{L^1(\Omega)}+\mathrm{NLTV}_\omega(u).
\end{equation}

As in \cite{Kindermann}, we also denote by $\mathrm{NLBV}(\Omega)$ the space corresponding to the weight which is equal to $1$ everywhere in $\Omega \times \Omega$. While the authors of \cite{Kindermann,Wang-Ng} introduced the aforementioned space and established some related preliminary results, in this paper we conduct a more detailed study of $\mathrm{NLBV}_\omega(\Omega)$: The definitions, with further remarks and properties are given in Subsection \ref{section_NLBV}. Moreover, we consider the link between $\mathrm{NLBV}_\omega(\Omega)$ and $\mathrm{BV}(\Omega)$ according to the behavior of $\omega$. To start off, we prove that if $\Omega$ has finite measure, then $\mathrm{NLBV}(\Omega)=\mathrm{BV}(\Omega)$ with equivalence of norms, see Proposition \ref{NLBV=BVEquivNorm}. Furthermore, we prove that for $\Omega$ with finite measure one has $\mathrm{NLBV}_\omega(\Omega)=\mathrm{BV}(\Omega)$ with equivalence of norms provided that one makes additional assumptions on $\omega$ (in particular, one can assume $\omega$ to be bounded away from $0$ and $+\infty$ in $\Omega \times \Omega$). In fact, we give a sufficient condition for one of the continuous embeddings and a necessary and sufficient condition for the other:
\begin{theorem}[Necessary and sufficient conditions for the equality between $\mathrm{NLBV}_\omega(\Omega)$ and $\mathrm{BV}(\Omega)$ when $\Omega$ has finite measure]\label{Theorem_main_equality_NLBV}
Assume that $\Omega$ is a nonempty open subset of $\mathbb{R}^N$ with finite measure and let $\omega$ be an admissible weight function in $\Omega$. If there exists a constant $k>0$ such that for all $(x,y) \in \Omega \times \Omega$ one has
\begin{equation}\label{Muckenhoupt}
k\le \omega(x,y),
\end{equation}
then, $\mathrm{NLBV}_\omega(\Omega) \hookrightarrow \mathrm{BV}(\Omega)$. Moreover, the converse embedding $\mathrm{BV}(\Omega) \hookrightarrow \mathrm{NLBV}_\omega(\Omega)$ holds if and only if there exists a constant $K>0$ such that for all $x\in \Omega$
\begin{equation}\label{fwpboundedintro}
\int_\Omega \omega(x,y) \mathrm{d}y \le K.
\end{equation}
\end{theorem}
Condition \eqref{Muckenhoupt} implies that $\omega$ belongs to the well-known global Muckenhoupt class in $\Omega \times \Omega$, usually termed as $A_1(\Omega \times \Omega)$ in the literature \cite{Muckenhoupt,GCRB}. That is, there exists a positive constant $\tilde{c}$ such that
\begin{equation}
\frac{\tilde c}{\lvert B((x,y),r) \rvert}\int_{B((x,y),r)} \omega (\tilde x, \tilde y) \mathrm{d}\tilde x \mathrm{d} \tilde y \le \omega(x,y),
\end{equation}
for all $(x,y) \in \Omega \times \Omega$ and $r>0$ such that $B((x,y),r) \subset \overline{\Omega\times \Omega}$. If, in addition, $\Omega$ is bounded and $\omega$ is assumed to be positive everywhere in $\Omega \times \Omega$, then the aforementioned implication becomes an equivalence. At this point, it is worth pointing out that the \textit{local} counterparts of $\mathrm{NLBV}_\omega(\Omega)$ were studied by Baldi\footnote{Indeed, an equivalent definition to \eqref{NLTVDefIntro} (see Lemma \ref{NLTV^omegaLemma}) can be given which closely resembles Baldi's weighted $\mathrm{TV}$ definition.} in \cite{Baldi01}. In particular, \cite[Theorem 4.1]{Baldi01} is somewhat of an analogous result to Theorem \ref{Theorem_main_equality_NLBV}. See also \cite{ChenRao2003,davoli2024} for applications to image denoising, as well as \cite{AthavaleJerrardNovagaOrlandi,CarlierComte,IonescuLachandRobert} for other applications.

The proof of Theorem \ref{Theorem_main_equality_NLBV} is given in our Subsection \ref{GeneralWeightsSubsection}, where we also address other questions. Proving that \eqref{Muckenhoupt} is a sufficient condition for the corresponding embedding will readily follow using a different characterization of $\mathrm{NLTV}_\omega$, see Lemma \ref{NLTV^omegaLemma}. Moreover, proving that \eqref{fwpboundedintro} is a sufficient condition for the converse embedding will also be straightforward and will follow from some simple computations. Instead, proving that \eqref{fwpboundedintro} is also a necessary condition for the embedding $\mathrm{BV}(\Omega) \hookrightarrow \mathrm{NLBV}_\omega(\Omega)$ will be more involved. For a brief description of the argument used to prove this last part, we point the reader to the text below the forthcoming Theorem \ref{Theorem_equality_Sobolev}.

Notice that the embedding $\mathrm{BV} \hookrightarrow \mathrm{NLBV}_\omega$ is ``solved", in the sense that condition \eqref{fwpboundedintro} is a sufficient and necessary condition for the embedding to hold. However, a condition which is only sufficient (see Remark \ref{NotNecRemark}) is given for the converse embedding $\mathrm{NLBV}_\omega \hookrightarrow \mathrm{BV}$. In Subsection \ref{GeneralWeightsSubsection} a necessary condition analogous to \eqref{fwpboundedintro} will be derived, see Proposition \ref{NecCondCompletenessBV}. At the time of writing, we have not been able to check whether it is also sufficient or not, and thus we leave the matter of finding a necessary and sufficient condition for the aforementioned embedding as an open question.

After having focused on the bounded variation setting, a natural question is that of deriving a Sobolev counterpart which mimics the classical relation between $W^{1,1}$ and $\mathrm{BV}$. As it turns out, we can establish such spaces in a natural way using the nonlocal partial derivatives as in \eqref{NlStrongDerIntro}: Given $u \in L^1_{\mathrm{loc}}(\Omega)$ and $i = 1,\ldots, N$, we say that $v_i \in L^1_{\mathrm{loc}}(\Omega \times \Omega)$ is a \textit{weak $i$-th nonlocal partial derivative} of $u$ if for all $\psi \in \mathcal{C}^1_c(\Omega \times \Omega)$ one has
\begin{equation}
\int_\Omega u(x) \partial^*_{\omega,i}(\psi)(x) \mathrm{d}x=-\int_{\Omega \times \Omega} v_i(x,y)\psi(x,y) \mathrm{d}x \mathrm{d}y.
\end{equation}

Similarly to classical weak derivatives, one readily checks that if such a weak nonlocal derivative exists, it is unique in the a.e.~sense, so we systematically set $\partial_{\omega,i}u:= v_i$. Moreover, an integration by parts shows that if $u \in C^1(\Omega)$, then 
\begin{equation}
    \partial_{\omega,i}u(x,y)=\omega(x,y)(\partial_i u(y)-\partial_i u(x)) \quad \mbox{for all } (x,y) \in \Omega \times \Omega.
\end{equation}
When the weak nonlocal derivatives exist for each direction, we define the \textit{weak nonlocal gradient} as
\begin{equation}
    \nabla_\omega u:= (\partial_{\omega,1} u,\ldots,\partial_{\omega,N} u) \in L^1_{\mathrm{loc}}(\Omega \times \Omega;\mathbb{R}^N).
\end{equation}

For all $p \in [1,+\infty]$, we then define the set
\begin{equation}
W^{1,p}_\omega(\Omega):= \{ u \in L^p(\Omega): \nabla_\omega u \in L^p(\Omega \times \Omega;\mathbb{R}^N)\},
\end{equation}
where we fix the norm in $L^p(\Omega\times\Omega;\mathbb{R}^N)$ to be
\begin{equation}
    \|v\|_{L^p(\Omega\times\Omega;\mathbb{R}^N)} := \left( \int_\Omega\int_\Omega |v|^p \mathrm{d}x \mathrm{d}y \right)^{\frac{1}{p}} \quad \mbox{ for all } v \in L^p(\Omega \times \Omega;\mathbb{R}^N),
\end{equation}
with $|\cdot|$ denoting the Euclidean norm in $\mathbb{R}^N$. One then shows that this space is a Banach space when endowed with the norm

\begin{equation}
    \|u\|_{W^{1,p}_\omega(\Omega)}:=\|u\|_{L^p(\Omega)}+\|\nabla_\omega u\|_{L^p(\Omega\times\Omega;\mathbb{R}^N)},
\end{equation}
see Proposition \ref{Proposition_Banach_Sob}.

The introduction and study of these nonlocal Sobolev spaces is one of the main developments of this article. To our knowledge, such function spaces have not been considered before in the literature. Indeed, the closest match (to our knowledge) to our weak formulation corresponds to the work recently developed by Comi and Stefani et~al.~\cite{Comi-Stefani19,ComSte19,BrCaCoSt20}, which establishes a distributional approach to fractional Sobolev and $\mathrm{BV}$ spaces. Also of interest is Šilhavý’s paper \cite{vsilhavy2020fractional}, where fractional operators are defined weakly in a similar manner to our setting. A comparison between our spaces and the fractional ones is given below the forthcoming Theorem \ref{Theorem_Trivial}. For more general literature on weighted nonlocal Sobolev spaces we point the reader to \cite{d2021towards, hepp2024divergence, pedregal2024special} and the works cited therein.

Similarly to the bounded variation case, a detailed study of $W^{1,p}_\omega(\Omega)$ is carried out in Subsection \ref{section_NL_Sobolev}, where we also extend some of the results in \cite{Kindermann} to the Sobolev setting. Moreover, we also study the relation between the classical Sobolev space $W^{1,p}$ and our nonlocal Sobolev spaces $W^{1,p}_\omega$ according to the behavior of $\omega$. To start off, let $W^{1,p}_1(\Omega)$ be the space associated to the weight which is equal to $1$ in $\Omega\times \Omega$. We prove the following:
\begin{theorem}[Relation between $W^{1,p}_1(\Omega)$ and $W^{1,p}(\Omega)$ when $\Omega$ has finite measure]\label{Theorem_equality_constant_weight}
Let $p\in[1,+\infty]$ and $\Omega$ be a nonempty open subset of $\mathbb{R}^N$ with finite measure if $p \in [1,+\infty)$ or general if $p=+\infty$. Then, we have $W^{1,p}_1(\Omega)=W^{1,p}(\Omega)$, with equivalence of norms.
\end{theorem}
The proof of Theorem \ref{Theorem_equality_constant_weight} is postponed until Subsection \ref{ConstWeightSubsection}, and relies on the properties of the nonlocal Sobolev space studied in Subsection \ref{section_NL_Sobolev}, namely that a function having weak ``local" derivatives is equivalent to it having weak nonlocal derivatives induced by $\omega \equiv 1$, see Propositions \ref{WeakDerImplyNlWeakDer} and \ref{NlWeakDerImplyWeakDer}. This fact, coupled with the integral/supremum Lemma \ref{integrallemma} will show that the spaces are equal as sets. Afterwards, to prove equivalence of norms, we will bind the local/nonlocal Sobolev seminorms by the appropriate norms using the formulas derived from the aforementioned propositions.

Furthermore, we also prove an analogous result to Theorem \ref{Theorem_main_equality_NLBV}, namely:
\begin{theorem}[Necessary and sufficient conditions for the equality between $W^{1,p}_\omega(\Omega)$ and $W^{1,p}(\Omega)$ when $\Omega$ has finite measure]\label{Theorem_equality_Sobolev}
    Let $p\in[1,+\infty]$, $\Omega$ be a nonempty open subset of $\mathbb{R}^N$ with finite measure if $p\in [1, +\infty)$ or general if $p=+\infty$, and $\omega$ an admissible weight function in $\Omega$. If there exists a constant $k>0$ such that for all $(x,y) \in \Omega\times\Omega$ one has

    \begin{equation} \label{omegaBoundedBelow}
        k \le \omega(x,y),
    \end{equation}
    then $W^{1,p}_\omega(\Omega) \hookrightarrow W^{1,p}(\Omega)$. Moreover, the embedding $W^{1,p}(\Omega) \hookrightarrow W^{1,p}_\omega(\Omega)$ holds if and only if there exists a constant $K>0$ such that for all $x\in \Omega$

    \begin{equation} \label{fwpBounded}
        f_\omega^p(x)\le K,
    \end{equation}
    where, for all $x\in \Omega$,

    \begin{equation}
        f_\omega^p(x) := \begin{cases}
            \int_\Omega \omega(x,y)^p \mathrm{d}y & \mbox{ if } p\in[1,+\infty),\\
            \esssup_{y\in\Omega} \omega(x,y) & \mbox{ if } p=+\infty.
        \end{cases}
    \end{equation}
\end{theorem}

Notice that condition \eqref{fwpboundedintro} of Theorem \ref{Theorem_main_equality_NLBV} coincides with \eqref{fwpBounded} with $p=1$. As in Theorem \ref{Theorem_main_equality_NLBV}, Theorem \ref{Theorem_equality_Sobolev} ``solves" one of the embeddings but only gives a sufficient condition for the other (again, see Remark \ref{NotNecRemark}). We point the reader towards Proposition \ref{NecCondCompletenessSob} for a necessary condition, and leave as an open question the matter of establishing a necessary and sufficient condition for the ``unsolved" embedding. 

Let us give some comments on the proof of Theorem \ref{Theorem_equality_Sobolev}, which will be carried out in Subsection \ref{GeneralWeightsSubsection}: Throughout the proof we will make heavy use of Theorem \ref{Theorem_equality_constant_weight} above. Proving that \eqref{omegaBoundedBelow} is a sufficient condition for the embedding $W^{1,p}_\omega(\Omega) \hookrightarrow W^{1,p}(\Omega)$ (for any $p\in[1,+\infty]$) follows immediately from the fact that the weak nonlocal gradient induced by $\omega$ of any function in $W^{1,p}(\Omega)$ is $\omega$ times the weak nonlocal gradient induced by $1$ of the function, see Remark \ref{1impliesomegaRemark}. Showing that condition \eqref{fwpBounded} implies $W^{1,p}(\Omega) \hookrightarrow W^{1,p}_\omega(\Omega)$ is also straightforward and follows from an integral/supremum inequality.

Instead, the bulk of the proof of Theorem \ref{Theorem_equality_Sobolev} will be dedicated to proving that \eqref{fwpBounded} is also a necessary condition for the aforementioned embedding. To do this in the $p\in[1,+\infty)$ case, for a given $x_0\in\Omega$ we will construct a family of test functions whose nonlocal Sobolev seminorm tends to $f_\omega^p(x_0)$, see Lemma \ref{P1TechLemmaNecCond}. This will only give us the embeddings ``seminorm-wise" (see Proposition \ref{NecCondProp}), and additionally we will have to check that the $L^p$ norms of the test functions tend to zero. To prove the $p=+\infty$ case, a different argument involving the same test functions will be used. Since this procedure essentially relies on test functions, the argument can be adapted to prove the corresponding part of Theorem \ref{Theorem_main_equality_NLBV} with $p=1$ using the fact that $\mathrm{NLTV}_\omega$ coincides with the nonlocal Sobolev seminorm for weakly derivable functions, see Proposition \ref{C1NLTV}.

On another note, it may come as a surprise that, despite our weight functions being $\mathcal{C}^1$, the space of smooth functions with compact support may not be included in the nonlocal Sobolev/$\mathrm{BV}$ space for some choices of $\omega$. What is more, certain weight functions induce nonlocal spaces for which the null function is the only $\mathcal{C}_c^\infty$ function contained in them. The weight functions which ensure that this is not the case are characterized as follows:
\begin{theorem}[Relation with the space of test functions]\label{Theorem_TestFunctions}
Let $\Omega$ be a nonempty open subset of $\mathbb{R}^N$, $p \in [1,+\infty)$ and $\omega$ be an admissible weight function in $\Omega$. Then, we have that $W^{1,p}_\omega(\Omega) \cap \mathcal{C}^\infty_c(\Omega) \not = \{0\}$ if and only if there exists a compact set $K \subset \Omega$ with nonempty interior such that
\begin{equation}\label{locally_bounded_omega}
\int_K\int_\Omega\omega(x,y)^p\mathrm{d}y\mathrm{d}x<+\infty.
\end{equation}

Moreover, $W^{1,\infty}_\omega(\Omega) \cap \mathcal{C}^\infty_c(\Omega) \not = \{0\}$ if and only if there exists a compact set $K \subset \Omega$ with nonempty interior such that
\begin{equation}\label{locally_bounded_omega_infty}
\sup_{(x,y) \in K\times \Omega} \omega(x,y) <+\infty.
\end{equation}

Similarly, for $p\in[1,+\infty)$, $\mathcal{C}^\infty_c(\Omega) \subset W^{1,p}_\omega(\Omega)$ is true if and only if \eqref{locally_bounded_omega} holds for any compact set $K \subset \Omega$, and $\mathcal{C}^\infty_c(\Omega) \subset W^{1,\infty}_\omega(\Omega)$ is true if and only if \eqref{locally_bounded_omega_infty} holds for any compact set $K \subset \Omega$. If $p=1$, the same statements hold when $W^{1,1}_\omega(\Omega)$ is replaced by $\mathrm{NLBV}_\omega(\Omega)$.
\end{theorem}

The proof of Theorem \ref{Theorem_TestFunctions} will be given in Section \ref{RelationTestFunctionsSubsection}. The argument is based on the fact that the nonlocal Sobolev seminorm induced by $\omega$ of a (locally) weakly derivable function $u$ splits into two integrals (or essential supremums if $p=+\infty$) depending on $\nabla u$ and $\omega$; one dealing with the interactions within $\esssupp(u)\times \esssupp(u)$, and the other dealing with the interactions in $\esssupp(u)\times \left( \Omega \setminus \esssupp(u) \right)$, see Remark \ref{InteractionsRemark}. We will see that for test functions the first integral/supremum is always finite, and this will imply, after some arguments, that any test function belongs to the nonlocal Sobolev space if and only if the second integral/supremum is finite, see Lemma \ref{TestFuncLemma}. After seeing that this second term is finite if and only if \eqref{locally_bounded_omega}/\eqref{locally_bounded_omega_infty} holds, the proof will follow. Once the Sobolev case is proven, the $\mathrm{NLBV}_\omega$ statements will follow using Proposition \ref{C1NLTV}.

Notice that, in contrast to Theorems \ref{Theorem_main_equality_NLBV}, \ref{Theorem_equality_constant_weight} and \ref{Theorem_equality_Sobolev}, Theorem \ref{Theorem_TestFunctions} above does not require that $\Omega$ have finite measure, even for $p\neq+\infty$. It is here that we point out why we ask that the domain (i.e. $\Omega$) be of finite measure instead of being bounded, which is a more standard (yet restrictive) condition. The reason is that, for $p\in[1,+\infty)$, in the aforementioned results we are essentially only working with integrals, and to bound some of these integrals it is sufficient to ask that $\Omega$ has finite measure. This being said, on some results we do ask that $\Omega$ be bounded, and this is due to the fact that their proofs pass through classical Sobolev and $\mathrm{BV}$ results which require it.

With this in mind, let us investigate the case in which $\Omega$ has infinite measure. It turns out that, under such an assumption, the landscape changes dramatically, and for some choices of weights one is lead to (possibly) trivial spaces. Indeed, as an immediate corollary of Theorem \ref{Theorem_TestFunctions} we have:

\begin{cor} \label{cor_trivial}
    Assume that $\Omega$ is a nonempty open subset of $\mathbb{R}^N$ that has infinite measure and let $p \in [1,+\infty)$ and $\omega$ be an admissible weight function in $\Omega$. If there exists a constant $c>0$ such that for all $(x,y) \in \Omega \times \Omega$ one has
    \begin{equation*}
    c \le \omega(x,y),
    \end{equation*}
    then $\mathcal{C}^\infty_c(\Omega) \cap W^{1,p}_\omega(\Omega)=\mathcal{C}^\infty_c(\Omega) \cap \mathrm{NLBV}_\omega(\Omega)=\{0\}$.
\end{cor}

The proof of the corollary above readily follows from noticing that, if $\omega$ is bounded away from zero and $\Omega$ is of infinite measure, the quantity \eqref{locally_bounded_omega} is infinite regardless of $K$ (assuming it has nonempty interior). 

It is natural to conjecture that in the case in which $W^{1,p}_\omega(\Omega) \cap \mathcal{C}^\infty_c(\Omega)=\{0\}$ one actually has $W^{1,p}_\omega(\Omega)=\{0\}$, and analogously for $\mathrm{NLBV}_\omega(\Omega)$. This would easily follow from Corollary \ref{cor_trivial} combined with density and approximation results for $W^{1,p}_\omega(\Omega)$ or $\mathrm{NLBV}_\omega(\Omega
)$ with respect to compactly supported smooth functions. We have not been able to prove such results, essentially due to the fact that we do not have an appropriate product rule for the nonlocal operators $\mathrm{div}_\omega$ and $\nabla_\omega$ (for more comments on approximation results in our setting, we direct the reader to the start of Subsection \ref{section_NLBV}). Nevertheless, we have been able to give a positive answer to the previous conjecture for the Sobolev statement in all cases, and for the $\mathrm{NLBV}_\omega$ statement provided $\Omega$ satisfies a certain condition:
\begin{theorem}[The case in which $\Omega$ has infinite measure and $\omega$ is bounded away from zero]\label{Theorem_Trivial}
Assume that $\Omega$ is a nonempty open subset of $\mathbb{R}^N$ that has infinite measure and let $p \in [1,+\infty)$ and $\omega$ be an admissible weight function in $\Omega$. If there exists a constant $c>0$ such that for all $(x,y) \in \Omega \times \Omega$ one has
\begin{equation*}
c \le \omega(x,y),
\end{equation*}
then we have that $W^{1,p}_\omega(\Omega) = \{0\}$. Moreover, if we further assume that there exists $(\Omega_{n})_{n \in \mathbb{N}}$, a sequence of open subsets of $\Omega$ such that for each $n \in \mathbb{N}$ the set $\Omega_n$ is of infinite measure and $\mathrm{dist}(\Omega_n,\partial \Omega)$ is positive, then $\mathrm{NLBV}_\omega(\Omega)=\{0\}$.
\end{theorem}
The proof of Theorem \ref{Theorem_Trivial} will be given in Section \ref{Section_trivial}. A straightforward consequence of Theorem \ref{Theorem_Trivial} is that $W^{1,p}_\omega(\mathbb{R}^N)=\mathrm{NLBV}_\omega(\mathbb{R}^N)=\{0\}$ whenever $\omega$ is bounded below by a positive constant on $\mathbb{R}^N \times \mathbb{R}^N$. The proof of Theorem \ref{Theorem_Trivial} is based on the fact that, for $p\neq+\infty$, the nonlocal Sobolev seminorm induced by a bounded away from zero weight function of a non-affine function is always infinite whenever $\Omega$ has infinite measure, see Lemma \ref{NonConstantInfSeminorm}. With this result, the proof of Theorem \ref{Theorem_Trivial} is obtained directly in the Sobolev case, and is based on a convolution argument in the bounded variation case. In particular, the proof of Theorem \ref{Theorem_Trivial} does not rely on Theorem \ref{Theorem_TestFunctions} or Corollary \ref{cor_trivial}.

At this point, one might wonder about the link between the function spaces considered in this paper and the \textit{fractional} spaces. The latter type of spaces has been widely studied in Sobolev-type frameworks; see Di Nezza, Palatucci and Valdinoci \cite{DiNezza-Palatucci-Valdinoci} and the references therein. More recently, bounded variation counterparts have been also studied, we refer to Comi and Stefani \cite{Comi-Stefani19} and Zust \cite{Zust19}, see also Antil, Diaz, Jing and Schikorra \cite{antil2024}, as well as Bessas et~al. \cite{bessas2022fractional,bessas2025non} for applications to image processing. In fact, as already emphasized, our spaces are rather different from the fractional ones. This is essentially due to the fact that, although fractional spaces can also be seen as weighted spaces, such weights are singular on the diagonal of $\Omega \times \Omega$, in contrast with those we consider in this article which do not have singularities at the interior of $\Omega \times \Omega$. On a related note, let us mention that we are aware that coordinate-based definitions of fractional operators are known to be flawed, in the sense that they depend on the choice of orthogonal basis (see \cite{vsilhavy2020fractional}). However, one can readily prove that this is not the case for the nonlocal operators introduced in this paper.

To motivate our next results, let us give a brief overview of the image processing literature that inspired \cite{Kindermann,Wang-Ng}, of which our spaces are based on: The seminal paper \cite{Rudin-Osher-Fatemi} by  Rudin, Osher and Fatemi (see also Chambolle and P.-L. Lions \cite{Chambolle-Lions}), established \emph{total variation} as a strong regularization term in image denoising which greatly improved the state of the art at the time, and sparked a significant rise of other total variation-based models. During the last decades, several alternatives to the total variation have been proposed as regularization terms, see, e.g., \cite{Bergounioux, Bredies, Jalalzai, Chan, Hu, Pang,  Baifractional}. In the highly influential work \cite{Buades-Coll-Morel,Buades-Coll-Morel-bis}, Buades, Coll and Morel introduced a nonlocal filter for image denoising, generalizing previous approaches by Yaroslavsky and Eden \cite{Yaroslavsky,YaroslavskyEden}. Several variational interpretations of such a filter were given later in \cite{Kindermann} as well as by Gilboa and Osher \cite{Gilboa-Osher,Gilboa-Osher-bis}. In particular, in the latter reference the total variation term in \cite{Rudin-Osher-Fatemi} is replaced by a \textit{nonlocal} term of the type:
\begin{equation}\label{nl_mean}
\int_{\Omega \times \Omega} \omega(x,y)\lvert u(y)-u(x) \rvert \mathrm{d}x\mathrm{d}y,
\end{equation}
 where $\omega$ is positive and symmetric. In general, nonlocal terms average every pixel with each of the other pixels in the image, rather than simply its neighbors. This is advantageous because, as argued in \cite{Kindermann,Buades-Coll-Morel,Buades-Coll-Morel-bis}, natural images are self-similar, in the sense that every small patch has many similar patches in the same image. It is then clear that nonlocal regularization terms promote smoothness among similar pixels and thus better preserve the image's characteristics.

The approach in \cite{Kindermann} consisted on considering a higher order nonlocal term similar to \eqref{nl_mean}, namely
\begin{equation}\label{nl-Kindermann}
\int_{\Omega \times \Omega}\lvert \nabla u(y)-\nabla u(x) \rvert \mathrm{d}x\mathrm{d}y,
\end{equation}
which coincides with $\mathrm{NLTV}_1(u)$ for $u \in \mathcal{C}^1(\Omega)$. This regularization term was claimed to not suffer from the so-called \textit{staircase effect} and was showed through numerical experiments to be more advantageous than \eqref{nl_mean} when applied to image deblurring problems. A more general variant was proposed in \cite{Wang-Ng} with the scope of addressing problems in image decomposition. More precisely, they considered the following weighted version of \eqref{nl-Kindermann}:
\begin{equation}\label{nl_weighted}
\int_{\Omega \times \Omega} \omega(x,y)\lvert \nabla u(y)-\nabla u(x) \rvert \mathrm{d}x\mathrm{d}y,
\end{equation}
which, again, corresponds to $\mathrm{NLTV}_\omega(u)$ as long as $u\in \mathcal{C}^1(\Omega)$.

In this paper we show the existence and uniqueness to the solution of two minimizing problems which involve the $\mathrm{NLBV}_\omega$ and $W^{1,p}_\omega$ spaces and their respective seminorms. In contrast to \cite{Kindermann,Wang-Ng}, our \emph{fidelity terms} will be general and not fixed. Additionally, our results do not rely on any sort of compact Sobolev-type embeddings, as $F$ is assumed to be convex. See Remark \ref{remark_existence_NLBV} for additional comments.

First, for a fixed $1<q<+\infty$, we consider the functional
\begin{equation}\label{functional_NLBV_def}
I_\omega(u):= \mathrm{NLTV}_\omega(u)+F(u),
\end{equation}
defined for $u\in\mathrm{NLBV}_\omega(\Omega) \cap L^q(\Omega)$, where $F: L^q(\Omega) \to [-\infty,+\infty]$ is a fidelity term which we assume to be proper, bounded from below, convex, lower semicontinuous (with respect to the $L^q$ norm) and to satisfy the coercivity condition
\begin{equation}\label{coercivity_F_NLBV}
F(u) \to +\infty \quad \mbox{ as } \quad \|u\|_{L^q(\Omega)} \to +\infty.
\end{equation}

We prove the following:
\begin{theorem}[Existence and uniqueness of minimizers for $I_\omega$]\label{Theorem_main_functional}
Let $\Omega$ be a nonempty open subset of $\mathbb{R}^N$ of finite measure, $\omega$ be an admissible weight function in $\Omega$, $1<q<+\infty$, $F: L^q(\Omega) \to [-\infty,+\infty]$ be proper, bounded from below, convex, lower semicontinuous and satisfying \eqref{coercivity_F_NLBV}, and $I_\omega$ be as in \eqref{functional_NLBV_def}. Then, there exists $u_* \in \mathrm{NLBV}_\omega(u) \cap L^q(\Omega)$ such that
\begin{equation}
I_\omega(u_*)= m_*:= \inf_{u \in \mathrm{NLBV}_\omega(\Omega)\cap L^q(\Omega)} I_\omega(u).
\end{equation}

Moreover, if $F$ is strictly convex then the minimizer is unique.
\end{theorem}
The proof of Theorem \ref{Theorem_main_functional} follows by an application of the Direct Method and will be given in Section \ref{section_applications}. Such a procedure is only possible when one has established the necessary properties for the underlying space $\mathrm{NLBV}_\omega(\Omega)$. The $q=1$ case is problematic, see Remark \ref{remark_existence_NLBV}. In the $\mathrm{BV}$ case, one might relax the convexity assumption on $F$ and use the compact embeddings into $L^q$ spaces. This is still possible to some extent for $\mathrm{NLBV}_\omega(\Omega)$, again, see Remark \ref{remark_existence_NLBV}. Additional assumptions on $\Omega$ and $\omega$ are needed for such a result, as otherwise the compact embeddings are lost. Indeed, one might think about the limit case in which $\omega=0$ in $\Omega \times \Omega$, so that $\mathrm{NLBV}_\omega(\Omega)=L^1(\Omega)$.

The second functional we consider is
\begin{equation}\label{functional_nonlocal_Sobolev}
I_\omega^p(u):= \int_{\Omega \times \Omega} \lvert \nabla_\omega u\rvert^p+F(u),
\end{equation}
defined for $u \in W^{1,p}_\omega(\Omega)\cap L^q(\Omega)$, where $p\in(1,+\infty)$, $p\le q <+\infty$ and $F$ satisfies the properties described above Theorem \ref{Theorem_main_functional}.

We prove the following:
\begin{theorem}[Existence and uniqueness of minimizers for $I_\omega^p$]\label{Theorem_main_existence_Sobolev}
Let $\Omega$ be a nonempty open subset of $\mathbb{R}^N$, $p \in (1,+\infty)$, $\omega$ be an admissible weight function in $\Omega$, $p\le q <+\infty$, $F: L^q(\Omega) \to [-\infty,+\infty]$ be proper, bounded from below, convex, lower semicontinuous and satisfying \eqref{coercivity_F_NLBV}, and $I_\omega^p$ be as in \eqref{functional_nonlocal_Sobolev}. If $q > p$, assume furthermore that $\Omega$ is of finite measure. Then, there exists $u_{*,p} \in W^{1,p}_\omega(\Omega) \cap L^q(\Omega)$ such that
\begin{equation}
I_\omega^p(u_{*,p})=m_{*,p}:= \inf_{u \in W^{1,p}_\omega(\Omega)\cap L^q(\Omega)} I_\omega^p(u).
\end{equation}

Moreover, if $F$ is strictly convex then the minimizer is unique.
\end{theorem}
The proof of Theorem \ref{Theorem_main_existence_Sobolev} is given in Section \ref{section_applications}. It too follows from an application of the Direct Method, and it relies on a previous study of the properties of $W^{1,p}_\omega(\Omega)$. The $1<q<p$ case can be considered if one modifies the underlying space, and the $p=1$ case can be treated as in the bounded variation setting, see Remark \ref{Remark_qlessthanp}. Moreover, the $p=+\infty$ case can also be considered but requires some restrictive conditions on $F$, see Remark \ref{remark_p_infinity}.

As for future work, we hypothesize that the weighted $\mathrm{NLBV}$ and nonlocal Sobolev spaces can be advantageous in an image restoration context, and we look forward to derive and implement variational models  based on the theory developed in this article. In this direction, Theorems \ref{Theorem_main_functional} and \ref{Theorem_main_existence_Sobolev} serve as a preliminary theoretical basis. Moreover, the nonlocal Sobolev space can potentially lead to some new PDE problems, some of which we hope can model real life phenomena governed by nonlocal interactions.

\textbf{Acknowledgements.} F.A., J.D. and C.S. are supported by the MaLiSat project TED2021-132644B-I00 funded by MICIU/AEI
/10.13039/501100011033 and the European Union NextGenerationEU/PRTR, and also by the MoMaLIP project PID2021-125711OB-I00 funded by MICIU/AEI/10.13039/501100011033 and the European Union NextGeneration EU/PRTR. R.O.-B. is supported by the Program Horizon Europe Marie Sklodowska-Curie Post-Doctoral Fellowship (HORIZON-MSCA-2023-PF-01). Grant agreement: 101149877. Project acronym: NFROGS.

\tableofcontents

\section{Nonlocal Sobolev and BV spaces}

In this section we propose a systematic and coherent framework for two function spaces induced by nonfractional weights, one of them being brand new, to our knowledge. The definitions appearing in the introduction will be restated and expanded upon, with further observations and remarks. Additionally, we will derive some preliminary results. We point out the notable ones: Relation between ``local" and weak nonlocal derivatives, which will be key when proving Theorem \ref{Theorem_equality_constant_weight} (Propositions \ref{WeakDerImplyNlWeakDer} and \ref{NlWeakDerImplyWeakDer}), completeness of the spaces (Propositions \ref{Proposition_Banach_Sob} and \ref{Proposition_Banach_NLBV}), reflexiveness and separability of the nonlocal Sobolev space for certain exponents (Proposition \ref{Proposition_reflexives}), and some initial trace theory results (Propositions \ref{Proposition_trace_Sob} and \ref{Proposition_trace_NLBV}).

Throughout this whole manuscript $\Omega$ will denote a nonempty open subset of $\mathbb{R}^N$, which we will often refer to as \emph{domain}\footnote{To make it clear, no prior assumptions on the boundedness or the connectedness of $\Omega$ are imposed.}. We restate the conditions we impose on our weight functions in a more formal manner in the definition below.

\begin{defi}[Admissible weight function]
    We say $\omega:\Omega\times\Omega\to [0,+\infty)$ is an \emph{admissible weight function in $\Omega$} if it belongs to $\mathcal{C}^1(\Omega\times\Omega)$ and it is symmetric, i.e.,
    $$\omega(x,y)=\omega(y,x) \quad \mbox{ for all }(x,y)\in \Omega\times\Omega.$$
\end{defi}

When the domain is clear from context, the ``in $\Omega$" part will often be omitted. As stated in the introduction, our weight functions are nonsingular at the interior, in contrast to the weights associated to the fractional calculus theory developed in the past few years. This being said, our definition still allows for the weight functions to tend to infinity at the boundary of $\Omega\times\Omega$ and to vanish anywhere in $\Omega\times\Omega$ and its boundary. In particular $\omega \equiv 0$ in $\Omega\times\Omega$ is allowed, which will act as a simple yet useful counterexample to certain statements.

\subsection{Nonlocal weighted Sobolev spaces}\label{section_NL_Sobolev}

We start by defining ``strong" nonlocal first order derivatives, which will be necessary to derive the integration by parts formula which weak nonlocal derivatives will need to satisfy.

\begin{defi}[Nonlocal partial derivatives]
    Let $\omega$ be an admissible weight function. For $i=1,\ldots,N$ we define the \emph{$i$-th nonlocal partial derivative induced by} $\omega$ of $\psi \in \mathcal{C}^1_0(\Omega \times\Omega)$ as

    $$\partial_{\omega, i}^\ast (\psi) (x) := \int_\Omega \frac{\partial}{\partial x_i} \left[ \omega(x,y) (\psi(y,x)-\psi(x,y)) \right] \mathrm{d}y.$$
\end{defi}

With the ``strong" nonlocal first order derivatives, we define weak nonlocal first order derivatives.

\begin{defi}[Weak nonlocal partial derivatives and nonlocal gradient] \label{NlWeakDerDef}
    Let $\omega$ be an admissible weight function. We say $u \in L^1_{\mathrm{loc}}(\Omega)$ has \emph{weak $i$-th nonlocal partial derivative induced by $\omega$} if there exists ${\partial_{\omega,i}u\in L^1_{\mathrm{loc}}(\Omega\times\Omega)}$ such that

    \begin{equation} \label{WeakDerEq}
        \int_\Omega u(x) \partial_{\omega, i}^\ast (\psi) (x) \mathrm{d}x = - \int_{\Omega\times\Omega} (\partial_{\omega, i} u)(x,y) \psi(x,y) \mathrm{d}x \mathrm{d}y
    \end{equation}
    for all $\psi \in \mathcal{C}_c^1(\Omega \times\Omega)$. In such a case, we say $\partial_{\omega, i} u$ is the \emph{weak $i$-th nonlocal partial derivative} of $u$.

    When $u$ has weak $i$-th nonlocal partial derivative induced by $\omega$ for any $i=1,\ldots,N$, we say

    $$\nabla_\omega u (x,y) := \left( \partial_{\omega, 1} u(x,y),\ldots,\partial_{\omega, N} u(x,y) \right)$$
    is the \emph{weak nonlocal gradient induced by $\omega$} of $u$. 
\end{defi}

When there is no doubt as to which weight function we are referring to, the ``induced by $\omega$" parts will often be omitted. Weak nonlocal derivatives are uniquely determined a.e., which justifies our notation. Some other notable properties which follow from their definition are that they vanish a.e.~in the region where the weight function itself vanishes and that they are anti-symmetric, i.e., are such that
\begin{equation}
    \partial_{\omega,i} u(x,y) = -\partial_{\omega,i} u(y,x) \quad \mbox{ for almost every } (x,y)\in\Omega\times\Omega.
\end{equation}

Note that the $i$-th nonlocal partial derivative operator $\partial_{\omega, i}^\ast$ acts on $\mathcal{C}^1(\Omega\times\Omega)$ functions and gives functions from $\Omega$ to $\mathbb{R}$, while the weak $i$-th nonlocal partial derivative operator $\partial_{\omega, i}$ acts on $L^1_{\mathrm{loc}}(\Omega)$ functions and gives functions from $\Omega\times\Omega$ to $\mathbb{R}$. In fact, we may think of $\partial_{\omega, i}^\ast$ as minus the adjoint operator of $\partial_{\omega, i}$ (and vice versa), which justifies our notation.

\begin{remark} \label{NlWeakDerDefExtension}
    Similarly to the local case, given $u \in L^1_{\mathrm{loc}}(\Omega)$ and $i = 1,...,N$, $u$ has weak $i$-th nonlocal partial derivative if and only if there exists $v\in L^1_{\mathrm{loc}}(\Omega\times\Omega)$ such that

    \begin{equation}
        \int_\Omega u(x) \partial_{\omega, i}^\ast (\hat\psi) (x) \mathrm{d}x = - \int_{\Omega\times\Omega} v(x,y) \hat\psi(x,y) \mathrm{d}x \mathrm{d}y
    \end{equation}
    for all $\hat \psi \in \mathcal{C}_0^1(\Omega \times\Omega)$. In such a case, $v=\partial_{\omega,i}u$ a.e.~in $\Omega\times\Omega$.
\end{remark}

\begin{remark} \label{1impliesomegaRemark}
    Let $\omega$ be an admissible weight function and $i=1,...,N$. Then, if $u \in L^1(\Omega)$ has weak $i$-th nonlocal partial derivative induced by the constant weight function $1$, it also has weak $i$-th nonlocal partial derivative induced by $\omega$, and

    $$\partial_{\omega,i}u(x,y)=\omega(x,y) \partial_{1,i}u(x,y) \quad \mbox{ for almost every } (x,y) \in \Omega\times \Omega.$$
\end{remark}

As for the converse of the remark above, we have the following:

\begin{remark} \label{omegaimplies1Remark}
    Let $\omega$ be an admissible weight function in $\Omega$, $i=1,...,N$ and $A \subset \Omega$ be a nonempty open set such that

    $$A\times A \subset \omega^{-1}\left((0,+\infty)\right).$$
    
    If $u \in L^1(\Omega)$ has weak $i$-th nonlocal partial derivative induced by $\omega$ (in $\Omega$), then $u$ has $i$-th nonlocal derivative induced by $1$ in $A$ given by

    $$\partial_{1,i}^A u (x,y) = \frac{\partial_{\omega,i} u(x,y)}{\omega(x,y)} \; \mbox{ for almost every } (x,y) \in A\times A.$$

    In particular, whenever $\omega$ does not vanish in $\Omega\times\Omega$, if $u \in L^1(\Omega)$ has weak $i$-th nonlocal partial derivative induced by $\omega$, then $u$ has weak $i$-th nonlocal partial derivative induced by $1$ and $\partial_{1,i}u=\partial_{\omega,i} u/ \omega$ a.e. in $\Omega\times\Omega$.
\end{remark}

A straightforward computation shows that if a function has standard weak derivatives it also has weak nonlocal derivatives induced by any weight function, as shown next.

\begin{prop}[Weak derivatives imply weak nonlocal derivatives]
\label{WeakDerImplyNlWeakDer}

    If $\omega$ is an admissible weight function and $u \in L^1_{\mathrm{loc}}(\Omega)$ has weak $i$-th (standard) partial derivative for some $i=1,\ldots,N$, then $u$ has weak $i$-th nonlocal partial derivative induced by $\omega$ and

    $$(\partial_{\omega,i} u)(x,y) = \omega(x,y) \left(\partial_i u (y) - \partial_i u (x) \right) \quad \mbox{ for almost every } (x,y)\in\Omega\times\Omega.$$
\end{prop}

\begin{proof}
    Given $\psi \in \mathcal{C}_c^1(\Omega\times\Omega)$, by Leibniz rule we have

    \begin{align} \label{auxEx1}
        \int_\Omega u(x) \partial_{\omega, i}^\ast (\psi) (x) \mathrm{d}x &= \int_\Omega u(x) \int_\Omega \frac{\partial}{\partial x_i} \left[ \omega(x,y) (\psi(y,x)-\psi(x,y)) \right] \mathrm{d}y \mathrm{d}x\nonumber\\
        &=\int_\Omega u(x) \frac{\partial}{\partial x_i} \left[\int_\Omega \omega(x,y) \psi(y,x) \mathrm{d}y \right] \mathrm{d}x\nonumber\\
        &\qquad \qquad - \int_\Omega u(x) \frac{\partial}{\partial x_i} \left[ \int_\Omega \omega(x,y) \psi(x,y) \mathrm{d}y \right] \mathrm{d}x.
    \end{align}

    It is immediate to check that the functions of $x\in\Omega$ given by

    \begin{equation}
        \int_\Omega \omega(x,y) \psi(y,x) \mathrm{d}y, \qquad \int_\Omega \omega(x,y) \psi(x,y) \mathrm{d}y
    \end{equation}
    belong to $\mathcal{C}_c^1(\Omega)$. Thus, continuing the chain of equalities and by definition of standard weak derivatives we have

    \begin{align}
        \int_\Omega u(x) \partial_{\omega, i}^\ast (\psi) (x) \mathrm{d}x &= - \int_\Omega \partial_i u(x) \int_\Omega \omega(x,y) \psi(y,x) \mathrm{d}y \mathrm{d}x\\
        &\qquad \qquad + \int_\Omega \partial_i u(x) \int_\Omega \omega(x,y) \psi(x,y) \mathrm{d}y \mathrm{d}x\\
        &= - \int_\Omega \partial_i u(y) \int_\Omega \omega(y,x) \psi(x,y) \mathrm{d}x \mathrm{d}y\\
        &\qquad \qquad + \int_\Omega \partial_i u(x) \int_\Omega \omega(x,y) \psi(x,y) \mathrm{d}y \mathrm{d}x\\
        &= - \int_{\Omega\times\Omega} \left[\omega(x,y) \left( \partial_i u(y) - \partial_i u(x) \right) \right] \psi(x,y) \mathrm{d}x \mathrm{d}y,
    \end{align}
    where in the second equality we have interchanged the variables of integration (only notation-wise) and in the last we have used Fubini's theorem and the fact that $\omega$ is symmetric. From this computation the result follows.
\end{proof}

We now show that the converse implication holds for constant weight functions. Remarkably, each $\theta \in \mathcal{C}_c^1(\Omega)$ with non-vanishing integral over $\Omega$ yields a valid formula for the standard weak derivative of a function in terms of its weak nonlocal derivative. Note that in the proof below Fubini's theorem will be used extensively without mention.

\begin{prop}
    [Weak nonlocal derivatives imply weak derivatives for constant weight functions] \label{NlWeakDerImplyWeakDer}
    If $u \in L^1_{\mathrm{loc}}(\Omega)$ has weak $i$-th nonlocal partial derivative induced by $1$ for some $i=1,\ldots,N$, then $u$ has weak $i$-th (standard) partial derivative and

    \begin{equation} \label{WeakDerInTermsOfNlWeakDer}
        \partial_i u(x) = -\frac{1}{\int_\Omega \theta(z) \mathrm{d}z} \int_\Omega \left[ u(y) \partial_i \theta(y) + \partial_{1,i} u(x,y) \theta(y) \right]\mathrm{d}y \quad \mbox{ for almost every } x \in \Omega
    \end{equation}
    for any $\theta \in \mathcal{C}_c^1(\Omega)$ such that

    \begin{equation} \label{ThetaIntCond}
        \int_\Omega \theta (z)\mathrm{d}z \neq 0.
    \end{equation}
\end{prop}

\begin{proof}
    Fix $\theta \in \mathcal{C}_c^1(\Omega)$ satisfying \eqref{ThetaIntCond}. For any $\varphi \in \mathcal{C}_c^1(\Omega)$ consider

    $$\phi_{\theta,\varphi}(x,y):=\theta(y) \varphi(x),$$
    which is a function of $\mathcal{C}_c^1(\Omega\times\Omega)$. Then

    \begin{align}
        \int_\Omega u(x) \partial_{1,i}^\ast (\phi_{\theta,\varphi})(x) \mathrm{d}x &= \int_\Omega u(x) \int_\Omega \frac{\partial}{\partial x_i} \left[ \theta(x) \varphi(y) - \theta(y) \varphi(x) \right] \mathrm{d}y \mathrm{d}x\\
        &= \int_\Omega u(x) \int_\Omega \left[ \partial_i \theta(x) \varphi(y) - \theta(y) \partial_i \varphi(x) \right] \mathrm{d}y \mathrm{d}x\\
        &= \int_\Omega u(x) \partial_i \theta(x) \left( \int_\Omega\varphi(y) \mathrm{d}y \right) \mathrm{d}x\\
        &\qquad \qquad - \left( \int_\Omega \theta(y) \mathrm{d}y \right) \int_\Omega u(x) \partial_i\varphi(x) \mathrm{d}x.
    \end{align}

    If we isolate the last integral of the last equality we get

    \begin{align}
        \int_\Omega u(x) \partial_i \varphi(x) \mathrm{d}x &= -\frac{1}{\int_\Omega \theta(y) \mathrm{d}y} \left[\int_\Omega u(x) \partial_{1,i}^\ast (\phi_{\theta,\varphi})(x) \mathrm{d}x\right. \left.- \int_\Omega u(x) \partial_i \theta(x) \left( \int_\Omega\varphi(y) \mathrm{d}y \right) \mathrm{d}x \right],
    \end{align}
    and by relabeling the variable of the $\theta$ integral and by definition of weak nonlocal derivatives we have

    \begin{align}
        \int_\Omega u(x) \partial_i \varphi(x) \mathrm{d}x &= -\frac{1}{\int_\Omega \theta(z) \mathrm{d}z} \left[ -\int_{\Omega\times\Omega} \partial_{1,i} u(x,y)\phi_{\theta,\varphi}(x,y) \mathrm{d}x \mathrm{d}y \right. \left.- \int_\Omega \left(\int_\Omega u(x) \partial_i \theta(x) \mathrm{d}x \right) \varphi(y) \mathrm{d}y \right]\\
        &= \frac{1}{\int_\Omega \theta(z) \mathrm{d}z} \left[\int_{\Omega} \left( \int_{\Omega} \partial_{1,i} u(x,y)\varphi(x) \mathrm{d}x \right) \theta(y) \mathrm{d}y \right. \left.+ \int_\Omega \left(\int_\Omega u(x) \partial_i \theta(x) \mathrm{d}x \right) \varphi(y) \mathrm{d}y \right]\\
        &= \frac{1}{\int_\Omega \theta(z) \mathrm{d}z} \left[\int_{\Omega} \left( \int_{\Omega} \partial_{1,i} u(x,y)\varphi(x) \mathrm{d}x \right) \theta(y) \mathrm{d}y \right. \left.+ \int_\Omega \left(\int_\Omega u(y) \partial_i \theta(y) \mathrm{d}y \right) \varphi(x) \mathrm{d}x \right]\\
        &= \int_\Omega \left[ \frac{1}{\int_\Omega \theta(z) \mathrm{d}z} \int_\Omega \left[u(y) \partial_i \theta(y) + \partial_{1,i} u(x,y) \theta(y)\right] \mathrm{d}y \right] \varphi(x) \mathrm{d}x\\
        &= -\int_\Omega \left[ \frac{-1}{\int_\Omega \theta(z) \mathrm{d}z} \int_\Omega \left[ u(y) \partial_i \theta(y) + \partial_{1,i} u(x,y) \theta(y) \right] \mathrm{d}y \right] \varphi(x) \mathrm{d}x,
    \end{align}
    where in the third to last equality we have relabeled the integration variables in the second term.

    Thus, $u$ has weak $i$-th (standard) partial derivative equal (a.e.) to

    $$\partial_i u(x) = -\frac{1}{\int_\Omega \theta(z) \mathrm{d}z} \int_\Omega \left[u(y) \partial_i \theta(y) + \partial_{1,i} u(x,y) \theta(y)\right] \mathrm{d}y.$$
\end{proof}

While expression \eqref{WeakDerInTermsOfNlWeakDer} seemingly depends on the $\theta \in \mathcal{C}_c^1(\Omega)$ chosen, it is the same a.e.~for any $\theta$ satisfying \eqref{ThetaIntCond}, as weak (standard) derivatives are unique (again, in the a.e.~sense).

\begin{remark}
    Suppose $u\in L^1(\Omega)$ has weak $i$-th nonlocal partial derivative for all $i=1,\ldots,N$, then by   \eqref{WeakDerInTermsOfNlWeakDer} we have that the weak (standard) gradient can be written as

    \begin{equation} \label{WeakGradInTermsOfNlWeakGrad}
        \nabla u(x) =  -\frac{1}{\int_\Omega \theta(z) \mathrm{d}z} \int_\Omega \left[u(y) \nabla \theta(y) + \nabla_1 u(x,y) \theta(y)\right] \mathrm{d}y,
    \end{equation}
    where $\theta \in\mathcal{C}_c^1(\Omega)$ satisfies \eqref{ThetaIntCond}.
\end{remark}

\begin{remark} \label{WeakDerInTermsOfNlWeakDerExtension}
    In view of Remark \ref{NlWeakDerDefExtension}, expression \eqref{WeakDerInTermsOfNlWeakDer} of Proposition \ref{NlWeakDerImplyWeakDer} also holds for $\theta \in \mathcal{C}_0^1(\Omega)$ satisfying \eqref{ThetaIntCond}. In particular, formula \eqref{WeakGradInTermsOfNlWeakGrad} of the remark above is also true for any $\theta \in \mathcal{C}_0^1(\Omega)$ satisfying \eqref{ThetaIntCond}.
\end{remark}

Having introduced a notion of weak derivatives, we define the nonlocal Sobolev space as follows:

\begin{defi}[Nonlocal Sobolev space] \label{NlSobDef}
    Let $\omega$ be an admissible weight function and $p\in[1,+\infty]$. We define the \emph{nonlocal Sobolev space induced by $\omega$ and exponent $p$} as

    $$W^{1,p}_\omega(\Omega) := \{ u \in L^p(\Omega) : \partial_{\omega, i} u \in L^p(\Omega\times\Omega), i =1,\ldots,N \},$$
    where $\partial_{\omega, i} u$ are the weak $i$-th nonlocal partial derivatives as in Definition \ref{NlWeakDerDef}.
\end{defi}

When referring to such spaces, the ``exponent $p$" part will often be omitted. Notice that, similarly to the local case, the nonlocal Sobolev induced by an admissible weight function $\omega$ and exponent $p\in[1,+\infty]$ can be written as

$$W^{1,p}_\omega(\Omega) = \left\{u\in L^p(\Omega) : \nabla_\omega u \in L^p(\Omega\times\Omega;\mathbb{R}^N)\right\}.$$

\begin{prop}[The nonlocal Sobolev space is a Banach space]\label{Proposition_Banach_Sob}
    Given an admissible weight function $\omega$ and $p\in[1,+\infty]$, the nonlocal Sobolev space $W^{1,p}_\omega(\Omega)$ is a Banach space endowed with the norm

    $$\|u\|_{W^{1,p}_\omega(\Omega)} := \|u\|_{L^p(\Omega)} + \| \nabla_\omega u \|_{L^p(\Omega\times\Omega;\mathbb{R}^N)}.$$

    Moreover, $H^1_\omega(\Omega):=W^{1,2}_\omega(\Omega)$ is a Hilbert space with the inner product

    \begin{equation}
        \langle u, v \rangle_{H^1_\omega(\Omega)} := \langle u, v \rangle_{L^2(\Omega)} + \langle \nabla_\omega u , \nabla_\omega v \rangle_{L^2(\Omega\times\Omega;\mathbb{R}^N)}.
    \end{equation}
\end{prop}

\begin{proof}
    The fact that $\|\cdot\|_{W^{1,p}_\omega(\Omega)}$ is a norm and $\langle \cdot, \cdot \rangle_{H^1_\omega(\Omega)}$ is an inner product follows from noticing that 
    $\| \nabla_\omega \cdot \|_{L^p(\Omega\times\Omega;\mathbb{R}^N)}$ is a seminorm and $\langle \nabla_\omega \cdot , \nabla_\omega \cdot \rangle_{L^2(\Omega\times\Omega;\mathbb{R}^N)}$ is positive, symmetric and linear in each of its components.

    Let us now show completeness of the space. Let $( u_n )_{n\in\mathbb{N}}$ be a Cauchy sequence in $W^{1,p}_\omega(\Omega)$, with respect to the $\|\cdot\|_{W^{1,p}_\omega(\Omega)}$ norm. Then $( u_n )_{n \in \mathbb{N}}$ and $( \nabla_\omega u_n )_{n \in \mathbb{N}}$ are Cauchy sequences in $L^p(\Omega)$ and $L^p(\Omega\times\Omega;\mathbb{R}^N)$, respectively. By completeness of these spaces we have that there exist $u \in L^p(\Omega)$ and $g = (g_1,\ldots,g_N) \in L^p(\Omega\times\Omega;\mathbb{R}^N)$ such that $u_n \to u$ in $L^p(\Omega)$ and $\nabla_\omega u_n \to g$ in $L^p(\Omega\times\Omega;\mathbb{R}^N)$.

    Now, given any $\psi \in \mathcal{C}_c^1(\Omega \times\Omega)$, notice that $\psi \in L^{p'}(\Omega\times\Omega)$ and $\partial_{\omega, i}^\ast (\psi) \in \mathcal{C}_c(\Omega)\subset L^{p'}(\Omega)$, where $p'\in[1,+\infty]$ is such that $1/p + 1/p'=1$. Thus, by Hölder's inequality, the following limits hold:

    \begin{gather}
        \lim_{n \to +\infty} \int_\Omega u_n(x) \partial_{\omega, i}^\ast (\psi) (x) \mathrm{d}x = \int_\Omega u(x) \partial_{\omega, i}^\ast (\psi) (x) \mathrm{d}x,\\
        \lim_{n \to +\infty} \int_{\Omega\times\Omega} (\partial_{\omega, i} u_n)(x,y) \psi(x,y) \mathrm{d}x \mathrm{d}y
        = \int_{\Omega\times\Omega} g_i (x,y) \psi(x,y) \mathrm{d}x \mathrm{d}y,
    \end{gather}
    and, by definition of weak nonlocal derivatives, we have

    \begin{align*}
        \int_\Omega u(x) \partial_{\omega, i}^\ast (\psi) (x) \mathrm{d}x &= \lim_{n \to +\infty} \int_\Omega u_n(x) \partial_{\omega, i}^\ast (\psi) (x) \mathrm{d}x\\
        &= - \lim_{n \to +\infty} \int_{\Omega\times\Omega} (\partial_{\omega, i} u_n)(x,y) \psi(x,y) \mathrm{d}x \mathrm{d}y\\
        &=  - \int_{\Omega\times\Omega} g_i (x,y) \psi(x,y) \mathrm{d}x \mathrm{d}y
    \end{align*}
    for any $i = 1,\ldots,N$. This computation yields that $g = \nabla_\omega u$ a.e.~in $\Omega\times\Omega$ and thus $\nabla_\omega u \in L^p(\Omega\times\Omega;\mathbb{R}^N)$, from which $(u_n)_{n \in \mathbb{N}}$ converges in $W^{1,p}_\omega(\Omega)$.
\end{proof}

Note that the inner product $\langle \cdot, \cdot \rangle_{H^1_\omega(\Omega)}$ does not generate the norm $\|\cdot\|_{W_{\omega}^{1,2}(\Omega)}$, but rather an equivalent one. The space $H^1_\omega(\Omega)$ is shown to be a Hilbert space for completeness sake. A deeper study of this space and its intrinsic properties is left as future work.

By adapting from the standard setting (see e.g.~Brezis \cite[Proposition VIII.1]{Brezis}), we can readily show the reflexive and separable character of the nonlocal Sobolev space for certain exponents.

\begin{prop}[Reflexiveness and separability of the nonlocal Sobolev space]\label{Proposition_reflexives}
    Let $\omega$ be an admissible weight function in $\Omega$. Then $W^{1,p}_\omega(\Omega)$ is reflexive for $p\in(1,+\infty)$ and separable for $p\in[1,+\infty)$.
\end{prop}

\begin{proof}
    For $p\in(1,+\infty)$ we know that the product space $E:=L^p(\Omega)\times L^p(\Omega\times\Omega;\mathbb{R}^N)$ is reflexive, and the operator $T: W^{1,p}_\omega(\Omega) \to E$ given by $T(u)=(u,\nabla_\omega u)$ is clearly an isometry (and therefore continuous). Thus, as $T(W^{1,p}_\omega(\Omega))$ is a closed subspace of $E$, it is reflexive \cite[Proposition III.17]{Brezis}, from which $W^{1,p}_\omega(\Omega)$ is also reflexive.

    Similarly, $E$ is separable for $p\in[1,+\infty)$, and therefore, as $T(W^{1,p}_\omega(\Omega))$ is a subspace of $E$, it is separable \cite[Proposition III.22]{Brezis}, from which $W^{1,p}_\omega(\Omega)$ is also separable.
\end{proof}

To wrap this subsection up, we define the traces of the functions in $W^{1,p}_\omega(\Omega)$ under suitable assumptions on $\Omega$ and $\omega$. This can be of some interest in view of subsequent applications to variational and PDE problems. We assume that the weight $\omega$ is bounded below by a positive constant near $\partial \Omega$ and that $\Omega$ is bounded and of Lipschitz boundary. From Theorem \ref{Theorem_equality_Sobolev} (which is independent of Proposition \ref{Proposition_trace_Sob} and will be proven in Subsection \ref{GeneralWeightsSubsection}), it follows that the functions in $W^{1,p}_\omega(\Omega)$ are of $W^{1,p}$ in a neighborhood of $\partial \Omega$, which means that they have a trace in the standard sense. More precisely, we prove the following:

\begin{prop}[Trace operator of the nonlocal Sobolev space]\label{Proposition_trace_Sob}
Let $\Omega$ be bounded and of Lipschitz boundary. Assume that there exists an open set $O_T \subset \Omega$ with Lipschitz boundary such that $\partial \Omega \subset \overline{O_T}$ and $c>0$ such that for all $(x,y) \in O_T\times O_T$

\begin{equation} \label{traceswBoundedBelow}
c\le \omega(x,y).
\end{equation}

Then, for all $p \in [1,+\infty]$ there exists a bounded linear operator $T_p: W^{1,p}_\omega(\Omega) \to W^{1-1/p,p}(\partial \Omega)$ such that $T_p v=v|_{\partial \Omega}$ whenever $v \in W^{1,p}_\omega(\Omega)\cap \mathcal{C}(\overline\Omega)$. Moreover, if additionally there exists a constant $C>0$ such that for all $x \in O_T$

\begin{equation} \label{tracesfwpBoundedAbove}
    f_\omega^p(x) \le C,
\end{equation}
then the operator $T_p$ is surjective.
\end{prop}

\begin{proof}
Consider the standard trace operator on $O_T$, $T_{O_T,p}: W^{1,p}(O_T) \to W^{1-1/p,p}(\partial O_T)$, which we know is linear, bounded and surjective. Moreover, given $v \in W^{1,p}_\omega(\Omega)$, one readily checks that $v|_{O_T}\in W^{1,p}_\omega(O_T)$ and $\|v|_{O_T}\|_{W^{1,p}_\omega(O_T)} \le \|v \|_{W^{1,p}_\omega(\Omega)}$, and thus the restriction operator

\begin{align}
    \cdot|_{O_T} : W^{1,p}_\omega (\Omega) &\to W^{1,p}_\omega(O_T)\\
    v &\mapsto v|_{O_T}
\end{align}
is linear, bounded and surjective\footnote{Given $\tilde v \in W^{1,p}_\omega(O_T)$, one can easily construct a $v \in W^{1,p}_\omega(\Omega)$ so that $v|_{O_T} = \tilde v$ in $O_T$. The same is true for the restriction operator $\cdot|_{\partial \Omega}: W^{1-1/p,p}(\partial O_T) \to W^{1-1/p,p}(\partial \Omega)$.}.

Furthermore, as $\partial \Omega \subset \overline{O_T}= O_T \cup \partial O_T$ and $\partial \Omega \cap O_T \subset \partial \Omega \cap \Omega = \emptyset$, we must have $\partial \Omega \subset \partial O_T$, and so the restriction operator

\begin{align}
    \cdot|_{\partial \Omega}: W^{1-1/p,p}(\partial O_T) &\to W^{1-1/p,p}(\partial \Omega)\\
     v &\mapsto v|_{\partial\Omega}
\end{align}
is also linear, bounded and surjective.

In addition, the inclusion operator $i: W^{1,p}_\omega(O_T) \to W^{1,p}(O_T)$ given by Theorem \ref{Theorem_equality_Sobolev} is clearly linear and bounded, and surjective if one additionally assumes \eqref{tracesfwpBoundedAbove} using the same theorem.

From these observations, it follows that the operator $T_p:W^{1,p}_\omega(\Omega) \to W^{1-1/p,p}(\partial \Omega)$ given by 
$$T_p := \cdot|_{\partial\Omega} \circ T_{O_T,p} \circ i \circ \cdot |_{O_T}$$
satisfies all the desired properties.
\end{proof}

In particular, if $\Omega$ is a bounded and of Lipschitz boundary and $\omega$ satisfies both of the conditions appearing in Proposition \ref{Proposition_trace_Sob}, for any $g \in W^{1-1/p,p}(\partial \Omega)$ one can define the (nonempty) trace spaces
\begin{equation}
W^{1,p}_{\omega,g}(\Omega):= \{ u \in W^{1,p}_\omega(\Omega): T_p u=g\},
\end{equation}
which are closed in $W^{1,p}_\omega(\Omega)$. A natural question, beyond the scope of this paper, would be to investigate whether the lower bound on $\partial \Omega$ can be relaxed in order to still obtain some meaningful notion of trace. 

\subsection{The spaces of functions with nonlocal weighted bounded variation}\label{section_NLBV}

As stated in the introduction, the weighted $\mathrm{NLBV}$ space was first introduced in \cite{Wang-Ng}. However, there are some issues with the statements and proofs presented in the aforementioned article. Let us discuss the issues relating to \cite[Proposition 2.3]{Wang-Ng} and \cite[Proposition 2.4]{Wang-Ng}. The first result is a nonlocal analogue of the classical Approximation Theorem for $\mathrm{BV}$ functions. In their proof, they implicitly prove that $u_n \to u$ in $L^1(\Omega)$ implies that $\mathrm{NLTV}_\omega(u_n) \to \mathrm{NLTV}_\omega(u)$ regardless of $\omega$, $(u_n)_{n \in \mathbb{N}}$ and $u$, which is not true\footnote{Indeed, a counterexample may be produced by taking $\Omega=(0,1)$, $\omega\equiv 1$, $u$ the zero function on $(0,1)$ and $(u_n)_{n \in \mathbb{N}}$ such that $u_n(x)=1-nx$ for $x \in (0,1/n)$ and $u_n=0$ on $(1/n,1)$}. On the other hand, \cite[Proposition 2.4]{Wang-Ng} is a compactness result that is contradicted if one chooses the weight function $\omega \equiv 0$ in $\Omega\times\Omega$ (which trivially satisfies the conditions imposed in \cite{Wang-Ng}). That being said, at the end of this subsection we give a simple condition on $\omega$ which, if imposed, allows the compactness result to hold true, see Proposition \ref{NLBVCompactness}.

As for an approximation-type result, let us mention that in \cite{Kindermann} the result is stated to hold for $\omega \equiv 1$ in $\Omega\times\Omega$ (see \cite[Lemma 5.3]{Kindermann}), and its proof is said to be analogous to that of the $\mathrm{BV}$ case. However, we have not been able to reproduce this argument, essentially due to the fact that, as mentioned in the introduction, we do not have an appropriate product rule for $\nabla_\omega$ and $\mathrm{div}_\omega$. Therefore, as far as we know, the validity of the result remains an open question, even for constant weight functions.

Note that the intent of the discussion above is only to justify our work and to show that our results are original, and also to explain why we cannot use some of the tools and results appearing on the two articles discussed.

Before delving straight into nonlocal total variation, we first need to establish a notion of \emph{nonlocal divergence}. We do so in the next definition.

\begin{defi}[Nonlocal divergence]
    Let $\omega$ be an admissible weight function. Given a function $\phi \in \mathcal{C}^{1}_0(\Omega \times \Omega; \mathbb{R}^N)$, we define its \emph{nonlocal divergence induced by $\omega$}, $\mathrm{div}_\omega (\phi): \Omega \to \mathbb{R}$ as
        
    $$\mathrm{div}_\omega (\phi)(x) := \int_\Omega \mathrm{div}_x\left[\omega (x,y) (\phi(y,x) - \phi(x,y)) \right] \mathrm{d}y.$$
\end{defi}

As stated in the introduction, the nonlocal divergence of a $\mathcal{C}^{1}_0(\Omega \times \Omega; \mathbb{R}^N)$ function is the sum of the nonlocal derivatives of each of its component functions.

With the nonlocal divergence established, we can naturally define the nonlocal total variation as follows:

\begin{defi}[Nonlocal total variation] \label{NLBVDef}
    Let $\omega$ be an admissible weight function. Given $u \in L^1(\Omega)$, we define its \emph{nonlocal total variation induced by $\omega$} as
        
    $$\mathrm{NLTV}_\omega(u) := \sup \left\{ \int_\Omega u(x) (\text{div}_\omega \phi)(x) \mathrm{d}x : \phi \in \mathcal{C}^1_c (\Omega \times \Omega, \mathbb{R}^N), \| \phi \|_\infty \le 1 \right\}.$$
        
    The \emph{set of functions with nonlocal bounded variation induced by $\omega$} is defined as
        
    $$\mathrm{NLBV}_\omega(\Omega) := \{ u \in L^1(\Omega) : \mathrm{NLTV}_\omega(u) < +\infty \}.$$
\end{defi}

The following characterization of nonlocal weighted total variation closely resembles the definition of the (standard) weighted total variation introduced in \cite{Baldi01}, and will be used in the proof of Theorem \ref{Theorem_main_equality_NLBV}.

\begin{lemma}[$\mathrm{NLTV}^\omega$] \label{NLTV^omegaLemma}
    Let $\omega$ be an admissible weight function. Then

    $$\mathrm{NLTV}_\omega(u)=\mathrm{NLTV}^\omega(u)$$
    for any $u \in L^1(\Omega)$, where

    $$\mathrm{NLTV}^\omega(u) := \sup \left\{ \int_\Omega u(x) (\mathrm{div}_1 \phi)(x) \mathrm{d}x : \phi \in \mathcal{C}^1_c (\Omega \times \Omega, \mathbb{R}^N), | \phi | \le \omega \mbox{ on } \Omega\times\Omega \right\}.$$
\end{lemma}

\begin{proof}
    Fix $u \in L^1(\Omega)$. Given $\phi \in \mathcal{C}_c^1(\Omega\times\Omega;\mathbb{R}^N)$ such that $\| \phi \|_\infty \le 1$, define $\phi^\omega : \Omega\times\Omega \to \mathbb{R}^N$ as $\phi^\omega := \omega \phi$. By definition $\phi^\omega \in \mathcal{C}_c^1(\Omega\times\Omega;\mathbb{R}^N)$, $|\phi^\omega|\le \omega$ in $\Omega\times\Omega$ and $(\mathrm{div}_\omega\phi) = (\mathrm{div}_1 \phi^\omega)$ in $\Omega$. Therefore,

    $$\int_\Omega u(x) (\mathrm{div}_\omega \phi)(x) \mathrm{d}x = \int_\Omega u(x) (\mathrm{div}_1 \phi^\omega)(x) \mathrm{d}x \le \mathrm{NLTV}^\omega(u),$$
    from which $\mathrm{NLTV}_\omega(u) \le \mathrm{NLTV}^\omega(u)$.

    On the other hand, given $\phi \in \mathcal{C}_c^1(\Omega\times\Omega;\mathbb{R}^N)$ such that $|\phi|\le \omega$ in $\Omega\times\Omega$, define $\phi_\omega:\Omega\times\Omega\to \mathbb{R}^N$ as
    
    $$\phi_\omega(x,y) = \frac{\phi(x,y)}{\omega(x,y)} \text{ for } \omega(x,y)>0,$$
    and defined elsewhere in a way such that $\phi_\omega \in \mathcal{C}_c^1(\Omega\times\Omega;\mathbb{R}^N)$ and $\|\phi_\omega\|_\infty \le 1$ (this can be done because $|\phi| / \omega \le 1$ in $\{ \omega(x,y)>0 \}$). Notice that $\mathrm{div}_1 \phi = \mathrm{div}_\omega \phi_\omega$ in $\Omega$, and thus

    $$\int_\Omega u(x) (\mathrm{div}_1 \phi)(x) \mathrm{d}x = \int_\Omega u(x) (\mathrm{div}_\omega \phi_\omega)(x) \mathrm{d}x \le \mathrm{NLTV}_\omega(u),$$
    which implies that $\mathrm{NLTV}^\omega(u) \le \mathrm{NLTV}_\omega(u)$.
\end{proof}

\begin{remark} \label{NLBVDefExtension}
    Similarly to the local case, the supremum in $\mathrm{NLTV}_\omega$ coincides with the same supremum taken over $\mathcal{C}_0^1(\Omega\times\Omega;\mathbb{R}^N)$ functions. The same holds true for $\mathrm{NLTV}^\omega$ as defined in Lemma \ref{NLTV^omegaLemma}. 
\end{remark}

As shown next, nonlocal total variation is lower semicontinuous. This is a desirable property which will be key in the arguments presented in Section \ref{section_applications}.

\begin{lemma}[Weak lower semicontinuity of the nonlocal total variation]\label{lemma:lsc_nlbv}
    Let $\omega$ be an admissible weight function. If $u\in L^1(\Omega)$ and $(u_n)_{n\in\mathbb{N}}$ is a sequence in $\mathrm{NLBV}_\omega(\Omega)$ such that $u_n \rightharpoonup u$ weakly in $L^1(\Omega)$, then

    \begin{equation}
        \mathrm{NLTV}_\omega(u) \le \liminf_{n \to +\infty} \mathrm{NLTV}_\omega(u_n).
    \end{equation}
\end{lemma}

\begin{proof}
    For any $\phi \in \mathcal{C}_c^1(\Omega\times\Omega;\mathbb{R}^N)$ such that $\|\phi\|_{L^\infty(\Omega\times\Omega;\mathbb{R}^N)}\le 1$ one readily checks that $\mathrm{div}_\omega(\phi)$ belongs to $L^\infty(\Omega)$, which by assumption implies that 
    \begin{equation*}
    \lim_{n \to \infty}\int_\Omega u_n \mathrm{div}_\omega(\phi) \mathrm{d}x= \int_\Omega u \; \mathrm{div}_\omega(\phi) \mathrm{d}x.
    \end{equation*}
    
    As a consequence,
    \begin{equation}
        \int_\Omega u \; \mathrm{div}_\omega(\phi) \mathrm{d}x \le \liminf_{n \to +\infty} \mathrm{NLTV}_\omega(u_n),
    \end{equation}
    and by taking supremum over all $\phi$ the result is proven.
\end{proof}

From lower semicontinuity it follows that the space of functions with nonlocal bounded variation is a Banach space for any admissible weight function.

\begin{prop}[The space of functions with nonlocal bounded variation is a Banach space] \label{Proposition_Banach_NLBV}
    Let $\omega$ be an admissible weight function. Then $\mathrm{NLBV}_\omega(\Omega)$ is a Banach space with the norm

    $$\|u\|_{\mathrm{NLBV}_\omega(\Omega)} := \|u\|_{L^1(\Omega)} + \mathrm{NLTV}_\omega(u).$$
\end{prop}

\begin{proof}
    The fact that $\| \cdot \|_{\mathrm{NLBV}_\omega(\Omega)}$ is a norm follows from noticing that $\mathrm{NLTV}_\omega$ is a seminorm, which is easily proven.

    Let us show completeness of the space. Let $(u_n)_{n\in\mathbb{N}}$ be a Cauchy sequence in $(\mathrm{NLBV}_\omega(\Omega),{\|\cdot\|_{\mathrm{NLBV}_\omega(\Omega)}})$. Then, the sequence is also Cauchy with respect to the $L^1(\Omega)$ norm, and, as $L^1(\Omega)$ is complete, there exists $u\in L^1(\Omega)$ such that $u_n\to u$ in $L^1(\Omega)$. Moreover, for any $\epsilon>0$ there exists $n_0\in \mathbb{N}$ such that $\mathrm{NLTV}_\omega(u_{n+p}-u_n) < \epsilon$ for $n\ge n_0$ and $p \in \mathbb{N}$. For a fixed $n\ge n_0$, notice that $u_{n+p} - u_n \to u - u_n$ in $L^1(\Omega)$ as we take $p$ to infinity. Thus, by lower semicontinuity (Lemma \ref{lemma:lsc_nlbv})

    \begin{equation}
        \mathrm{NLTV}_\omega(u-u_n) \le \liminf_{p \to +\infty} \mathrm{NLTV}_\omega(u_{n+p} - u_n) \le \epsilon.
    \end{equation}

    From this, we deduce that $u_n \to u$ in the $\|\cdot\|_{\mathrm{NLBV}_\omega(\Omega)}$ norm, and the fact that $u \in \mathrm{NLBV}_\omega(\Omega)$ follows from noticing that

    \begin{equation}
        \mathrm{NLTV}_\omega(u) \le \mathrm{NLTV}_\omega(u-u_n) + \mathrm{NLTV}_\omega(u_n) < +\infty
    \end{equation}
    for a large enough $n \in \mathbb{N}$.
\end{proof}

As stated in the introduction, the nonlocal Sobolev space with exponent $p=1$ and the space of functions with nonlocal bounded variation satisfy the classical continuous embedding, as shown implicitly from the following formula:

\begin{prop}[$\mathrm{NLTV}_\omega$ formula for functions with weak nonlocal derivatives] \label{C1NLTV}
    Let $\omega$ be an admissible weight function. For any $u \in L^1(\Omega)$ with weak nonlocal derivatives induced by $\omega$ in all directions, we have

    $$\mathrm{NLTV}_\omega(u) = \int_{\Omega \times \Omega} |\nabla_\omega u| \mathrm{d}x \mathrm{d}y.$$
\end{prop}

\begin{proof}
    Analogous to the local case.
\end{proof}

As for trace theory, an analogous result to Proposition \ref{Proposition_trace_Sob} can be derived in the $\mathrm{NLBV}_\omega$ setting. The proof uses Theorem \ref{Theorem_main_equality_NLBV}, which is independent of this result and will be proven in the forthcoming Subsection \ref{GeneralWeightsSubsection}.

\begin{prop}[Trace operator of the space of functions with nonlocal bounded variation]\label{Proposition_trace_NLBV}
Let $\Omega$ be bounded and of Lipschitz boundary. Assume that there exists an open set $O_T \subset \Omega$ with Lipschitz boundary such that $\partial \Omega \subset \overline{O_T}$ and $c>0$ such that for all $(x,y) \in O_T\times O_T$
\begin{equation}
c \le \omega(x,y).
\end{equation}

Then, there exists a bounded linear operator $T: \mathrm{NLBV}_\omega(\Omega) \to L^1(\partial \Omega)$ such that $Tv=v|_{\partial \Omega}$ whenever $u \in \mathrm{NLBV}_\omega(\Omega) \cap  \mathcal{C}(\overline \Omega)$. Moreover, if additionally there exists a constant $C>0$ such that for all $x\in O_T$
\begin{equation}
    \int_\Omega \omega(x,y)\mathrm{d}y \le C,
\end{equation}
then the operator $T$ is surjective.
\end{prop}
\begin{proof}
Analogous to the proof of Proposition \ref{Proposition_trace_Sob}, using Theorem \ref{Theorem_main_equality_NLBV}.
\end{proof}

In particular, if $\Omega$ is bounded and of Lipschitz boundary and $\omega$ satisfies both of the conditions appearing in Proposition \ref{Proposition_trace_NLBV}, for any $g \in L^1(\partial \Omega)$ we can define
\begin{equation}
\mathrm{NLBV}_{\omega,g}(\Omega):= \{ u \in \mathrm{NLBV}(\Omega): T u=g\},
\end{equation}
which is closed in $\mathrm{NLBV}_\omega(\Omega)$ and nonempty. 

Again, whether the lower bound near $\partial \Omega$ can be relaxed in order to still obtain some meaningful notion of trace is left as future work. 

For completeness sake, we wrap this subsection up with a compactness result analogous to \cite[Proposition 2.4]{Wang-Ng}, which recall that is false without further hypothesis on the weight function. Indeed, consider, for example, the limit case $\omega \equiv 0$ in $\Omega\times\Omega$. The condition we impose is that the weight function is bounded below from zero, and in the proof we use Theorem \ref{Theorem_main_equality_NLBV} $-$which will be proven in the forthcoming Subsection \ref{GeneralWeightsSubsection} and is completely independent of the result below$-$ and pass through the compactness of $\mathrm{BV}$. Because of this, we need to also impose that the domain is bounded and of Lipschitz boundary. 

\begin{prop}[Compactness of the space of functions with nonlocal bounded variation] \label{NLBVCompactness}
    Let $\Omega$ be bounded and of Lipschitz boundary and $\omega$ an admissible weight function such that there exists $c>0$ such that for all $(x,y)\in\Omega\times\Omega$

    \begin{equation}
        c \le \omega(x,y).
    \end{equation}
    
    Then, if $(u_n)_{n\in\mathbb{N}}$ is uniformly bounded in $\mathrm{NLBV}_\omega(\Omega)$, there exists a subsequence $(u_{n_k})_{k\in\mathbb{N}}$ and $u \in \mathrm{NLBV}_\omega(\Omega)$ such that $u_{n_k} \to u$ in $L^1(\Omega)$.
\end{prop}

\begin{proof}
    By Theorem \ref{Theorem_main_equality_NLBV}, we have the continuous embedding $\mathrm{NLBV}_\omega(\Omega) \hookrightarrow \mathrm{BV}(\Omega)$. Thus $(u_n)_{n\in\mathbb{N}}$ is also uniformly bounded in $\mathrm{BV}(\Omega)$, and by the well-known compactness of $\mathrm{BV}$ (see, e.g., \cite[Section 5.2.3]{Evans}), we can extract a subsequence $(u_{n_k})_{k\in\mathbb{N}}$ and $u \in \mathrm{BV}(\Omega)$ such that $u_{n_k} \to u$ in $L^1(\Omega)$. Finally, the fact that $u$ also belongs to $\mathrm{NLBV}_\omega(\Omega)$ follows from Lemma \ref{lemma:lsc_nlbv}.
\end{proof}

\section{Relation with classical spaces} \label{RelationClassSpacesSection}

With the nonlocal spaces defined, our first objective is to study their relation with their local counterparts. We will first study the constant weight function case, and afterwards, with the aid of the results proven in this setting, we will move onto general weight functions.

\subsection{For constant weight functions} \label{ConstWeightSubsection}

Let us start with the simplest case in which our weight function is simply $\omega \equiv 1$. Remarkably, we will see that $W^{1,p}(\Omega)=W^{1,p}_1(\Omega)$ with equivalence of norms for any $\Omega$ if $p=+\infty$ or for $\Omega$ with finite measure if $p \in [1,+\infty)$ (i.e., Theorem \ref{Theorem_equality_constant_weight}), and that $\mathrm{BV}(\Omega)=\mathrm{NLBV}(\Omega)$ with equivalence of norms for any $\Omega$ with finite measure. In both the Sobolev and $\mathrm{BV}$ case the procedure will be the same: We will be able to easily bind the seminorm of the nonlocal space by the seminorm of the corresponding local space, which will yield one of the continuous embedding. The proof of the other embedding will be more involved, and will require that we bind the seminorm of the local space by the ``full" norm of the corresponding nonlocal space (i.e., the $L^p$ norm of the function itself plus the $L^p$ norm of the nonlocal gradient in the Sobolev case, or the $L^1$ norm of the function plus its nonlocal bounded variation).

Moreover, in this subsection, an insightful characterization of both nonlocal seminorms will be given in Propositions \ref{WpSeminormChar} and \ref{NLBVSeminormChar}.

We begin by comparing the nonlocal Sobolev space corresponding to a constant weight function $W^{1,p}_1(\Omega)$ with the classic Sobolev space $W^{1,p}(\Omega)$. The first step is a technical observation:

\begin{lemma} \label{integrallemma}
    Let $p\in[1,+\infty]$ and $\Omega$ be of finite measure if $p \in [1,+\infty)$, or general if $p=+\infty$. Let $f: \Omega \to \mathbb{R}^m$ with $m \in \mathbb{N}$ be a measurable function. Then,
    $$f \in L^p(\Omega;\mathbb{R}^m) \iff \tilde f \in L^p(\Omega\times\Omega;\mathbb{R}^m),$$
    where we have set $\tilde f(x,y) := f(y)-f(x)$ for all $(x,y) \in \Omega \times \Omega$.
\end{lemma}

\begin{proof}
    We start with the $p\in[1,+\infty)$ case. Notice that
    
    $$|f(y)-f(x)|^p \le 2^{p-1}\left(|f(y)|^p + |f(x)|^p \right).$$

    Using this, the left to right implication follows, as
    
    \begin{align}
        \int_\Omega \int_\Omega |f(x)-f(y)|^p \mathrm{d}x \mathrm{d}y &\le 2^{p-1}|\Omega| \left( \int_\Omega |f(x)|^p \mathrm{d}x + \int_\Omega |f(y)|^p \mathrm{d}y \right)\\
        &= 2^p |\Omega| \int_\Omega |f(x)|^p \mathrm{d}x \label{ImplContEmbed}
    \end{align}

    For the other implication, Fubini's theorem yields that
    
    $$\int_\Omega |f(x)-f(y)|^p \mathrm{d}x$$
    exists and is finite for almost every $y \in \Omega$. Thus, for such a $y$ we have

    $$\int_\Omega |f(x)|^p \mathrm{d}x \le \int_\Omega |f(x)-f(y)|^p \mathrm{d}x + |\Omega| |f(y)|^p < +\infty.$$

    We move onto the $p=+\infty$ case. We have

    \begin{equation} \label{ImplContEmbedInfty}
        \esssup_{(x,y)\in\Omega\times\Omega}|f(y)-f(x)|\le \esssup_{(x,y)\in\Omega\times\Omega}(|f(y)|+|f(x)|)\le 2  \esssup_{x\in\Omega} f(x),
    \end{equation}
    and the left to right implication is proven.

    For the other implication, assume first that $\tilde f \in L^\infty(\Omega\times\Omega;\mathbb{R}^m)$ and pick $y\in\Omega$. Then, for almost any $x\in\Omega$

    \begin{equation}
        |f(x)|\le |f(y)| + |f(y)-f(x)| \le |f(y)| + \|\tilde f\|_{L^\infty(\Omega\times\Omega;\mathbb{R}^m)} <+\infty.
    \end{equation}

    As the bound is independent of $x$, we deduce that $f \in L^\infty(\Omega;\mathbb{R}^m)$.
\end{proof}

For $p\in[1,+\infty)$, we emphasize that the finite measure of the domain is key for the lemma to hold, see the forthcoming Lemma \ref{NonConstantInfSeminorm}.

As a direct consequence of Lemma \ref{integrallemma} and Propositions \ref{WeakDerImplyNlWeakDer} and \ref{NlWeakDerImplyWeakDer}, we have that $W^{1,p}(\Omega)=W^{1,p}_1(\Omega)$ for any $p\in[1,+\infty)$ and $\Omega$ with finite measure, or for $p=+\infty$ and any $\Omega$. With the help of the aforementioned results, we can now prove Theorem \ref{Theorem_equality_constant_weight}, which states that additionally we have equivalence of norms.

\begin{proof}[Proof of Theorem \ref{Theorem_equality_constant_weight}]
    The $W^{1,p}(\Omega) \hookrightarrow W_1^{1,p}(\Omega)$ embedding follows from Proposition \ref{WeakDerImplyNlWeakDer} and inequality \eqref{ImplContEmbed} for $p\in[1,+\infty)$, or inequality \eqref{ImplContEmbedInfty} for $p=+\infty$.

    We move onto proving $W_1^{1,p}(\Omega) \hookrightarrow W^{1,p}(\Omega)$. We start with the $p\in[1,+\infty)$ case. Fix any $\theta \in \mathcal{C}_c^1(\Omega)$ satisfying \eqref{ThetaIntCond}. For any $u \in W_1^{1,p}(\Omega)$ we have, by Proposition \ref{NlWeakDerImplyWeakDer} (more specifically \eqref{WeakGradInTermsOfNlWeakGrad}), that

    \begin{equation}
        \nabla u(x) = - \frac{1}{\int_\Omega \theta(z) \mathrm{d}z} \int_\Omega \left[ u(y) \nabla \theta(y) + \nabla_1 u(x,y) \theta(y) \right] \mathrm{d}y.
    \end{equation}

    Then\footnote{In the third line we have used that, if $1/p + 1/{p'}=1$, Hölder's inequality yields
    $$\left(\int_\Omega |f(x)|\mathrm{d}x\right)^p \le \left( |\Omega|^{\frac{1}{p'}} \left[ \int_\Omega |f(x)|^p \mathrm{d}x \right]^{\frac{1}{p}} \right)^p = |\Omega|^{\frac{p}{p'}} \int_\Omega |f(x)|^p \mathrm{d}x = |\Omega|^{p-1} \int_\Omega |f(x)|^p \mathrm{d}x.$$}

    \begin{align}
        \int_\Omega |\nabla u(x)|^p \mathrm{d}x &= \int_\Omega \left| -\frac{1}{\int_\Omega \theta(z) \mathrm{d}z} \int_\Omega \left[ u(y) \nabla \theta(y) + \nabla_1 u(x,y) \theta(y) \right] \mathrm{d}y \right|^p \mathrm{d}x\\
        &\le \frac{1}{\left|\int_\Omega \theta(z) \mathrm{d}z\right|^p} \int_\Omega \left( \int_\Omega \left| u(y) \nabla \theta(y) + \nabla_1 u(x,y) \theta(y) \right| \mathrm{d}y \right)^p \mathrm{d}x\\
        &\le \frac{|\Omega|^{p-1}}{\left|\int_\Omega \theta(z) \mathrm{d}z\right|^p} \int_\Omega \int_\Omega \left| u(y) \nabla \theta(y) + \nabla_1 u(x,y) \theta(y) \right|^p \mathrm{d}x\mathrm{d}y\\
        &\le \frac{|\Omega|^{p-1} 2^{p-1}}{\left|\int_\Omega \theta(z) \mathrm{d}z\right|^p} \int_\Omega \int_\Omega \left| u(y) \nabla \theta(y)\right|^p+ \left|\nabla_1 u(x,y) \theta(y) \right|^p \mathrm{d}x\mathrm{d}y\\
        &\le \frac{(2|\Omega|)^{p-1}}{\left|\int_\Omega \theta(z) \mathrm{d}z\right|^p} \left( |\Omega| \|\nabla \theta\|_{L^\infty(\Omega;\mathbb{R}^N)}^p \int_\Omega |u(x)|^p \mathrm{d}x\right.\\
        &\qquad \qquad \qquad \qquad \qquad \left.+ \|\theta\|_{L^\infty(\Omega)}^p \int_{\Omega\times\Omega} |\nabla_1 u(x,y)|^p \mathrm{d}x\mathrm{d}y \right)\\
        &\le C^p(\theta) \|u\|_{W^{1,p}_1(\Omega)}^p,
    \end{align}
    where

    \begin{equation} \label{CpConstantDef}
        C^p(\theta) := \frac{(2|\Omega|)^{p-1}}{\left|\int_\Omega \theta(z) \mathrm{d}z\right|^p} \max\left\{ |\Omega| \|\nabla \theta\|_{L^\infty(\Omega;\mathbb{R}^N)}^p, \|\theta\|_{L^\infty(\Omega)}^p \right\}.
    \end{equation}
    The chain of inequalities yields the desired embedding. 
    
    We now prove the $p=+\infty$ case. Again, fix any $\theta \in \mathcal{C}_c^1(\Omega)$ satisfying \eqref{ThetaIntCond}. For any $u \in W_1^{1,\infty}(\Omega)$ we have, by Proposition \ref{NlWeakDerImplyWeakDer} (more specifically \eqref{WeakGradInTermsOfNlWeakGrad}), that

    \begin{equation}
        \nabla u(x) = - \frac{1}{\int_\Omega \theta(z) \mathrm{d}z} \int_\Omega \left[ u(y) \nabla \theta(y) + \nabla_1 u(x,y) \theta(y) \right] \mathrm{d}y.
    \end{equation}

    With this, for almost any $x\in\Omega$ we have

    \begin{align}
        |\nabla u(x)| &= \left| -\frac{1}{\int_\Omega \theta(z) \mathrm{d}z} \int_\Omega \left[ u(y) \nabla \theta(y) + \nabla_1 u(x,y) \theta(y) \right] \mathrm{d}y \right|\\
        &\le \frac{1}{\left|\int_\Omega \theta(z) \mathrm{d}z\right|} \left( \int_\Omega |u(y)| |\nabla\theta(y)| \mathrm{d}y + \int_\Omega |\nabla_1 u(x,y)| |\theta(y)| \mathrm{d}y\right)\\
        &\le \frac{1}{\left|\int_\Omega \theta(z) \mathrm{d}z\right|} \left( \|u\|_{L^\infty(\Omega)} \|\nabla\theta\|_{L^1(\Omega;\mathbb{R}^N)} + \|\nabla_1 u\|_{L^\infty(\Omega\times\Omega;\mathbb{R}^N)} \|\theta\|_{L^1(\Omega)} \right)\\
        &\le C^\infty(\theta) \|u\|_{W^{1,\infty}_1(\Omega)},
    \end{align}
    where

    \begin{equation} \label{CInftyConstantDef}
        C^\infty(\theta) := \frac{1}{\left|\int_\Omega \theta(z) \mathrm{d}z\right|} \max\left\{ \|\nabla\theta\|_{L^1(\Omega;\mathbb{R}^N)}, \|\theta\|_{L^1(\Omega)} \right\}.
    \end{equation}

    The chain of inequalities yields the desired embedding.
\end{proof}

Finding an optimal bound for $C^p(\theta)$ and $C^\infty(\theta)$ as defined in \eqref{CpConstantDef} and \eqref{CInftyConstantDef} falls beyond the scope of this paper and is left as future work. As a preliminary bound, notice that $2^{p-1}/|\Omega| \le C^p(\theta)$ and $1 \le C^\infty(\theta)$
for any $\theta \in \mathcal{C}_c^1(\Omega)$ satisfying \eqref{ThetaIntCond}.

The following proposition characterizes the $W^{1,p}_1$ seminorm, and is essentially the Sobolev version of \cite[Lemma 5.2]{Kindermann}. It shows that the seminorm of $W^{1,p}_1$ is equivalent to the seminorm of the quotient space $W^{1,p}\backslash \sim_{\text{aff}}$, where $\sim_{\text{aff}}$ is the equivalence relation that relates functions in $W^{1,p}_1$ which differ a.e.~by an affine function.

\begin{prop} \label{WpSeminormChar}
    Let $p\in[1,+\infty]$ and $\Omega$ be of finite measure if $p \in [1,+\infty)$ or general if $p=+\infty$. For $p\in[1,+\infty)$ and $u \in W^{1,p}_1(\Omega)$ we have

    \begin{equation} \label{AffInequalityp}
        |\Omega| \inf_{g \emph{ affine}} \int_\Omega|\nabla(u-g)|^p \mathrm{d}x \le \int_{\Omega\times\Omega} |\nabla_1 u|^p \mathrm{d}x\mathrm{d}y \le 2^p |\Omega| \inf_{g \emph{ affine}} \int_\Omega|\nabla(u-g)|^p \mathrm{d}x,
    \end{equation}
    and for $u \in W^{1,\infty}_1(\Omega)$ we have

    \begin{equation} \label{AffInequalityInfty}
        \inf_{g \emph{ affine}} \esssup_{\Omega}|\nabla(u-g)| \le \esssup_{\Omega\times\Omega} |\nabla_1 u| \le 2 \inf_{g \emph{ affine}} \esssup_\Omega |\nabla(u-g)|.
    \end{equation}
\end{prop}

\begin{proof}
    We start with the $p\in[1,+\infty)$ case. The second inequality follows from noticing that

    $$\int_{\Omega\times\Omega} |\nabla_1 u|\mathrm{d}x\mathrm{d}y = \int_{\Omega\times\Omega} |\nabla_1 (u-g)|\mathrm{d}x\mathrm{d}y,$$
    for any affine function $g$, then applying the inequality \eqref{ImplContEmbed} previously computed, and finally taking infimum over all affine $g$.

    We move onto the first inequality. Fix any $\theta \in \mathcal{C}_c^1(\Omega)$ satisfying \eqref{ThetaIntCond}, and consider the affine function

    \begin{equation} \label{g_thetaDef}
        g_\theta(x) := -\frac{1}{\int_\Omega \theta(z) \mathrm{d}z} \left( \int_\Omega u(y) \nabla \theta(y) \mathrm{d}y \right) \cdot x.
    \end{equation}
    
    Recall (equation \eqref{WeakGradInTermsOfNlWeakGrad}) that

    \begin{equation}
        \nabla u(x) =  -\frac{1}{\int_\Omega \theta(z) \mathrm{d}z} \int_\Omega \left[ u(y) \nabla \theta(y) + \nabla_1 u(x,y) \theta(y) \right] \mathrm{d}y.
    \end{equation}

    Thus

    \begin{equation} \label{Grad u-g_theta}
        \nabla (u-g_\theta)(x) = \nabla u(x) - \nabla g_\theta(x) = - \frac{1}{\int_\Omega \theta(z) \mathrm{d}z} \int_\Omega \nabla_1 u(x,y) \theta(y) \mathrm{d}y,
    \end{equation}
    and we have

    \begin{align}
        \int_\Omega |\nabla(u-g_{\theta})(x)|^p \mathrm{d}x &= \int_\Omega \left| \int_\Omega - \frac{1}{\int_\Omega \theta(z) \mathrm{d}z} |\nabla_1 u(x,y) \theta(y)| \mathrm{d}y \right|^p \mathrm{d}x\\
        &\le |\Omega|^{p-1} \int_\Omega \int_\Omega \frac{1}{\left|\int_\Omega \theta(z) \mathrm{d}z\right|^p} |\nabla_1 u(x,y)|^p |\theta(y)|^p \mathrm{d}x\mathrm{d}y\\
        &\le \frac{|\Omega|^{p-1} \|\theta\|_{L^\infty(\Omega)}^p}{\left|\int_\Omega \theta(z) \mathrm{d}z\right|^p} \int_{\Omega\times\Omega} |\nabla_1 u(x,y)|^p \mathrm{d}x\mathrm{d}y.
    \end{align}

    This means that

    \begin{equation}
        \inf_{g \text{ affine}} \int_\Omega|\nabla(u-g)|^p \mathrm{d}x \le \int_\Omega|\nabla(u-g_\theta)|^p \mathrm{d}x \le K^p(\theta) \int_{\Omega\times\Omega} |\nabla_1 u|^p \mathrm{d}x\mathrm{d}y,
    \end{equation}
    where

    \begin{equation}
        K^p(\theta) := \frac{|\Omega|^{p-1} \|\theta\|_{L^\infty(\Omega)}^p}{\left|\int_\Omega \theta(z) \mathrm{d}z\right|^p}.
    \end{equation}

    Notice that, for any $\theta\in \mathcal{C}_c^1(\Omega)$ we have

    \begin{equation}
        \left|\int_\Omega \theta(z) \mathrm{d}z\right|^p \le |\Omega|^p \|\theta\|_{L^\infty(\Omega)}^p \implies K^p(\theta)=\frac{|\Omega|^{p-1} \|\theta\|_{L^\infty(\Omega)}^p}{\left|\int_\Omega \theta(z) \mathrm{d}z\right|^p} \ge  \frac{|\Omega|^{p-1} \|\theta\|_{L^\infty(\Omega)}^p}{|\Omega|^p \|\theta\|_{L^\infty(\Omega)}^p} = \frac{1}{|\Omega|}.
    \end{equation}

    Now, let $(\theta_n)_{n \in \mathbb{N}}$ be a sequence in $\mathcal{C}_c^1(\Omega)$ such that $\|\theta_n\|_{L^\infty(\Omega)} \le C$ for all $n \in \mathbb{N}$ and $\theta_n \to C$ in $L^1(\Omega)$, where $C>0$ is a fixed constant (for example, we can construct this sequence by defining $\theta_n = C$ in $\{ d(x,\partial \Omega) \ge 1/n \}$ and smoothing $\theta_n$ in the rest of $\Omega$ so that $\theta_n \in \mathcal{C}_c^1(\Omega)$ and $\|\theta_n\|_{L^\infty(\Omega)} \le C$). Note that then

    \begin{equation}
        \|\theta_n\|_{L^\infty(\Omega)} \xrightarrow[n \to +\infty]{} C, \qquad \qquad \left| \int_\Omega \theta_n(z) \mathrm{d}z \right| \xrightarrow[n \to +\infty]{} C|\Omega|,
    \end{equation}
    and thus

    \begin{equation}
        K^p(\theta_n)=\frac{|\Omega|^{p-1} \|\theta_n\|_{L^\infty(\Omega)}^p}{\left|\int_\Omega \theta_n(z) \mathrm{d}z\right|^p} \xrightarrow[n \to +\infty]{} \frac{|\Omega|^{p-1} C^p}{C^p |\Omega|^p} = \frac{1}{|\Omega|}.
    \end{equation}

    Therefore we have the stronger bound

    \begin{equation}
        |\Omega| \inf_{g \text{ affine}} \int_\Omega|\nabla(u-g)|^p \mathrm{d}x \le \int_{\Omega\times\Omega} |\nabla_1 u|^p \mathrm{d}x\mathrm{d}y.
    \end{equation}

    We now prove the $p=+\infty$ case. For the second inequality, notice that, by using \eqref{ImplContEmbedInfty}, for almost any $(x,y)\in\Omega\times\Omega$ we have

    \begin{equation}
        \esssup_{(x,y) \in \Omega\times\Omega}|\nabla_1 u(x,y)| = \esssup_{(x,y) \in \Omega\times\Omega}|\nabla_1 (u-g) (x,y)| \le 2 \esssup_{x \in \Omega} |\nabla(u-g)(x)|
    \end{equation}
    for any affine function $g$, hence, by taking infimum over all $g$ the inequality follows.

    Let us move onto the first inequality. Fix any $\theta\in\mathcal{C}_c^1(\Omega)$ satisfying \eqref{ThetaIntCond}, and consider the affine function $g_\theta$ as in \eqref{g_thetaDef}. By \eqref{Grad u-g_theta} we have that for almost any $x \in \Omega$

    \begin{align}
        |\nabla (u-g_\theta)(x)| &\le \frac{1}{\left|\int_\Omega \theta(z) \mathrm{d}z\right|} \int_\Omega \left| \nabla_1 u(x,y) \theta(y) \right| \mathrm{d}y\\
        &\le \frac{1}{\left|\int_\Omega \theta(z) \mathrm{d}z\right|} \|\nabla_1 u(x,\cdot)\|_{L^\infty(\Omega);\mathbb{R}^N} \|\theta\|_{L^1(\Omega)}.
    \end{align}

    Therefore, if we take essential supremum over almost any $x \in \Omega$ we have

    \begin{equation}
         \inf_{g \text{ affine}} \esssup_{\Omega}|\nabla(u-g)|\le \esssup_{\Omega}|\nabla(u-g_\theta)| \le \frac{\|\theta\|_{L^1(\Omega)}}{\left|\int_\Omega \theta(z) \mathrm{d}z\right|} \esssup_{\Omega\times\Omega} |\nabla_1 u|.
    \end{equation}

    Notice that for any $\theta \in \mathcal{C}_c^1(\Omega)$ satisfying \eqref{ThetaIntCond}

    \begin{equation}
        \left|\int_\Omega \theta(z) \mathrm{d}z\right| \le \|\theta\|_{L^1(\Omega)} \implies \frac{\|\theta\|_{L^1(\Omega)}}{\left|\int_\Omega \theta(z) \mathrm{d}z\right|} \ge 1,
    \end{equation}
    so the optimal bound is obtained whenever

    \begin{equation}
        \left|\int_\Omega \theta(z) \mathrm{d}z\right| = \|\theta\|_{L^1(\Omega)},
    \end{equation}
    which holds if and only if $\theta \in \mathcal{C}_c^1(\Omega)$ is such that $\theta \ge 0$ in $\Omega$. Thus, by choosing a $\theta$ satisfying this property (and also \eqref{ThetaIntCond}), we can obtain the stronger bound

    \begin{equation}
         \inf_{g \text{ affine}} \esssup_{\Omega}|\nabla(u-g)|\le  \esssup_{\Omega\times\Omega} |\nabla_1 u|.
    \end{equation}
\end{proof}

Now we study the relation between the space of functions with nonlocal bounded variation induced by the constant weight function $1$, which recall that we will denote by $\mathrm{NLBV}$, and the classical $\mathrm{BV}$ space. The conclusions will be the same as in the Sobolev case: The spaces are equal and we have equivalence of norms, provided the domain has finite measure. Additionally, an alternative proof to \cite[Lemma 5.2]{Kindermann} will be derived, see Proposition \ref{NLBVSeminormChar}.

As already mentioned in the introduction, the fact that $\mathrm{BV}(\Omega)=\mathrm{NLBV}(\Omega)$ with equivalence of norms for bounded $\Omega$ can be proven using the results in \cite{Kindermann}. However, one of the results needed to derive such a proof is the approximation theorem (see \cite[Lemma 5.1]{Kindermann}), which, at the time of writing this paper and to our knowledge, is not clear. This being said, our argument does not rely on any approximation-type result, and instead follows an approach similar to the Sobolev case as seen at the start of this subsection. Moreover, the condition on the domain is relaxed to only having finite measure.

\begin{prop}[Relation between $\mathrm{NLBV}(\Omega)$ and $\mathrm{BV}(\Omega)$ when $\Omega$ has finite measure] \label{NLBV=BVEquivNorm}
    Let $\Omega$ be of finite measure. Then $\mathrm{NLBV}(\Omega)=\mathrm{BV}(\Omega)$ with equivalence of norms.
\end{prop}

\begin{proof}
    We start with the embedding $\mathrm{BV}(\Omega)\hookrightarrow\mathrm{NLBV}(\Omega)$. Let $u \in \mathrm{BV}(\Omega)$. For any $\phi \in \mathcal{C}_c^1(\Omega\times\Omega;\mathbb{R}^N)$ such that $\|\phi\|_{L^\infty(\Omega\times\Omega;\mathbb{R}^N)}\le 1$, we have that

    \begin{equation}
        \varphi_\phi^1(x):= \frac{1}{|\Omega|}\int_\Omega \phi(y,x) \mathrm{d}y, \; \varphi_\phi^2(x):= -\frac{1}{|\Omega|}\int_\Omega \phi(x,y) \mathrm{d}y 
    \end{equation}
    are functions belonging to $\mathcal{C}_c^1(\Omega;\mathbb{R}^N)$ such that $\|\varphi_\phi^1\|_{L^\infty(\Omega;\mathbb{R}^N)}, \|\varphi_\phi^2\|_{L^\infty(\Omega;\mathbb{R}^N)} \le 1$. Thus, by Leibniz rule

    \begin{align}
        \int_\Omega u(x)\; \mathrm{div}_1(\phi)(x) \mathrm{d}x &= \int_\Omega u(x)\; \mathrm{div}_x \left[\int_\Omega (\phi(y,x) - \phi(x,y)) \mathrm{d}y\right] \mathrm{d}x\\
        &= |\Omega| \int_\Omega u(x)\; \mathrm{div}(\varphi_\phi^1)(x) \mathrm{d}x + |\Omega| \int_\Omega u(x)\; \mathrm{div}(\varphi_\phi^2)(x) \mathrm{d}x\\
        &\le 2 |\Omega| \mathrm{TV}(u).
    \end{align}

    Therefore, by taking supremum over all $\phi$ we obtain

    \begin{equation} \label{NLTV<TV}
        \mathrm{NLTV}(u)\le 2 |\Omega| \mathrm{TV}(u),
    \end{equation}
    from which we deduce that $\mathrm{BV}(\Omega) \hookrightarrow \mathrm{NLBV}(\Omega)$.

    Let us now prove that $\mathrm{NLBV}(\Omega)\hookrightarrow\mathrm{BV}(\Omega)$. Let $u \in \mathrm{NLBV}(\Omega)$, and fix $\theta \in \mathcal{C}_c^1(\Omega)$ satisfying \eqref{ThetaIntCond}. For any $\varphi \in \mathcal{C}_c^1(\Omega;\mathbb{R}^N)$ such that $\|\varphi\|_{L^\infty(\Omega;\mathbb{R}^N)} \le 1$, we have that
    
    \begin{equation}
        \phi_{\theta,\varphi}(x,y):= \frac{\theta(x) \varphi(y)}{\|\theta\|_{L^\infty(\Omega)}}
    \end{equation}
    is a function of $\mathcal{C}_c^1(\Omega\times\Omega;\mathbb{R}^N)$ such that $\|\phi_{\theta,\varphi}\|_{L^\infty(\Omega\times\Omega;\mathbb{R}^N)} \le 1$. Moreover

    \begin{align}
        \|\theta\|_{L^\infty(\Omega)} \int_\Omega u(x) \; \mathrm{div}_1(\phi_{\theta,\varphi})(x) \mathrm{d}x&= \int_\Omega u(x) \int_\Omega \mathrm{div}_x \left( \theta(y) \varphi(x) - \theta(x) \varphi(y) \right) \mathrm{d}y \mathrm{d}x\\
        &= \int_\Omega u(x) \int_\Omega \left[ \theta(y) \; \mathrm{div}(\varphi)(x) - \varphi(y) \cdot \nabla \theta(x) \right] \mathrm{d}y \mathrm{d}x\\
        &= \left( \int_\Omega \theta(y) \mathrm{d}y \right) \int_\Omega u(x) \; \mathrm{div} (\varphi)(x) \mathrm{d}x\\
        & \qquad \qquad - \int_\Omega u(x) \left[ \nabla \theta(x) \cdot \int_\Omega \varphi(y) \mathrm{d}y \right] \mathrm{d}x. \label{div1intermsofdiv}
    \end{align}

    Thus, by definition of $\mathrm{NLTV}$ and the fact that $\|\varphi\|_{L^\infty(\Omega;\mathbb{R}^N)} \le 1$ we have, after relabeling the variable of the $\theta$ integral

    \begin{align}
        \int_\Omega u(x) & \; \mathrm{div} (\varphi)(x) \mathrm{d}x\\
        &= \frac{1}{\int_\Omega \theta(z) \mathrm{d}z} \left[ \|\theta\|_{L^\infty(\Omega)} \int_\Omega u(x) \; \mathrm{div}_1(\phi_{\theta,\varphi})(x) \mathrm{d}x \right.\\
        &\qquad \qquad \left. + \int_\Omega u(x) \left[ \nabla \theta(x) \cdot \int_\Omega \varphi(y) \mathrm{d}y \right] \mathrm{d}x \right]\\
        &\le \frac{1}{|\int_\Omega \theta(z) \mathrm{d}z|} \left[ \|\theta\|_{L^\infty(\Omega)} \mathrm{NLTV}(u) + \int_\Omega |u(x)||\nabla \theta (x)| \int_\Omega |\varphi(y)| \mathrm{d}y \mathrm{d}x \right]\\
        &\le \frac{1}{|\int_\Omega \theta(z) \mathrm{d}z|} \left[ \|\theta\|_{L^\infty(\Omega)} \mathrm{NLTV}(u) + |\Omega| \|\nabla \theta\|_{L^\infty(\Omega;\mathbb{R}^N)} \|u\|_{L^1(\Omega)} \right]\\
        &\le C^1(\theta) \|u\|_{\mathrm{NLBV}(\Omega)},
    \end{align}
    where

    \begin{equation} \label{C1BVConstantDef}
        C^1(\theta) = \frac{1}{|\int_\Omega \theta(z) \mathrm{d}z|} \max\left\{ \|\theta\|_{L^\infty(\Omega)}, |\Omega| \|\nabla \theta\|_{L^\infty(\Omega;\mathbb{R}^N)} \right\}>0.
    \end{equation}

    Finally, taking supremum over all $\varphi$ gives the inequality

    \begin{equation} \label{NLBVContEmbBVEq}
        \mathrm{TV}(u) \le C^1(\theta) \|u\|_{\mathrm{NLBV}(\Omega)},
    \end{equation}
    which yields the desired continuous embedding.
\end{proof}

Note that the constant $C^1(\theta)$ defined above in \eqref{C1BVConstantDef} coincides with the constant $C^p(\theta)$ with $p=1$ as defined in \eqref{CpConstantDef}. Again, finding an optimal bound for $C^1(\theta)$ is left as future work.

The following result corresponds to \cite[Lemma 5.2]{Kindermann}, with the difference that our argument does not depend on the validity of the approximation theorem, and that the conditions on $\Omega$ are relaxed from it being bounded and an extension domain to it simply having finite measure. It shows that the seminorm of $\mathrm{NLBV}$ is equivalent to the seminorm of the quotient space $\mathrm{BV}\backslash \sim_{\text{aff}}$, where $\sim_{\text{aff}}$ is the equivalence relation that relates functions in $\mathrm{BV}$ which differ a.e.~by an affine function. 

\begin{prop} \label{NLBVSeminormChar}
    Let $\Omega$ be of finite measure, and $u \in \mathrm{NLBV}(\Omega)$. Then

    \begin{equation}
        |\Omega| \inf_{g \emph{ affine}} \mathrm{TV}(u-g) \le \mathrm{NLTV}(u) \le 2 |\Omega| \inf_{g \emph{ affine}} \mathrm{TV}(u-g).
    \end{equation}
\end{prop}

\begin{proof}
    The second inequality follows from noticing that $\mathrm{NLTV}(u)=\mathrm{NLTV}(u-g)$ for any affine function $g$ and then applying inequality \eqref{NLTV<TV} appearing in the proof of Proposition \ref{NLBV=BVEquivNorm}.

    Let us prove the first inequality. Fix any $\theta \in \mathcal{C}_c^1(\Omega)$ satisfying \eqref{ThetaIntCond} and such that $\|\theta\|_{L^\infty(\Omega)} \le 1$,
    and consider the affine function

    \begin{equation}
        g_\theta(x) := \left( -\frac{1}{\int_\Omega \theta(z)\mathrm{d}z} \int_\Omega u(y) \nabla \theta(y) \mathrm{d}y\right)\cdot x.
    \end{equation}

    By the previous computation \eqref{div1intermsofdiv} we have

    \begin{align}
        \int_\Omega u(x) \; \mathrm{div} (\varphi)(x) \mathrm{d}x
        &= \frac{1}{\int_\Omega \theta(z) \mathrm{d}z} \left[ \int_\Omega u(x) \; \mathrm{div}_1(\phi_{\theta,\varphi})(x) \mathrm{d}x \right.\\
        &\qquad \qquad \qquad \qquad \left. + \left( \int_\Omega u(x) \nabla \theta(x) \mathrm{d}x \right) \cdot \int_\Omega \varphi(y) \mathrm{d}y \right],
    \end{align}
    where $\phi_{\theta,\varphi}(x,y) := \theta(x) \varphi(y)$. Note that $\phi_{\theta,\varphi} \in \mathcal{C}_c^1(\Omega\times\Omega;\mathbb{R}^N)$ and is such that

    \begin{equation}
        \|\phi_{\theta,\varphi}\|_{L^\infty(\Omega\times\Omega;\mathbb{R}^N)}
        \le \|\theta\|_{L^\infty(\Omega)} \|\varphi\|_{L^\infty(\Omega;\mathbb{R}^N)} \le 1.
    \end{equation}

    Moreover, as $g_\theta$ is weakly differentiable (it is affine), we have

    \begin{align}
        \int_\Omega g_\theta(x) \mathrm{div}(\varphi)(x) \mathrm{d}x &= - \int_\Omega \nabla g_\theta(x) \cdot \varphi(x) \mathrm{d}x\\
        &= \frac{1}{\int_\Omega \theta(z) \mathrm{d}z} \left( \int_\Omega u(y) \nabla \theta(y) \mathrm{d}y \right) \cdot \int_\Omega \varphi(x) \mathrm{d}x.
    \end{align}

    With all this in mind, we have that

    \begin{align}
        \int_\Omega (u(x)-g_\theta(x)) &\mathrm{div} (\varphi)(x) \mathrm{d}x\\
        &= \int_\Omega u(x) \mathrm{div} (\varphi)(x) \mathrm{d}x - \int_\Omega g_\theta(x) \mathrm{div} (\varphi)(x) \mathrm{d}x\\
        &= \frac{1}{\int_\Omega \theta(z) \mathrm{d}z} \int_\Omega u(x) \; \mathrm{div}_1(\phi_{\theta,\varphi})(x) \mathrm{d}x\\
        &\le \frac{1}{\left|\int_\Omega \theta(z) \mathrm{d}z\right|} \mathrm{NLTV}(u).
    \end{align}

    Therefore, by taking supremum over all $\varphi$, we have

    \begin{equation}
        \inf_{g \text{ affine}} \mathrm{TV}(u-g) \le \mathrm{TV}(u-g_\theta) \le \frac{1}{\left|\int_\Omega \theta(z) \mathrm{d}z\right|} \mathrm{NLTV}(u).
    \end{equation}

    Finally, the bound above works for all $\theta \in \mathcal{C}_c^1(\Omega)$ satisfying \eqref{ThetaIntCond} and such that $\|\theta\|_{L^\infty(\Omega)} \le 1$, 
    and the optimal bound is obtained as $\theta \to 1$ in $L^1(\Omega)$, in which case

    \begin{equation}
        \frac{1}{\left|\int_\Omega \theta(z) \mathrm{d}z\right|} \to \frac{1}{|\Omega|}.
    \end{equation}

    Thus, we have the stronger bound

    \begin{equation}
        |\Omega| \inf_{g \text{ affine}} \mathrm{TV}(u-g) \le \mathrm{NLTV}(u).
    \end{equation}
\end{proof}

\subsection{For general weight functions} \label{GeneralWeightsSubsection}

In this subsection we give the proofs of Theorems \ref{Theorem_main_equality_NLBV} and \ref{Theorem_equality_Sobolev}, which give some necessary and/or sufficient conditions on our weight functions that ensure that the continuous embeddings between $\mathrm{BV}$ and $\mathrm{NLBV}_\omega$ or $W^{1,p}$ and $W^{1,p}_\omega$ are satisfied.

Let us begin with Theorem \ref{Theorem_equality_Sobolev}, as its proof can be adapted to prove part of Theorem \ref{Theorem_main_equality_NLBV}. The proofs of all of the statements but one are rather direct. The statement involving the embedding $W^{1,p}_\omega(\Omega) \hookrightarrow W^{1,p}(\Omega)$ follows directly from Remark \ref{1impliesomegaRemark}, and the statement which gives a sufficient condition for the embedding $W^{1,p}(\Omega) \hookrightarrow W^{1,p}_\omega(\Omega)$ follows from the fact that, for $u \in W^{1,p}(\Omega)$, we have the formula $\nabla_\omega u(x,y) = \omega(x,y) \left(\nabla u(y) - \nabla u(x)\right)$ for almost every $(x,y) \in \Omega\times\Omega$ due to Remark \ref{1impliesomegaRemark} and Proposition \ref{WeakDerImplyNlWeakDer}. The remaining statement, which gives a necessary condition for the embedding $W^{1,p}(\Omega) \hookrightarrow W^{1,p}_\omega(\Omega)$, requires more work and some preliminary results.

We start with the following lemma, which constructs test functions whose gradient has (Euclidean) norm equal to the standard mollifier. From hereafter, by \emph{standard mollifier} we will be referring to the family of functions $( J_\epsilon)_{\epsilon>0}$ given by $J_\epsilon(x) := \epsilon^{-N}J(x/\epsilon)$ for $x\in\mathbb{R}^N$ and $\epsilon>0$, where $J \in \mathcal{C}_c^\infty(\mathbb{R}^N)$ is nonnegative, radially symmetric and such that

\begin{gather}
    J(x)=0 \; \mbox{ for } |x|\ge 1, \mbox{ and }\int_{\mathbb{R}^N} J\mathrm{d}x=1.
\end{gather}

Additionally $-$and for technical reasons$-$, we will ask that $J$ attains its supremum (i.e.~maximum) at $x=0$, i.e.,

\begin{equation} \label{MollifierTechSup}
    \sup_{x\in\mathbb{R}^N} J(x) = \max_{x\in\mathbb{R}^N} J(x) = J(0).
\end{equation}

The usual example $J(x)=A \mathcal{X}_{B_1(0)}(x) \exp(1/(|x|^2-1))$ (with $A>0$ ensuring that the integral condition is satisfied) satisfies all the properties described above.

\begin{lemma} \label{Ham_JacLemma}
    Let $\epsilon>0$ and $f:[0,+\infty) \to \mathbb{R}$ be the (smooth) profile of $J$ (i.e., the function such that $J(x)=f(|x|)$ for all $x\in\mathbb{R}^N$), with $J$ such that $J_\epsilon(x) = \epsilon^{-N}J(x/\epsilon)$ for all $x\in\mathbb{R}^N$, where $J_\epsilon$ is the standard mollifier. Then, if $F(z):=\int_0^z f(t) \mathrm{d}t$ for all $z\in\mathbb{R}$, the function
    
    \begin{equation}
        h_\epsilon(x):=
        \epsilon^{-N+1}\left( F(1)-F(\lvert x \rvert/\epsilon)\right) \quad \mbox{ for all } x \in \mathbb{R}^N
    \end{equation}
    belongs to $W^{1,p}(\mathbb{R}^N)$ for all $p\in [1,+\infty]$, and is such that

    \begin{equation} \label{Ham_JacCond}
        |\nabla h_\epsilon(x)| = J_\epsilon(x) \text{ for almost every } x \in \mathbb{R}^N.
    \end{equation}
\end{lemma}

\begin{proof}
    Notice first that $h_\epsilon$ vanishes outside of $B_\epsilon(0)$, as $F(z) = F(1)$ for $z>1$, since $f(z)=0$ for $z>1$. Also, as $f$ is nonnegative (because $J$ itself is nonnegative), $F$ is increasing in $[0,+\infty)$, and thus

    \begin{equation} \label{h_epsBound}
        |h_\epsilon(x)| \le \epsilon^{-N+1}(F(1)-F(0)) = \epsilon^{-N+1} F(1) \; \mbox{ for all } x \in \mathbb{R}^N,
    \end{equation}
    from which we deduce that $h_\epsilon\in L^p(\mathbb{R}^N)$ for all $p\in[1,+\infty]$. Moreover

    \begin{equation}
        \nabla h_\epsilon(x) = -\epsilon^{-N} f(|x|/\epsilon) (x/|x|) = -\epsilon^{-N} J(x/\epsilon) (x/|x|) = - J_\epsilon(x) (x/|x|),
    \end{equation}
    for $x\in \mathbb{R}^N\setminus \{0\}$, from which condition \eqref{Ham_JacCond} follows. Using this condition and the fact that $J_\epsilon\in \mathcal{C}_c^\infty(\Omega)$, it follows that $h_\epsilon \in W^{1,p}(\mathbb{R}^N)$ for all $p\in [1,+\infty]$.
\end{proof}

Once the $h_\epsilon$ functions are defined, we establish the limits of its nonlocal Sobolev seminorms. For the forthcoming lemma, let us point out the following:

\begin{remark} \label{TechMollRemark}
    Given $p\in[1, +\infty)$, for $\epsilon>0$, if we consider the constant
    
    \begin{equation}
        C_{p,\epsilon}:= \epsilon^{N(p-1)/p} \left( \int_{B_1(0)} J^p \mathrm{d}x \right)^{-1/p} >0,
    \end{equation}
    (note that $C_{1,\epsilon}=1$ for all $\epsilon>0$) then $\hat{J}_{p,\epsilon} := C_{p,\epsilon}^p J_\epsilon^p$ is a function of $\mathcal{C}_c^\infty(\mathbb{R}^N)$ that vanishes outside of $B_\epsilon(0)$ and satisfies

    \begin{equation}
        \int_{\mathbb{R}^N} \hat J_{p,\epsilon} \mathrm{d} x = \int_{B_\epsilon(0)} \hat J_{p,\epsilon} \mathrm{d} x = 1.
    \end{equation}
\end{remark}

\begin{lemma} \label{P1TechLemmaNecCond}
     Let $\omega$ be an admissible weight function in $\Omega$ and $p\in[1,+\infty)$. Fix $x_0 \in \Omega$ such that

    \begin{equation}
        f_\omega^p(x_0) < +\infty,
    \end{equation}
    where
    \begin{equation}
        f_\omega^p(x)=\int_\Omega \omega(x,y)^p \mathrm{d}y,
    \end{equation}
    and let $\epsilon_0>0$ be such that $B_{\epsilon_0}(x_0)\subset\Omega$. For $\epsilon\in(0,\epsilon_0)$, define $h_{p,\epsilon}^{x_0}(x):=C_{p,\epsilon} h_\epsilon(x-x_0)$ for all $x\in\Omega$, where $h_\epsilon$ is given by Lemma \ref{Ham_JacLemma} and $C_{p,\epsilon}$ is the normalization constant as defined in Remark \ref{TechMollRemark}. Then the following limit holds:

    \begin{equation} \label{P1SuffCondAux1}
        \int_{\Omega\times\Omega} |\nabla_\omega h_{p,\epsilon}^{x_0}|^p \mathrm{d}x\mathrm{d}y \xrightarrow[\epsilon \to 0]{} 2f_\omega^p(x_0).
    \end{equation}

    In particular, if additionally $\Omega$ has finite measure, for any $x_0 \in \Omega$ we have
    \begin{equation}\label{P1SuffCondAux2}
        \int_{\Omega\times\Omega} |\nabla_1 h_{p,\epsilon}^{x_0}|^p \mathrm{d}x\mathrm{d}y \xrightarrow[\epsilon \to 0]{} 2|\Omega|.
    \end{equation}
\end{lemma}

\begin{proof}
    To start off, by Lemma \ref{Ham_JacLemma}, for almost every $x\in \Omega$ we have

    \begin{equation} \label{NormOfNablahpe}
        |\nabla h_{p,\epsilon}^{x_0}(x)|^p = C_{p,\epsilon}^p |\nabla h_\epsilon(x-x_0)|^p = C_{p,\epsilon}^p J_\epsilon^p(x-x_0) = \hat J_{p,\epsilon}(x-x_0),
    \end{equation}
    where $\hat J_{p,\epsilon}$ is the function introduced in Remark \ref{TechMollRemark}, which recall that belongs to $\mathcal{C}_c^\infty(\mathbb{R}^N)$, vanishes outside of $B_\epsilon(0)$ and satisfies
    
    \begin{equation} \label{TechMollRemarkEq1}
        \int_{\mathbb{R}^N} \hat J_{p,\epsilon} \mathrm{d} x = \int_{B_\epsilon(0)} \hat J_{p,\epsilon} \mathrm{d} x = 1.
    \end{equation}
    
    With this in mind, by Remark \ref{1impliesomegaRemark} and Proposition \ref{WeakDerImplyNlWeakDer} we deduce that

    \begin{equation}
        \nabla_\omega h_{p,\epsilon}^{x_0}(x,y) = \omega(x,y) \left(\nabla h_{p,\epsilon}^{x_0}(y) - \nabla h_{p,\epsilon}^{x_0} (x)\right) \mbox{ for almost every } (x,y) \in \Omega\times\Omega,
    \end{equation}
    and thus, using that $\nabla h_{p,\epsilon}^{x_0}$ vanishes outside of $B_\epsilon(x_0)$ and the computation \eqref{NormOfNablahpe}, we have

    \begin{align}
        \int_{\Omega\times\Omega} |\nabla_\omega h_{p,\epsilon}^{x_0}|^p \mathrm{d}x\mathrm{d}y &= \int_{\Omega\times\Omega} \omega(x,y)^p |\nabla h_{p,\epsilon}^{x_0}(y)-\nabla h_{p,\epsilon}^{x_0}(x)|^p \mathrm{d}x\mathrm{d}y\\
        &=\int_{B_\epsilon(x_0)} \int_{B_\epsilon(x_0)} \omega(x,y)^p |\nabla h_{p,\epsilon}^{x_0}(y)-\nabla h_{p,\epsilon}^{x_0}(x)|^p \mathrm{d}x\mathrm{d}y\\
        &\qquad +2\int_{\Omega \backslash B_\epsilon(x_0)} \int_{B_\epsilon(x_0)} \omega(x,y)^p \left| \nabla h_{p,\epsilon}^{x_0}(x) \right|^p \mathrm{d}x\mathrm{d}y\\
        &=\int_{B_\epsilon(x_0)} \int_{B_\epsilon(x_0)} \omega(x,y)^p |\nabla h_{p,\epsilon}^{x_0}(y)-\nabla h_{p,\epsilon}^{x_0}(x)|^p \mathrm{d}x\mathrm{d}y\\
        &\qquad +2\int_{\Omega \backslash B_\epsilon(x_0)} \int_{B_\epsilon(x_0)} \omega(x,y)^p \hat{J}_{p,\epsilon}(x-x_0) \mathrm{d}x\mathrm{d}y\label{P1SuffCondAux4}\\
        &=: I_1^\epsilon+2I_2^\epsilon.
    \end{align}

    Let us check that

    \begin{equation} \label{IepsilonsLimits}
        I_1^\epsilon \xrightarrow[\epsilon \to 0]{} 0, \quad \text{ and } \quad I_2^\epsilon \xrightarrow[\epsilon \to 0]{} f_\omega^p(x_0),
    \end{equation}
    from which \eqref{P1SuffCondAux1} will follow.

    Firstly, notice that, by using \eqref{NormOfNablahpe}, for almost every $(x,y)\in\Omega\times\Omega$ we have

    \begin{align}
        0 \le |\nabla h_{p,\epsilon}^{x_0}(y)-\nabla h_{p,\epsilon}^{x_0}(x)|{^p} &\le 2^{p-1} (|\nabla h_{p,\epsilon}^{x_0}(y)|{^p}+|\nabla h_{p,\epsilon}^{x_0}(x)|{^p})\\
        &= 2^{p-1} (\hat{J}_{p,\epsilon}(y-x_0)+\hat{J}_{p,\epsilon}(x-x_0)),
    \end{align}
    and by integrating in $x,y$ over $B_\epsilon(x_0)\times B_\epsilon(x_0)$ we get, using \eqref{TechMollRemarkEq1}, that

    \begin{equation}
        0 \le \int_{B_\epsilon(x_0)\times B_\epsilon(x_0)} |\nabla h_{p,\epsilon}^{x_0}(y)-\nabla h_{p,\epsilon}^{x_0}(x)|{^p} \mathrm{d}x\mathrm{d}y \le 2^p |B_\epsilon^{x_0}|.
    \end{equation}

    With this, as $\omega$ is bounded close to $(x_0,x_0)$,

    \begin{align}
        I_1^\epsilon &=\int_{B_\epsilon(x_0)\times B_\epsilon(x_0)} \omega(x,y)^p |\nabla h_{p,\epsilon}^{x_0}(y)-\nabla h_{p,\epsilon}^{x_0}(x)|^p \mathrm{d}x\mathrm{d}y\\
        &\le \|\omega\|_{L^\infty(B_\epsilon(x_0)\times B_\epsilon(x_0))}^p \int_{B_\epsilon(x_0)\times B_\epsilon(x_0)} |\nabla h_{p,\epsilon}^{x_0}(y)-\nabla h_{p,\epsilon}^{x_0}(x)|^p \mathrm{d}x\mathrm{d}y\\
        &\le 2^p \|\omega\|_{L^\infty(B_\epsilon(x_0)\times B_\epsilon(x_0))}^p |B_\epsilon^{x_0}| \xrightarrow[\epsilon \to 0]{} 0,
    \end{align}
    which proves the first limit of \eqref{IepsilonsLimits}.

    On the other hand, notice that

    \begin{equation}
        I_2^\epsilon = \int_\Omega g_\epsilon^p(y) \mathrm{d}y,
    \end{equation}
    where

    \begin{equation} \label{P1SuffCondgDef}
        g_\epsilon^p(y) := \mathcal{X}_{\Omega\backslash B_\epsilon(x_0)}(y) \int_{B_\epsilon(x_0)} \hat{J}_{p,\epsilon}(x-x_0) \omega(x,y)^p \mathrm{d}x.
    \end{equation}

    We claim that

    \begin{equation} \label{P1SuffCondAux6}
        g_\epsilon^p \xrightarrow[\epsilon \to 0]{} \omega(x_0,\cdot)^p \; \text{ point-wise a.e. in } \; \Omega.
    \end{equation}

    Indeed, fix $y\in\Omega$ such that $y \neq x_0$. For $\delta>0$, as $\omega$ is continuous (at $(x_0,y)$), we can choose $\epsilon>0$ small enough such that $y \not \in B_\epsilon(x_0)$ and

    \begin{align}
        \bigg|\omega(x_0,y)^p- \int_{B_\epsilon(x_0)} \hat{J}_{p,\epsilon}(x-x_0) \omega(x,y)^p \mathrm{d}x \bigg|
        &= \left|\int_{B_\epsilon(x_0)} \left[\hat{J}_{p,\epsilon}(x-x_0) (\omega(x_0,y)^p-\omega(x,y)^p)\right]\mathrm{d}x\right|\\
        &\le \|\omega(x_0,y)^p-\omega(\cdot,y)^p\|_{L^\infty(B_\epsilon(x_0))} < \delta,
    \end{align}
    where we have used \eqref{TechMollRemarkEq1} in both the equality and the inequality, and thus for such an $\epsilon>0$

    \begin{equation}
        \left| \omega(x_0,y)^p - \mathcal{X}_{\Omega\backslash B_\epsilon(x_0)}(y) \int_\Omega \left[ \hat{J}_{p,\epsilon}(x-x_0) \omega(x,y)^p \right]\mathrm{d}x \right| < \delta,
    \end{equation}
    from which \eqref{P1SuffCondAux6} is proven.

    Now, recall that we are assuming that

    $$\int_\Omega \omega(x_0,y)^p \mathrm{d}y < +\infty,$$
    and thus, by continuity of $\omega$, we can find a fixed $\epsilon'>0$ small enough for which

    $$\int_\Omega \|\omega(\cdot,y)^p\|_{L^\infty(B_{\epsilon'}(x_0))}\mathrm{d}y < +\infty.$$

    Moreover, for $0< \epsilon \le \epsilon'$ and $y \in \Omega$, using \eqref{TechMollRemarkEq1}, we have

    $$|g_\epsilon^p(y)| \le \int_\Omega \hat{J}_{p,\epsilon}(x-x_0) \omega(x,y)^p \mathrm{d}x \le \|\omega(\cdot,y)^p\|_{L^\infty(B_\epsilon(x_0))}\le \|\omega(\cdot,y)^p\|_{L^\infty(B_{\epsilon'}(x_0))}.$$
    
    Therefore, the function $\alpha: \Omega \to \mathbb{R}$ given by $\alpha(y):=\|\omega(\cdot,y)^p\|_{L^\infty(B_{\epsilon'}(x_0))}$ for all $y \in \Omega$ is an integrable function that dominates $g_\epsilon^p$ for $\epsilon$ sufficiently small, and, recalling \eqref{P1SuffCondAux6}, we may use the dominated convergence theorem to get that

    \begin{equation} \label{P1SuffCondAux5}
        \lim_{\epsilon \to 0} I_2^\epsilon = \lim_{\epsilon \to 0} \int_\Omega g_\epsilon^p(x)\mathrm{d}x = \int_\Omega \omega(x_0,y)^p \mathrm{d}y = f_\omega^p(x_0),
    \end{equation}
    which establishes the second limit of \eqref{IepsilonsLimits}. Hence the result is proven.
\end{proof}

With the lemmas above, we may prove the following result, which states that if we have continuous embeddings between $W^{1,p}_1(\Omega)$ and $W^{1,p}_\omega(\Omega)$ ``seminorm wise", then either $f_\omega^p$ is bounded away from zero or away from infinity, depending on the embedding.

\begin{prop} \label{NecCondProp}
    Let $\Omega$ be of finite measure, $\omega$ be an admissible weight function and $p\in[1,+\infty)$. Then the following statements hold:

    \begin{itemize}
        \item[(i)] If there exists a constant $c>0$ such that

        \begin{equation} \label{P1SuffCondAux7}
            \int_{\Omega\times\Omega} |\nabla_1 u|^p \mathrm{d}x\mathrm{d}y \le c \int_{\Omega\times\Omega} |\nabla_\omega u|^p \mathrm{d}x\mathrm{d}y
        \end{equation}
        for all $u \in W^{1,p}_\omega(\Omega)$, then there exists a constant $k>0$ such that

        \begin{equation}
            k \le f_\omega^p(x)
        \end{equation}
        for all $x \in \Omega$.

        \item[(ii)] If there exists a constant $C>0$ such that
        
        \begin{equation} \label{P1SuffCondAux3}
            \int_{\Omega\times\Omega} |\nabla_\omega u|^p \mathrm{d}x\mathrm{d}y \le C \int_{\Omega\times\Omega} |\nabla_1 u|^p \mathrm{d}x\mathrm{d}y
        \end{equation}
        for all $u \in W^{1,p}_1(\Omega)$, then there exists a constant $K>0$ such that

        \begin{equation}
            f_\omega^p(x) \le K
        \end{equation}
        for all $x \in \Omega$.
    \end{itemize}
\end{prop}

\begin{proof}
    We start with (i). Let $x_0 \in \Omega$. We assume that $x_0$ is such that
    $$\int_\Omega \omega(x_0,y)^p \mathrm{d}y < +\infty,$$
    as any constant $k>0$ will work otherwise. Let $\epsilon_0>0$ be such that $B_{\epsilon_0}(x_0)\subset\Omega$ and consider for $\epsilon \in (0,\epsilon_0)$ the $h_{p,\epsilon}^{x_0}$ function introduced in Lemma \ref{P1TechLemmaNecCond}. The limit \eqref{P1SuffCondAux1} of the same lemma shows that $h_{p,\epsilon}^{x_0}\in W^{1,p}_\omega(\Omega)$ for sufficently small $\epsilon \in (0, \epsilon_0)$ (since we already have $h_{p,\epsilon}^{x_0} \in L^p(\Omega)$ by Lemma \ref{Ham_JacLemma}). Thus, for such an $\epsilon \in (0, \epsilon_0)$, by substituting $u$ in \eqref{P1SuffCondAux7} by $h_{p,\epsilon}^{x_0}$ and letting $\epsilon$ go to zero, we deduce using \eqref{P1SuffCondAux1} and \eqref{P1SuffCondAux2} of Lemma \ref{P1TechLemmaNecCond} that

    \begin{equation}
        2|\Omega| \le 2c f_\omega^p(x_0),
    \end{equation}
    from which, by choosing $k=|\Omega|/c$, statement (i) is proven.

    We now move onto (ii). Fix $x_0 \in \Omega$, let $\epsilon_0>0$ be such that $B_{\epsilon_0}(x_0)\subset \Omega$, and for $\epsilon\in(0,\epsilon_0)$ consider the $h_{p,\epsilon}^{x_0}$ function introduced in Lemma \ref{P1TechLemmaNecCond} and the function $g_\epsilon^p$ as defined in \eqref{P1SuffCondgDef}, which we know $-$\eqref{P1SuffCondAux6}$-$ converges pointwise a.e. in $\Omega$ to $\omega(x_0,\cdot)^p$ as $\epsilon\to 0$. As $h_{p,\epsilon}^{x_0} \in W^{1,p}_1(\Omega)$ for any $\epsilon\in(0,\epsilon_0)$ due to Lemma \ref{Ham_JacLemma} (and Theorem \ref{Theorem_equality_constant_weight}), we have by \eqref{P1SuffCondAux3} and the previous computation \eqref{P1SuffCondAux4} that

    \begin{align}
        2\int_\Omega g_\epsilon^p(y)\mathrm{d}y &= 2\int_{\Omega\backslash B_\epsilon(x_0)} \int_{B_\epsilon(x_0)} \hat{J}_{p,\epsilon}(x-x_0) \omega(x,y)^p \mathrm{d}x\mathrm{d}y\\
        &\le \int_{\Omega\times\Omega} |\nabla_\omega h_{p,\epsilon}^{x_0}|^p\mathrm{d}x\mathrm{d}y \le C \int_{\Omega\times\Omega} |\nabla_1 h_{p,\epsilon}^{x_0}|^p \mathrm{d}x\mathrm{d}y
    \end{align}
    for all $\epsilon\in(0,\epsilon_0)$. From which, by  Fatou's lemma and \eqref{P1SuffCondAux2}

    \begin{equation}
        f_\omega^p(x_0)=\int_\Omega \omega(x_0,y)^p \mathrm{d}y \le \liminf_{\epsilon \to 0} \int_\Omega g_\epsilon^p(y)\mathrm{d}y \le \liminf_{\epsilon \to 0} \frac{C}{2} \int_{\Omega\times\Omega} |\nabla_1 h_{p,\epsilon}^{x_0}|^p \mathrm{d}x\mathrm{d}y = C |\Omega|,
    \end{equation}
    hence, by choosing $K=C|\Omega|$, statement (ii) is proven.
\end{proof}

The proposition above is the final ingredient needed to prove the necessary condition statement of Theorem \ref{Theorem_equality_Sobolev}. The proof of the whole theorem is given below.

\begin{proof}[Proof of Theorem \ref{Theorem_equality_Sobolev}]
    Let us start with the first statement, that is:
    \begin{equation} \label{W1pProb1ProofAux1}
        \mbox{If there exists } k>0 \mbox{ such that } k \le \omega \mbox{ in } \Omega\times\Omega, \mbox{ then } W^{1,p}_\omega(\Omega) \hookrightarrow W^{1,p}(\Omega),
    \end{equation}
    for all $p\in[1,+\infty]$ and $\Omega$ with finite measure if $p\in[1,+\infty)$, or general $\Omega$ if $p=+\infty$.   
    
    We will prove the $p\in[1,+\infty)$ case, as the $p=+\infty$ case follows by an analogous argument. By Remark \ref{omegaimplies1Remark} we have that any $u \in W^{1,p}_\omega(\Omega)$ satisfies $\nabla_\omega u = \omega \nabla_1 u$ a.e. in $\Omega\times\Omega$. Therefore,

    \begin{equation}
        \int_\Omega |\nabla_1 u|^p \mathrm{d}x\mathrm{d}y \le \frac{1}{k^p}  \int_\Omega |\nabla_\omega u|^p \mathrm{d}x\mathrm{d}y,
    \end{equation}
    from which we deduce that $W^{1,p}_\omega(\Omega) \hookrightarrow W^{1,p}_1(\Omega)$, and this coupled with Theorem \ref{Theorem_equality_constant_weight} yields \eqref{W1pProb1ProofAux1}.

    We move onto proving the right to left implication of the second statement of the theorem, i.e.:

    \begin{equation} \label{W1pProb1ProofAux2}
        \mbox{If there exists } K>0 \mbox{ such that } f_\omega^p \le K \mbox{ in } \Omega, \mbox{ then } W^{1,p}(\Omega) \hookrightarrow W^{1,p}_\omega(\Omega),
    \end{equation}
    again, for all $p\in[1,+\infty]$ and $\Omega$ with finite measure if $p\in[1,+\infty)$, or general $\Omega$ if $p=+\infty$.

    For the $p\in[1,+\infty)$ case, by Remark \ref{1impliesomegaRemark} and Proposition \ref{WeakDerImplyNlWeakDer} we have that for any $u\in W^{1,p}(\Omega)$, $\nabla_\omega u(x,y) = \omega(x,y) \left(\nabla u(y) - \nabla u(x)\right)$ for almost every $(x,y) \in \Omega\times\Omega$, and thus

    \begin{align}
        \int_{\Omega\times\Omega} |\nabla_\omega u(x,y)|^p \mathrm{d}x\mathrm{d}y &= \int_{\Omega\times\Omega} \omega(x,y)^p \left| \nabla u(y) - \nabla u(x) \right|^p \mathrm{d}x\mathrm{d}y\\
        &\le 2^p \int_\Omega |\nabla u(x)|^p \left(\int_\Omega \omega(x,y)^p \mathrm{d}y\right) \mathrm{d}x\\
        &= 2^p \int_\Omega |\nabla u(x)|^p f_\omega^p(x) \mathrm{d}x\\
        &\le 2^p K \int_\Omega |\nabla u(x)|^p \mathrm{d}x,
    \end{align}
    which proves \eqref{W1pProb1ProofAux2} for $p \in [1,+\infty)$.

    For the $p=+\infty$ case, notice that a constant $K>0$ existing such that $f_\omega^\infty(x)\le K$ for $x\in \Omega$ is equivalent to a constant $K'>0$ existing such that

    \begin{equation}
        \omega(x,y) \le K' \mbox{ for } (x,y) \in \Omega\times\Omega,
    \end{equation}
    which, coupled with the fact that for $u\in W^{1,\infty}(\Omega)$ we have $\nabla_\omega u(x,y)=\omega(x,y) \left(\nabla u(y) - \nabla u(x)\right)$ for almost every $(x,y)\in\Omega\times\Omega$, proves \eqref{W1pProb1ProofAux2} for $p=+\infty$. Indeed, for almost every $(x,y)\in\Omega\times\Omega$ we have

    \begin{align}
        |\nabla_\omega u(x,y)| &\le 2 \omega(x,y) \|\nabla u\|_{L^\infty(\Omega;\mathbb{R}^N)} \le 2 K' \|\nabla u\|_{L^\infty(\Omega;\mathbb{R}^N)},
    \end{align}
    and by taking essential supremum over $(x,y)$ the embedding is deduced.

    Finally we prove the following:

    \begin{equation} \label{W1pProb1ProofAux3}
        \mbox{If } W^{1,p}(\Omega) \hookrightarrow W^{1,p}_\omega(\Omega), \mbox{ then there exists } K>0 \mbox{ such that } f_\omega^p\le K \mbox{ in } \Omega,
    \end{equation}
    for all $p\in[1,+\infty]$ and $\Omega$ with finite measure if $p\in[1,+\infty)$, or general $\Omega$ if $p=+\infty$.

    For the $p\in[1,+\infty)$ case, notice that $W^{1,p}_1(\Omega) \hookrightarrow W^{1,p}_\omega(\Omega)$ implies that there exists $C>0$ for which

    \begin{equation} \label{P1SuffCondAux9}
        \int_{\Omega\times\Omega} |\nabla_\omega u|^p \mathrm{d}x\mathrm{d}y \le C\left( \int_\Omega |u|^p \mathrm{d}x + \int_{\Omega\times\Omega} |\nabla_1 u|^p \mathrm{d}x\mathrm{d}y \right)
    \end{equation}
    for all $u\in W^{1,p}_1(\Omega)$. Fix $x_0 \in \Omega$, let $\epsilon_0>0$ be such that $B_{\epsilon_0}(x_0)\subset \Omega$, and for $\epsilon\in(0,\epsilon_0)$ consider the $h_{p,\epsilon}^{x_0}$ function introduced in Lemma \ref{P1TechLemmaNecCond} and the function $g_\epsilon^p$ defined in \eqref{P1SuffCondgDef}, which we know converges pointwise a.e. in $\Omega$ to $\omega(x_0,\cdot)^p$ as $\epsilon\to 0$. As $h_{p,\epsilon}^{x_0} \in W^{1,p}_1(\Omega)$ for any $\epsilon\in(0,\epsilon_0)$ (again, by Lemma \ref{Ham_JacLemma} and Theorem \ref{Theorem_equality_constant_weight}), we have by \eqref{P1SuffCondAux9} and the previous computation \eqref{P1SuffCondAux4} that

    \begin{align} \label{gChainOfInequalities}
        2\int_\Omega g_\epsilon^p(y)\mathrm{d}y &= 2\int_{\Omega\backslash B_\epsilon(x_0)} \int_{B_\epsilon(x_0)} \hat{J}_{p,\epsilon}(x-x_0) \omega(x,y)^p \mathrm{d}x\mathrm{d}y\nonumber\\
        &\le \int_{\Omega\times\Omega} |\nabla_\omega h_{p,\epsilon}^{x_0}|^p\mathrm{d}x\mathrm{d}y\nonumber\\
        &\le C \left( \int_\Omega |h_{p,\epsilon}^{x_0}|^p \mathrm{d}x + \int_{\Omega\times\Omega} |\nabla_1 h_{p,\epsilon}^{x_0}|^p \mathrm{d}x\mathrm{d}y\right)
    \end{align}
    for all $\epsilon\in(0,\epsilon_0)$. Now, we claim that

    \begin{equation} \label{h_epsilon_lp_to_zero}
        \lim_{\epsilon \to 0} \int_\Omega |h_{p,\epsilon}^{x_0}(x)|^p \mathrm{d}x = 0.
    \end{equation}
    
    Indeed, by recalling the expression of $h_\epsilon$ and $C_{p,\epsilon}$ (see Lemma \ref{Ham_JacLemma} and Remark \ref{TechMollRemark}) and applying the change of variables $\tilde x=(x-x_0)/\epsilon$ we have

    \begin{align}
        \int_\Omega |h_{p,\epsilon}^{x_0}(x)|^p \mathrm{d}x &= C_{p,\epsilon}^p \int_{B_\epsilon(x_0)} |h_\epsilon(x-x_0)|^p \mathrm{d}x\\
        &= \left( \int_\Omega J(z)^p \mathrm{d}z \right)^{-1} \epsilon^{N(p-1)} \int_{B_\epsilon(x_0)} \epsilon^{p(1-N)} \left|F(1)-F(|x-x_0|/\epsilon)\right|^p \mathrm{d}x\\
        &= \left( \int_\Omega J(z)^p \mathrm{d}z \right)^{-1} \epsilon^{p-N} \int_{B_1(0)} |F(1)-F(|\tilde x|)|^p \epsilon^{N} \mathrm{d}\tilde x\\
        &= \left[ \left( \int_\Omega J(z)^p \mathrm{d}z \right)^{-1} \int_{B_1(0)} |F(1)-F(|\tilde x|)|^p \mathrm{d}\tilde x \right] \epsilon^p, \label{hepsx0LpNorm}
    \end{align}
    from which, as $p\ge 1$, the limit \eqref{h_epsilon_lp_to_zero} holds.
    
    With this in mind, continuing from \eqref{gChainOfInequalities} and by Fatou's lemma

    \begin{align}
        f_\omega^p(x_0)=\int_\Omega \omega(x_0,y)^p \mathrm{d}y &\le \liminf_{\epsilon \to 0} \int_\Omega g_\epsilon^p(y)\mathrm{d}y\\
        &\le \liminf_{\epsilon \to 0} \frac{C}{2} \left(\int_\Omega |h_{p,\epsilon}^{x_0}|^p \mathrm{d}x + \int_{\Omega\times\Omega} |\nabla_1 h_{p,\epsilon}^{x_0}|^p \mathrm{d}x\mathrm{d}y \right)= C |\Omega|,
    \end{align}
    where we have used the \eqref{P1SuffCondAux2} and \eqref{h_epsilon_lp_to_zero} limits. Thus \eqref{W1pProb1ProofAux3} is proven for $p\in[1,+\infty)$.

    The last thing left to prove is the $p=+\infty$ case. By Remark \ref{1impliesomegaRemark}, $W^{1,\infty}(\Omega) \hookrightarrow W^{1,\infty}_\omega(\Omega)$ implies that there exists $C>0$ such that

    \begin{equation}
        \left\| \omega \nabla_1 u \right\|_{L^\infty(\Omega\times\Omega;\mathbb{R}^N)} \le C \left( \|u\|_{L^\infty(\Omega)} + \|\nabla u\|_{L^\infty(\Omega;\mathbb{R}^N)} \right)
    \end{equation}
    for all $u \in W^{1,\infty}(\Omega)$. This means that for any $u \in W^{1,\infty}(\Omega)$,
    
    \begin{equation} \label{W1pProb1ProofAux4}
        \omega(x,y) |\nabla_1 u(x,y)| \le C \left( \|u\|_{L^\infty(\Omega)} + \|\nabla u\|_{L^\infty(\Omega;\mathbb{R}^N)} \right) \quad \mbox{ for almost every } (x,y)\in\Omega\times\Omega.
    \end{equation}
    
    Now, fix $x_0 \in \Omega$ and $\epsilon_0>0$ such that $B_{\epsilon_0}(x_0)\subset \Omega$. For each $\epsilon\in(0,\epsilon_0)$ consider

    \begin{equation} \label{RhoDef}
        \rho_\epsilon^{x_0}:= \frac{\epsilon^N}{J(0)} h_{1,\epsilon}^{x_0},
    \end{equation}
    where $h_{1,\epsilon}^{x_0}$ is the function introduced in Lemma \ref{P1TechLemmaNecCond} (with $p=1$) and $J$ is the function such that $J_\epsilon(x)=(1/\epsilon)^N J(x/\epsilon)$. Recall $-$\eqref{MollifierTechSup}$-$ that we are assuming that the maximum of $J(x)$ is attained at $x=0$. Notice that then, for $\epsilon\in(0,\epsilon_0)$ we have
    
    \begin{equation}
        |\nabla \rho_\epsilon^{x_0}(x)| = \frac{\epsilon^N}{J(0)} |\nabla h_{1,\epsilon}^{x_0}(x)| = \frac{\epsilon^N}{J(0)} J_\epsilon(x-x_0) = \frac{\epsilon^N}{J(0)} \frac{J\left((x-x_0/\epsilon)\right)}{\epsilon^N} \le \frac{J(0)}{J(0)} = 1
    \end{equation}
    for almost every $x\in \Omega$. Thus, as this bound is attained as $x \xrightarrow[]{} x_0$,

    \begin{equation} \label{Problem1InftyAux7}
        \|\nabla \rho_\epsilon^{x_0}\|_{L^\infty(\Omega;\mathbb{R}^N)} = 1 \quad \mbox{ for all } \epsilon\in(0,\epsilon_0).
    \end{equation}
    
    Moreover, for $\epsilon\in(0,\epsilon_0)$ and $x\in B_\epsilon(x_0)$, by recalling \eqref{h_epsBound} it follows that

    \begin{equation} \label{Problem1InftyAux6}
        |\rho_\epsilon^{x_0}(x)| = \frac{\epsilon^N}{J(0)} \left| h_{1,\epsilon}^{x_0}(x)\right| \le \frac{\epsilon^N}{J(0)} \epsilon^{1-N} F(1) = \frac{\epsilon}{J(0)} F(1),
    \end{equation}
    and since this bound is attained whenever $x=x_0$, we deduce that

    \begin{equation} \label{Problem1InftyAux8}
        \|\rho_\epsilon^{x_0}\|_{L^\infty(\Omega)} = \frac{F(1)}{J(0)} \epsilon,
    \end{equation}
    which vanishes as $\epsilon\to 0$. In particular $\rho_\epsilon^{x_0} \in W^{1,\infty}(\Omega)$ for any $\epsilon\in(0,\epsilon_0)$ (a fact which can also be derived using Lemma \ref{Ham_JacLemma}). With this in mind, we plug $u=\rho_\epsilon^{x_0}$ with $\epsilon\in(0,\epsilon_0)$ in \eqref{W1pProb1ProofAux4}, which yields that

    \begin{equation}
        \omega(x,y) |\nabla_1 \rho_\epsilon^{x_0} (x,y)| \le C \left( \|\rho_\epsilon^{x_0}\|_{L^\infty(\Omega)} + \|\nabla \rho_\epsilon^{x_0}\|_{L^\infty(\Omega;\mathbb{R}^N)} \right) = C \left( \frac{F(1)}{J(0)} \epsilon + 1 \right),
    \end{equation}
    for almost every $(x,y)\in \Omega\times\Omega$, where we have used \eqref{Problem1InftyAux7} and \eqref{Problem1InftyAux8}. Using the fact that, as argued before,
    
    \begin{equation}
        \lim_{x\to x_0} |\nabla \rho_\epsilon^{x_0}(x)|=1,
    \end{equation}
    in particular, for almost every $y\in\Omega\setminus B_\epsilon(x_0)$ we have

    \begin{align}
        \omega(x_0,y) &= \lim_{x\to x_0} \omega(x,y) |\nabla \rho_\epsilon^{x_0}(x)| = \lim_{x\to x_0} \omega(x,y) |\nabla_1 \rho_\epsilon^{x_0}(x,y)|\le  C \left( \frac{F(1)}{J(0)} \epsilon + 1 \right),
    \end{align}
    where in the second equality we have used that by Proposition \ref{WeakDerImplyNlWeakDer} we have that $\nabla_1\rho_\epsilon^{x_0}(x,y) = \nabla \rho_\epsilon^{x_0}(y) - \nabla \rho_\epsilon^{x_0}(x)$ for almost every $(x,y)\in \Omega$, and thus, using the fact that $\rho_\epsilon^{x_0}$ vanishes outside of $B_\epsilon(x_0)$, $\nabla_1\rho_\epsilon^{x_0}(x,y) = \nabla \rho_\epsilon^{x_0}(x)$ for almost every $(x,y)\in\Omega\times\left( \Omega \setminus B_\epsilon(x_0)\right)$. Therefore, by taking $\epsilon$ to 0 in the above inequality we obtain that
    
    \begin{equation}
        \omega(x_0,y) \le C
    \end{equation}
    for almost every $y \in \Omega \setminus \{x_0\}$. By continuity of $\omega$ the inequality is satisfied everywhere in $\Omega$, and finally, by taking supremum over all $y \in \Omega$, \eqref{W1pProb1ProofAux3} is proven for $p=+\infty$.
\end{proof}

We move onto proving Theorem \ref{Theorem_main_equality_NLBV}. Again, the proofs to all of the statements are fairly straightforward, except for the one that gives a necessary condition for the embedding $\mathrm{BV}(\Omega)\hookrightarrow \mathrm{NLBV}_\omega(\Omega)$. Luckily, this last proof can be adapted from the Sobolev setting using Proposition \ref{C1NLTV}. The proof of the whole theorem is give below.

\begin{proof}[Proof of Theorem \ref{Theorem_main_equality_NLBV}]
    We start by proving the first statement:

    \begin{equation} \label{NLBVProb1ProofAux1}
        \mbox{If there exists } k>0 \mbox{ such that } k \le \omega \mbox{ in } \Omega\times\Omega, \mbox{ then } \mathrm{NLBV}_\omega(\Omega) \hookrightarrow \mathrm{BV}(\Omega),
    \end{equation}
    for any $\Omega$ of finite measure.

    Let $u \in \mathrm{NLBV}_\omega(\Omega)$. Given $\phi \in \mathcal{C}_c^1(\Omega\times\Omega;\mathbb{R}^N)$ such that $\|\phi\|_\infty \le 1$, notice that for all $(x,y)\in\Omega\times\Omega$

    \begin{equation}
        |k \phi(x,y)|\le k \le \omega(x,y). 
    \end{equation}

    With this in mind, by Lemma \ref{NLTV^omegaLemma}

    \begin{align}
        \int_\Omega u(x) \mathrm{div}_1(\phi)(x) \mathrm{d}x &= \frac{1}{k} \int_\Omega u(x) \mathrm{div}_1(k \phi)(x) \mathrm{d}x \le \frac{1}{k} \mathrm{NLTV}_\omega(u),
    \end{align}
    and by taking supremum over all $\phi$, we obtain the embedding $\mathrm{NLBV}_\omega(\Omega) \hookrightarrow \mathrm{NLBV}(\Omega)$. Coupling this fact with Proposition \ref{NLBV=BVEquivNorm} proves \eqref{NLBVProb1ProofAux1}.

    We move onto proving the following implication of the last statement:
    \begin{equation} \label{NLBVProb1ProofAux2}
        \mbox{If there exists } K>0 \mbox{ such that } f_\omega^1 \le K \mbox{ in } \Omega, \mbox{ then } \mathrm{BV}(\Omega) \hookrightarrow \mathrm{NLBV}_\omega(\Omega),
    \end{equation}
    for any $\Omega$ of finite measure.

    Take $u \in \mathrm{BV}(\Omega)$. For any $\phi \in \mathcal{C}_c^1(\Omega\times\Omega;\mathbb{R}^N)$ such that $|\phi|\le \omega$ in $\Omega\times\Omega$, we have that

    \begin{equation}
        \varphi_\phi^1(x):= \frac{1}{K}\int_\Omega \phi(y,x) \mathrm{d}y, \quad \varphi_\phi^2(x):= -\frac{1}{K}\int_\Omega \phi(x,y) \mathrm{d}y 
    \end{equation}
    are functions belonging to $\mathcal{C}_c^1(\Omega;\mathbb{R}^N)$. Moreover, using the fact that $\omega$ is symmetric and by hypothesis

    \begin{equation}
        |\varphi_\phi^1(x)| \le \frac{1}{K} \int_\Omega |\phi(y,x)| \mathrm{d}y\le \frac{1}{K} \int_\Omega \omega(y,x) \mathrm{d}y \le \frac{1}{K} K = 1
    \end{equation}
    for all $x\in\Omega$, from which $\|\varphi_\phi^1\|_\infty \le 1$. Analogously one checks that $\|\varphi_\phi^2\|_\infty \le 1$. Using this, we have the following inequality

    \begin{align}
        \int_\Omega u(x)\; \mathrm{div}_1(\phi)(x)\mathrm{d}x &= \int_\Omega u(x)\; \mathrm{div}_x \left[\int_\Omega (\phi(y,x) - \phi(x,y)) \mathrm{d}y\right] \mathrm{d}x\\
        &= K \int_\Omega u(x)\; \mathrm{div}_x(\varphi_\phi^1)(x) \mathrm{d}x + K \int_\Omega u(x)\; \mathrm{div}_x(\varphi_\phi^2)(x) \mathrm{d}x\\
        &\le 2 K \; \mathrm{TV}(u).
    \end{align}

    Therefore, by taking supremum over all $\phi$ we obtain

    \begin{equation}
        \mathrm{NLTV}_\omega(u)\le 2 K \; \mathrm{TV}(u),
    \end{equation}
    from which we deduce \eqref{NLBVProb1ProofAux2}.

    Finally, we prove the remaining statement:
    \begin{equation} \label{NLBVProb1ProofAux3}
        \mbox{If } \mathrm{BV}(\Omega) \hookrightarrow \mathrm{NLBV}_\omega(\Omega), \mbox{ then there exists } K>0 \mbox{ such that } f_\omega^1 \le K \mbox{ in } \Omega,
    \end{equation}
    for any $\Omega$ of finite measure.
    
    By Proposition \ref{NLBV=BVEquivNorm}, the embedding $\mathrm{BV}(\Omega) \hookrightarrow \mathrm{NLBV}_\omega(\Omega)$ is equivalent to $\mathrm{NLBV}(\Omega) \hookrightarrow \mathrm{NLBV}_\omega(\Omega)$, which in turn implies that there exists $C>0$ such that

    \begin{equation}
        \mathrm{NLTV}_\omega(u) \le C \left( \| u\|_{L^1(\Omega)} + \mathrm{NLTV}(u) \right)
    \end{equation}
    for all $u \in \mathrm{NLBV}(\Omega)$. Fix $x_0 \in \Omega$ and let $\epsilon_0>0$ be such that $B_{\epsilon_0}(x_0)\subset \Omega$. Consider for $\epsilon \in (0,\epsilon_0)$ the $h_{1,\epsilon}^{x_0}$ function introduced in Lemma \ref{P1TechLemmaNecCond} (with $p=1$). By Lemma \ref{Ham_JacLemma}, $h_{1,\epsilon}^{x_0}$ has (standard) weak first order derivatives in all directions. This means that, by Remark \ref{1impliesomegaRemark} and Proposition \ref{WeakDerImplyNlWeakDer}, $h_{1,\epsilon}^{x_0}$ has weak nonlocal partial derivatives induced by $1$ and also $\omega$ in all directions, from which, by Proposition \ref{C1NLTV}

    \begin{equation}
        \mathrm{NLTV}_\omega(h_{1,\epsilon}^{x_0}) =\int_{\Omega\times\Omega} |\nabla_\omega h_{1,\epsilon}^{x_0}| \mathrm{d}x\mathrm{d}y, \quad \mathrm{NLTV}_1(h_{1,\epsilon}^{x_0})=\int_{\Omega\times\Omega} |\nabla_1 h_{1,\epsilon}^{x_0}| \mathrm{d}x\mathrm{d}y.
    \end{equation}

    With this, we may proceed as in the proof of Theorem \ref{Theorem_equality_Sobolev}, in particular as in the implication \eqref{W1pProb1ProofAux3} with $p=1$, to obtain \eqref{NLBVProb1ProofAux3}.
\end{proof}

\begin{remark} \label{TechRemarkFiniteMeasureNotNec}
    Notice that in the proof of Theorems \ref{Theorem_main_equality_NLBV} and \ref{Theorem_equality_Sobolev} we implicitly prove that for $\Omega$ (with maybe infinite measure) and admissible weight function $\omega$ that is bounded away from zero in $\Omega\times\Omega$, we have $W^{1,p}_\omega(\Omega) \hookrightarrow W^{1,p}_1(\Omega)$ for $p\in[1,+\infty]$ and $\mathrm{NLBV}_\omega(\Omega) \hookrightarrow \mathrm{NLBV}(\Omega)$. The finiteness of the measure of $\Omega$ is only needed to use Theorem \ref{Theorem_equality_constant_weight} and Proposition \ref{NLBV=BVEquivNorm}.
\end{remark}

\begin{remark} \label{NotNecRemark}
    Let us show that \eqref{Muckenhoupt} is not a necessary condition for the embeddings $W^{1,p}_\omega\hookrightarrow W^{1,p}$ and $\mathrm{NLBV}_\omega \hookrightarrow \mathrm{BV}$ to hold.
    
    For the Sobolev case, let $\Omega$ be of finite measure and $p\in[1,+\infty)$. Fix $\theta \in \mathcal{C}_0^1(\Omega)$ such that $\theta >0$ in $\Omega$ and consider
    \begin{equation}
        \omega_\theta(x,y):=\theta(x)+\theta(y) \quad \mbox{ for all } (x,y)\in \Omega\times\Omega,
    \end{equation}
    which one readily checks that is an admissible weight function. 
    
    Clearly $\omega_\theta$ does not satisfy \eqref{Muckenhoupt}, but $W^{1,p}_{\omega_\theta}(\Omega) \hookrightarrow W^{1,p}(\Omega)$ does hold. Indeed, given $u \in W^{1,p}_{\omega_\theta}(\Omega)$, one has, by Remarks \ref{omegaimplies1Remark} and \ref{WeakDerInTermsOfNlWeakDerExtension}, that $u$ has local first weak derivatives in any direction and

    \begin{align}
        \int_\Omega |\nabla u(x)|^p \mathrm{d}x &= \int_\Omega \left| -\frac{1}{\int_\Omega \theta(z) \mathrm{d}z} \int_\Omega \left[ u(y) \nabla \theta(y) + \nabla_1 u(x,y) \theta(y) \right] \mathrm{d}y\right|^p \mathrm{d}x\\
        &\le \frac{1}{\left(\int_\Omega\theta(z) \mathrm{d}z\right)^p} \int_\Omega \left( \int_\Omega \left| u(y) \nabla \theta(y) + \nabla_1 u(x,y) \theta(y) \right| \mathrm{d}y\right)^p \mathrm{d}x\\
        &\le \frac{|\Omega|^{p-1}}{\left(\int_\Omega\theta(z) \mathrm{d}z\right)^p} \int_\Omega \int_\Omega \left| u(y) \nabla \theta(y) + \nabla_1 u(x,y) \theta(y) \right|^p \mathrm{d}y \mathrm{d}x\\
        &\le \frac{(2|\Omega|)^{p-1}}{\left(\int_\Omega\theta(z) \mathrm{d}z\right)^p} \int_\Omega \int_\Omega |u(y)|^p |\nabla \theta(y)|^p + |\nabla_1 u(x,y)|^p \theta(y)^p \mathrm{d}y \mathrm{d}x\\
        &\le \frac{(2|\Omega|)^{p-1}}{\left(\int_\Omega\theta(z) \mathrm{d}z\right)^p} \left(\|\nabla \theta \|_{L^\infty(\Omega;\mathbb{R}^N)}^p\int_\Omega |u(y)|^p \mathrm{d}y + \int_\Omega \int_\Omega \theta(y)^p |\nabla_1 u(x,y)|^p \mathrm{d}y \mathrm{d}x\right)\\
        &\le \frac{(2|\Omega|)^{p-1}}{\left(\int_\Omega\theta(z) \mathrm{d}z\right)^p} \left(\|\nabla \theta \|_{L^\infty(\Omega;\mathbb{R}^N)}^p\int_\Omega |u(y)|^p \mathrm{d}y + \int_\Omega \int_\Omega \left(\theta(y)^p + \theta(x)^p\right) |\nabla_1 u(x,y)|^p \mathrm{d}y \mathrm{d}x\right)\\
        &\le \frac{(2|\Omega|)^{p-1}}{\left(\int_\Omega\theta(z) \mathrm{d}z\right)^p} \left(\|\nabla \theta \|_{L^\infty(\Omega;\mathbb{R}^N)}^p\int_\Omega |u(y)|^p \mathrm{d}y + \int_\Omega \int_\Omega \omega_\theta(x,y)^p |\nabla_1 u(x,y)|^p \mathrm{d}y \mathrm{d}x\right)\\
        &\le \frac{(2|\Omega|)^{p-1}}{\left(\int_\Omega\theta(z) \mathrm{d}z\right)^p} \left(\|\nabla \theta \|_{L^\infty(\Omega;\mathbb{R}^N)}^p\int_\Omega |u(y)|^p \mathrm{d}y + \int_\Omega \int_\Omega |\nabla_{\omega_\theta} u(x,y)|^p \mathrm{d}x \mathrm{d}y\right)\\
    \end{align}
    where in the second to last inequality we have used that $|a|^p+|b|^p \le (|a|+|b|)^p$ for $a,b \in \mathbb{R}$. The computation above proves the desired embedding.

    We move onto the $p=+\infty$ case. Given $\Omega$, fix $\theta \in \mathcal{C}_0^1(\Omega)$ such that $\theta >0$ in $\Omega$ and consider $\omega_\theta(x,y):=\theta(x)+\theta(y)$ for $(x,y)\in\Omega\times\Omega$. Given $u\in W_{\omega_\theta}^{1,\infty}(\Omega)$ one has that, as argued above, for almost every $x \in \Omega$

    \begin{align}
        |\nabla u(x)| &= \left| -\frac{1}{\int_\Omega \theta(z) \mathrm{d}z} \int_\Omega \left[u(y) \nabla \theta(y) + \nabla_1 u(x,y) \theta(y) \right]\mathrm{d}y \right|\\
        &\le \frac{1}{\int_\Omega \theta(z) \mathrm{d}z} \left( |\Omega| \|\nabla \theta\|_{L^\infty(\Omega;\mathbb{R}^N)} \|u\|_{L^\infty(\Omega)} + \int_\Omega |\nabla_1 u(x,y)| \theta(y) \mathrm{d}y \right)\\
        &\le \frac{1}{\int_\Omega \theta(z) \mathrm{d}z} \left( |\Omega| \|\nabla \theta\|_{L^\infty(\Omega;\mathbb{R}^N)} \|u\|_{L^\infty(\Omega)} + \int_\Omega |\nabla_1 u(x,y)| (\theta(x) + \theta(y)) \mathrm{d}y \right)\\
        &\le \frac{1}{\int_\Omega \theta(z) \mathrm{d}z} \left( |\Omega| \|\nabla \theta\|_{L^\infty(\Omega;\mathbb{R}^N)} \|u\|_{L^\infty(\Omega)} + \int_\Omega |\nabla_1 u(x,y)| \omega_\theta(x,y) \mathrm{d}y \right)\\
        &\le \frac{1}{\int_\Omega \theta(z) \mathrm{d}z} \left( |\Omega| \|\nabla \theta\|_{L^\infty(\Omega;\mathbb{R}^N)} \|u\|_{L^\infty(\Omega)} + \int_\Omega |\nabla_{\omega_\theta} u(x,y)| \mathrm{d}y \right)\\
        &\le \frac{|\Omega|}{\int_\Omega \theta(z) \mathrm{d}z} \left( \|\nabla \theta\|_{L^\infty(\Omega;\mathbb{R}^N)} \|u\|_{L^\infty(\Omega)} + \|\nabla_{\omega_\theta} u\|_{L^\infty(\Omega\times\Omega;\mathbb{R}^N)} \right),
    \end{align}
    and by taking essential supremum over $x\in\Omega$ the embedding $W^{1,\infty}_{\omega_\theta}(\Omega) \hookrightarrow W^{1,\infty}(\Omega)$ is proven.

    For the bounded variation case, again, let $\Omega$ be of finite measure, fix $\theta \in \mathcal{C}_0^1(\Omega)$ such that $\theta >0$ in $\Omega$ and consider $\omega_\theta(x,y):=\theta(x)+\theta(y)$ for $(x,y)\in\Omega\times\Omega$. For any $\varphi \in \mathcal{C}_c^1(\Omega;\mathbb{R}^N)$ such that $\|\varphi\|_{L^\infty(\Omega;\mathbb{R}^N)} \le 1$ define the $\mathcal{C}_0^1(\Omega\times\Omega;\mathbb{R}^N)$ function $\phi_{\theta,\varphi}(x,y):=\theta(x) \varphi(y)$. For any $u \in \mathrm{NLBV}_{\omega_\theta}(\Omega)$, by proceeding as in the previous computation \eqref{div1intermsofdiv}, we have that

    \begin{equation} \label{NotNecRemarkAux1}
        \int_\Omega u(x) \mathrm{div}(\varphi)(x)\mathrm{d}x = \frac{1}{\int_\Omega \theta(z)\mathrm{d}z} \left( \int_\Omega u(x) \mathrm{div}_1 (\phi_{\theta,\varphi})(x)\mathrm{d} x + \int_\Omega u(x) \left[ \nabla \theta (x) \cdot \int_\Omega \varphi(y) \mathrm{d}y \right] \mathrm{d}x \right).
    \end{equation}
    
    Now, for $(x,y)\in\Omega\times\Omega$ we have

    \begin{equation}
        |\phi_{\theta,\varphi}(x,y)| \le \|\varphi\|_{L^\infty(\Omega;\mathbb{R}^N)} \theta(x) \le \theta(x) \le\theta(x)+\theta(y) = \omega_\theta(x,y),
    \end{equation}
    and by continuing from \eqref{NotNecRemarkAux1} and using Lemma \ref{NLTV^omegaLemma} (coupled with Remark \ref{NLBVDefExtension}), we have

    \begin{align}
        \int_\Omega u(x) \mathrm{div}(\varphi)(x)\mathrm{d}x &\le \frac{1}{\int_\Omega \theta(z)\mathrm{d}z} \left( \mathrm{NLTV}_{\omega_\theta}(u) + \int_\Omega u(x) \left[ \nabla \theta (x) \cdot \int_\Omega \varphi(y) \mathrm{d}y \right] \mathrm{d}x \right)\\
        &\le \frac{1}{\int_\Omega \theta(z)\mathrm{d}z} \left( \mathrm{NLTV}_{\omega_\theta}(u) + |\Omega| \|\nabla \theta\|_{L^\infty(\Omega;\mathbb{R}^N)} \int_\Omega |u(x)| \mathrm{d}x \right).
    \end{align}

    By taking supremum over all $\varphi$, the embedding $\mathrm{NLBV}_{\omega_\theta}(\Omega) \hookrightarrow\mathrm{BV}(\Omega)$ is proven.
\end{remark}

For completeness sake, we also prove that the condition of $f_\omega^p$ being bounded away from zero is a necessary condition for the $W^{1,p}_\omega \hookrightarrow W^{1,p}$ and $\mathrm{NLBV}_\omega \hookrightarrow \mathrm{BV}$ embeddings to hold.

\begin{prop} \label{NecCondCompletenessSob}
    Let $p\in[1,+\infty]$, $\Omega$ be of finite measure if $p\in [1,+\infty)$ or general if $p=+\infty$, and $\omega$ be an admissible weight function in $\Omega$. If  $W^{1,p}_\omega(\Omega) \hookrightarrow W^{1,p}(\Omega)$, then there exists $k>0$ such that for all $x\in\Omega$
    \begin{equation}
        k \le f_\omega^p(x).
    \end{equation}
\end{prop}

\begin{proof}
    We start with the $p\in [1,+\infty)$ case. By Theorem \ref{Theorem_equality_constant_weight}, the embedding $W^{1,p}_\omega(\Omega) \hookrightarrow W^{1,p}(\Omega)$ is equivalent to $W^{1,p}_\omega(\Omega) \hookrightarrow W^{1,p}_1(\Omega)$, which implies that there exists $c>0$ such that

    \begin{equation} \label{P1SuffCondAux8}
        \int_{\Omega\times\Omega} |\nabla_1 u|^p \mathrm{d}x\mathrm{d}y \le c\left( \int_\Omega |u|^p \mathrm{d}x + \int_{\Omega\times\Omega} |\nabla_\omega u|^p \mathrm{d}x\mathrm{d}y \right)
    \end{equation}
    for all $u\in W^{1,p}_\omega(\Omega)$. Fix $x_0 \in \Omega$ and let $\epsilon_0>0$ be such that $B_{\epsilon_0}(x_0)\subset \Omega$. We suppose that $f_\omega(x_0) <+\infty$, as any $k>0$ will work otherwise. Consider for $\epsilon \in (0,\epsilon_0)$ the $h_{p,\epsilon}^{x_0}$ function introduced in Lemma \ref{P1TechLemmaNecCond}. Recall that its $L^p(\Omega)$ norm tends to zero as $\epsilon \to 0$ $-$\eqref{h_epsilon_lp_to_zero}$-$. With this, substituting $u$ by $h_{p,\epsilon}^{x_0}$ in \eqref{P1SuffCondAux8} and letting $\epsilon$ go to zero yields by Lemma \ref{P1TechLemmaNecCond} that

     \begin{equation}
        2|\Omega| \le 2c \int_\Omega \omega(x_0,y)^p \mathrm{d}y = 2c f_\omega^p(x_0),
    \end{equation}
    from which, by choosing $k=|\Omega|/c$ the first statement is proven for $p\in [1,+\infty)$.

    We now prove the $p=+\infty$ case. Note that $W^{1,\infty}_\omega(\Omega) \hookrightarrow W^{1,\infty}(\Omega)$ means that there exists $c'>0$ such that

    \begin{equation} \label{P1SuffCondAux10}
        \|\nabla u \|_{L^\infty(\Omega;\mathbb{R}^N)} \le c'\left( \|u\|_{L^\infty(\Omega)} + \|\nabla_\omega u \|_{L^\infty(\Omega\times\Omega;\mathbb{R}^N)}\right)
    \end{equation}
    for all $u \in W^{1,\infty}_\omega(\Omega)$.  Fix $x_0 \in \Omega$ and let $\epsilon_0>0$ be such that $B_{\epsilon_0}(x_0)\subset \Omega$. We suppose that $f_\omega^\infty(x_0) <+\infty$, as any $k>0$ will work otherwise. For $\epsilon\in(0,\epsilon_0)$ consider the function $\rho_\epsilon^{x_0}$ as defined in \eqref{RhoDef}. By substituting $u$ by  $\rho_\epsilon^{x_0}$ in \eqref{P1SuffCondAux10} and using \eqref{Problem1InftyAux7} and \eqref{Problem1InftyAux8} we obtain that

    \begin{align}
        1 &\le c' \left( \frac{F(1)}{J(0)} \epsilon + \esssup_{(x,y)\in\Omega\times\Omega} \omega(x,y) |\nabla\rho_\epsilon^{x_0}(y) - \nabla \rho_\epsilon^{x_0}(x)|  \right)\\
        &\le c' \left( \frac{F(1)}{J(0)} \epsilon + 2 \esssup_{(x,y)\in\Omega\times\Omega} \omega(x,y) |\nabla \rho_\epsilon^{x_0}(x)|  \right)\\
        &= c' \left( \frac{F(1)}{J(0)} \epsilon + 2 \esssup_{(x,y)\in B_\epsilon(x_0)\times\Omega} \omega(x,y) |\nabla \rho_\epsilon^{x_0}(x)|  \right)\\
        &\le c' \left( \frac{F(1)}{J(0)} \epsilon + 2 \esssup_{(x,y)\in B_\epsilon(x_0)\times\Omega} \omega(x,y) \right),
    \end{align}
    and by taking the limit $\epsilon \to 0$ (and using that $\omega$ is continuous) we obtain that

    \begin{equation}
        1 \le 2 c' \esssup_{y\in \Omega} \omega(x_0,y) = 2 c' f_\omega^\infty(x_0),
    \end{equation}
    from which, by choosing $k=1/(2 c')$ the statement is proven for $p=+\infty$.
\end{proof}

\begin{prop} \label{NecCondCompletenessBV}
    Let $\Omega$ be of finite measure and $\omega$ be an admissible weight function in $\Omega$. If $\mathrm{NLBV}_\omega(\Omega) \hookrightarrow \mathrm{BV}(\Omega)$, then then there exists $k'>0$ such that for all $x\in\Omega$
    \begin{equation}
        k' \le f_\omega^1(x).
    \end{equation}
\end{prop}

\begin{proof}
    The result follows from reproducing the proof of Proposition \ref{NecCondCompletenessSob} with $p=1$ and using Proposition \ref{C1NLTV}.
\end{proof}

\section{Relation with the space of test functions} \label{RelationTestFunctionsSubsection}

In this section, we give the proof of Theorem \ref{Theorem_TestFunctions}. Recall that it gives necessary and sufficient conditions on a weight function so that the intersection of the space of smooth functions with compact support and the nonlocal weighted Sobolev/$\mathrm{BV}$ spaces is non zero, and so that the first space mentioned is included in one of the others.

To understand why smooth functions with compact support and a $\mathcal{C}^1$ weight function can give rise to infinite integrals/supremums, we give the remark below.

\begin{remark} \label{InteractionsRemark}
    Let $p\in[1,+\infty)$ and $\omega$ an admissible weight function in $\Omega$. Then, given $u\in L^p_{\mathrm{loc}}(\Omega)$ with weak (standard) first order derivatives in all directions, by Remark \ref{1impliesomegaRemark} and Proposition \ref{WeakDerImplyNlWeakDer}, one has

    \begin{align}
        \int_{\Omega} \int_{\Omega} |\nabla_\omega u|^p \mathrm{d}x\mathrm{d}y &= \int_{\esssupp(u)} \int_{\esssupp(u)} \omega(x,y)^p |\nabla u(y)-\nabla u(x)|^p \mathrm{d}x\mathrm{d}y\\
        &\qquad + 2\int_{\esssupp(u)} |\nabla u(x)|^p \int_{\Omega \backslash \esssupp(u)} \omega(x,y)^p \mathrm{d}y \mathrm{d}x.\label{InteractionsEq}
    \end{align}

    Moreover, given $u\in L^\infty_{\mathrm{loc}}(\Omega)$ with weak (local) first order derivatives in all directions, one has

    \begin{align}
        \esssup_{\Omega\times\Omega} |\nabla_\omega u| &= \max\left\{ \esssup_{(x,y)\in \esssupp(u) \times \esssupp(u)} \omega(x,y) |\nabla u(y) - \nabla u(x)| \right.,\nonumber\\
        &\qquad \qquad \quad \left. \esssup_{(x,y) \in \esssupp(u) \times \left(\Omega \setminus \esssupp(u)\right)} \omega(x,y) |\nabla u(x)| \right\}. \label{InteractionsEqInfty}
    \end{align}
\end{remark}

For $p\in[1,+\infty)$, the nonlocal Sobolev seminorm of $u$ splits into the sum of two integrals with $|\nabla_\omega u|$ as the integrand (though in the second integral the integrand takes on a simplified form) but with different domains of integration, one dealing with the interactions in $\esssupp(u) \times \esssupp(u)$ and the other with the interactions in $\esssupp(u)\times \left(\Omega \backslash \esssupp(u) \right)$. For $p=+\infty$, the nonlocal Sobolev seminorm of $u$ is the maximum of the essential supremum of $|\nabla_\omega u|$ in $\esssupp(u) \times \esssupp(u)$ and in $\esssupp(u)\times \left(\Omega \backslash \esssupp(u) \right)$ (again, in the second supremum the function takes on a simplified form). Thus, the seminorms are finite and therefore $u$ belongs to the corresponding nonlocal Sobolev spaces only if both integrals/supremums are finite. A fact which will be used repeatedly throughout this subsection is that if $u$ is smooth and of compact support, the integral/supremum in $\esssupp(u) \times \esssupp(u)$ is always finite, as then, in this region, $\nabla_1 u$ is integrable and bounded and $\omega$ is bounded. For such a $u$, the finiteness of its nonlocal Sobolev seminorm thus depends directly on the finiteness of the second integral/supremum of \eqref{InteractionsEq} and \eqref{InteractionsEqInfty}, respectively, see Lemma \ref{TestFuncLemma} below. Note that this integral/supremum is not always finite, as $\omega$ may not be integrable/bounded in one of its components in the whole domain.

\begin{lemma} \label{TestFuncLemma}
    Let $\omega$ be an admissible weight function and $\varphi \in \mathcal{C}_c^\infty(\Omega)$. For any $p\in[1,+\infty)$ we have

    \begin{equation} \label{TestFuncAux4}
        \varphi \in W^{1,p}_\omega(\Omega) \iff \int_{\mathrm{supp}(\varphi)} |\nabla \varphi(x)|^p \int_{\Omega \backslash \mathrm{supp}(\varphi)} \omega(x,y)^p \mathrm{d}y \mathrm{d}x <+\infty.
    \end{equation}

    Moreover, we have

    \begin{equation}
        \varphi \in W^{1,\infty}_\omega(\Omega) \iff \sup_{(x,y)\in\mathrm{supp}(\varphi)\times\left(\Omega\setminus\mathrm{supp}(\varphi)\right)} |\nabla\varphi(x)|\omega(x,y) < +\infty.
    \end{equation}
\end{lemma}

\begin{proof}
    For $p\in[1,+\infty)$, the claim follows from the computation

    \begin{align}
        \int_{\Omega} \int_{\Omega} |\nabla_\omega \varphi|^p \mathrm{d}x\mathrm{d}y &= \int_{\mathrm{supp}(\varphi)} \int_{\mathrm{supp}(\varphi)} \omega(x,y)^p |\nabla \varphi(y)-\nabla \varphi(x)|^p \mathrm{d}x\mathrm{d}y\\
        &\qquad + 2\int_{\mathrm{supp}(\varphi)} |\nabla \varphi(x)|^p \int_{\Omega \backslash \mathrm{supp}(\varphi)} \omega(x,y)^p \mathrm{d}y \mathrm{d}x.\label{Q2Answer1}
    \end{align}
    and the fact that the first integral of the right side of the equality is always finite, as $\omega^p$ is integrable in the compact set $\mathrm{supp}(\varphi)\times \mathrm{supp}(\varphi)$ and $|\nabla \varphi(y) - \nabla \varphi(x)|$ is bounded for $(x,y)\in\mathrm{supp}(\varphi)\times \mathrm{supp}(\varphi)$.

    Similarly, for the $p=+\infty$ case, the claim follows from the computation

    \begin{align}
        \sup_{\Omega\times\Omega} |\nabla_\omega \varphi| &= \max\left\{ \sup_{(x,y)\in\mathrm{supp}(\varphi)\times\mathrm{supp}(\varphi)} \omega(x,y) |\nabla \varphi(y)-\nabla \varphi(x)| \right.,\\
        &\qquad \qquad \qquad \qquad \left. \sup_{(x,y)\in\mathrm{supp}(\varphi)\times\left(\Omega\setminus\mathrm{supp}(\varphi)\right)} \omega(x,y)|\nabla\varphi(x)| \right\},
    \end{align}
    and the fact that the first supremum of the right side of the equality is always finite, as $\omega$ is bounded in the compact set $\mathrm{supp}(\varphi)\times \mathrm{supp}(\varphi)$ and, again, $|\nabla \varphi(y) - \nabla \varphi(x)|$ is bounded for $(x,y)\in\mathrm{supp}(\varphi)\times \mathrm{supp}(\varphi)$. 
\end{proof}

With the lemma above and the observations made at the start of the subsection, the proof of Theorem \ref{Theorem_TestFunctions} follows in a straightforward manner. 

\begin{proof}[Proof of Theorem \ref{Theorem_TestFunctions}]
    We start by proving that

    \begin{equation} \label{TestFuncAux1}
        W^{1,p}_\omega(\Omega)\cap \mathcal{C}_c^\infty(\Omega) \neq \{0\} \iff
        \left\{\begin{array}{c}
            \mbox{there exists } K \subset \Omega \mbox{ compact, with nonempty interior}  \\
            \mbox{and such that } f_\omega^p|_K \in L^1(K)  \\
     \end{array}
     \right\}
    \end{equation}
    for any $p\in[1,+\infty)$ and admissible weight function $\omega$. Suppose that there exists $\varphi \in W^{1,p}_\omega(\Omega) \cap \mathcal{C}_c^\infty(\Omega)$ such that $\varphi \not \equiv 0$, and choose ${\delta_\varphi}>0$ and $K_\varphi \subset \mathrm{supp}(\varphi)$ compact and with nonempty interior such that $\inf_{K_\varphi} |\nabla \varphi| \ge {\delta_\varphi}$. We then have

    \begin{align}
        \int_{K_\varphi} \int_{\Omega \setminus \mathrm{supp}(\varphi)} \omega(x,y)^p \mathrm{d}y \mathrm{d}x &\le \frac{1}{{\delta_\varphi}^p} \int_{K_\varphi} |\nabla \varphi(x)|^p \int_{\Omega \setminus \mathrm{supp}(\varphi)} \omega(x,y)^p \mathrm{d}y \mathrm{d}x\\
        &\le \frac{1}{{\delta_\varphi}^p} \int_{\mathrm{supp}(\varphi)} |\nabla \varphi(x)|^p \int_{\Omega \setminus \mathrm{supp}(\varphi)} \omega(x,y)^p \mathrm{d}y \mathrm{d}x <+\infty.
    \end{align}

    Hence left to right implication of \eqref{TestFuncAux1} follows form the fact that $\omega^p$ is integrable in ${K_\varphi}\times\mathrm{supp}(\varphi)$, as this is a compact set and $\omega$ is continuous.
    
    For the right to left implication, let $K\subset \Omega$ be a compact set with nonempty interior such that $f_\omega^p|_K \in L^1(K)$. Then choose $\varphi_K \in \mathcal{C}_c^\infty(\Omega)$ such that $\varphi_K \not \equiv 0$ and $\mathrm{supp}(\varphi_K) \subset K$. We have
    
    \begin{align}
        \int_{\mathrm{supp}(\varphi_K)} |\nabla \varphi_K(x)|^p \int_{\Omega\setminus \mathrm{supp}(\varphi_K)} \omega(x,y)^p \mathrm{d}y\mathrm{d}x &\le \|\nabla \varphi_K\|_{L^\infty(\Omega;\mathbb{R}^N)}^p \int_{\mathrm{supp}(\varphi_K)} \int_{\Omega\setminus \mathrm{supp}(\varphi_K)} \omega(x,y)^p \mathrm{d}y\mathrm{d}x\\
        &\le \|\nabla \varphi_K\|_{L^\infty(\Omega;\mathbb{R}^N)}^p \int_K \int_{\Omega} \omega(x,y)^p \mathrm{d}y\mathrm{d}x\\
        &\le \|\nabla \varphi_K\|_{L^\infty(\Omega;\mathbb{R}^N)}^p \int_K f_\omega^p(x) \mathrm{d}x <+\infty
    \end{align}
    by hypothesis. By Lemma \ref{TestFuncLemma}, this implies that $\varphi_K \in W^{1,p}_\omega(\Omega)$, from which $W^{1,p}_\omega(\Omega)\cap \mathcal{C}_c^\infty(\Omega) \neq \{0\}$ and \eqref{TestFuncAux1} is proven.

    We move onto proving

    \begin{equation} \label{TestFuncAux2}
        W^{1,\infty}_\omega(\Omega)\cap \mathcal{C}_c^\infty(\Omega) \neq \{0\} \iff \left\{\begin{array}{c}
            \mbox{there exists } K \subset \Omega \mbox{ compact, with nonempty interior}  \\
            \mbox{and such that } f_\omega^\infty|_K \in L^\infty(K)  \\
     \end{array}
     \right\}
    \end{equation}
    for any admissible weight function $\omega$. Let $\varphi \in W^{1,\infty}_\omega(\Omega)\cap \mathcal{C}_c^\infty(\Omega)$ be such that $\varphi \not \equiv 0$. Again, choose ${\delta_\varphi}>0$ and $K_\varphi\subset \mathrm{supp}(\varphi)$ compact and with nonempty interior such that $\inf_{K_\varphi} |\nabla \varphi| \ge {\delta_\varphi}$. By Lemma \ref{TestFuncLemma} we have

    \begin{align}
        \sup_{(x,y)\in {K_\varphi}\times\left(\Omega\setminus\mathrm{supp}(\varphi)\right)} \omega(x,y) &\le \frac{1}{{\delta_\varphi}} \sup_{(x,y)\in {K_\varphi}\times\left(\Omega\setminus\mathrm{supp}(\varphi)\right)} |\nabla \varphi(x)| \omega(x,y)\\
        &\le \frac{1}{{\delta_\varphi}} \sup_{(x,y)\in \mathrm{supp}(\varphi)\times\left(\Omega\setminus\mathrm{supp}(\varphi)\right)} |\nabla \varphi(x)| \omega(x,y) < +\infty.
    \end{align}

    The left to right implication of \eqref{TestFuncAux2} then follows from the fact that $\omega$ is bounded in $K\times \mathrm{supp}(\varphi)$, as this is a compact set and $\omega$ is continuous.

    To prove the right to left implication of \eqref{TestFuncAux2}, let $K\subset \Omega$ be a compact set with nonempty interior and such that $f_\omega^\infty|_K \in L^\infty(K)$. Then choose $\varphi_K \in \mathcal{C}_c^\infty(\Omega)$ such that $\varphi_K \not \equiv 0$ and $\mathrm{supp}(\varphi_K) \subset K$. We have

    \begin{align}
        \sup_{(x,y)\in \mathrm{supp}(\varphi_K)\times\left(\Omega\setminus\mathrm{supp}(\varphi_K)\right)} |\nabla \varphi_K(x)| \omega(x,y) &\le \|\nabla \varphi_K\|_{L^\infty(\Omega;\mathbb{R}^N)} \sup_{(x,y)\in \mathrm{supp}(\varphi_K)\times\left(\Omega\setminus\mathrm{supp}(\varphi_K)\right)} \omega(x,y)\\
        &\le \|\nabla \varphi_K\|_{L^\infty(\Omega;\mathbb{R}^N)} \sup_{(x,y)\in K\times\Omega} \omega(x,y)\\
        &\le  \|\nabla \varphi_K\|_{L^\infty(\Omega;\mathbb{R}^N)} \sup_{x\in K} f_\omega^\infty(x) <+\infty
    \end{align}
    by hypothesis. By Lemma \ref{TestFuncLemma}, $\varphi_K\in W^{1,\infty}_\omega(\Omega)$, and \eqref{TestFuncAux2} is proven.

    Let us now prove that
    \begin{equation} \label{TestFuncAux3}
        \mathcal{C}_c^\infty(\Omega) \subset W^{1,p}_\omega(\Omega) \iff f_\omega^p \in L^1_{\mathrm{loc}}(\Omega)
    \end{equation}
    for any $p\in[1,+\infty)$ and admissible weight function $\omega$.
    
    Assume $f_\omega^p \in L^1_{\mathrm{loc}}(\Omega)$. Then, given $\varphi \in \mathcal{C}_c^\infty(\Omega)$ we have
    \begin{align}
        \int_{\mathrm{supp}(\varphi)} |\nabla \varphi(x)|^p \int_{\Omega \backslash \mathrm{supp}(\varphi)} \omega(x,y)^p \mathrm{d}y \mathrm{d}x &\le \int_{\mathrm{supp}(\varphi)} |\nabla \varphi(x)|^p \int_{\Omega} \omega(x,y)^p \mathrm{d}y \mathrm{d}x\\
        &= \int_{\mathrm{supp}(\varphi)} |\nabla \varphi(x)|^p f_\omega^p(x) \mathrm{d}x <+\infty
    \end{align}
    by Hölder's inequality and the fact that $\left.f_\omega^p\right|_{\mathrm{supp}(\varphi)}\in L^1(\mathrm{supp}(\varphi))$ by hypothesis. By Lemma \ref{TestFuncLemma}, this implies that $\varphi \in W^{1,p}_\omega(\Omega)$, and thus the right to left implication of \eqref{TestFuncAux3} is proven.
    
    For the left to right implication of \eqref{TestFuncAux3}, $\mathcal{C}_c^\infty(\Omega) \subset W^{1,p}_\omega(\Omega)$ implies by Lemma \ref{TestFuncLemma} that 
    $$\int_{\mathrm{supp}(\varphi)} |\nabla \varphi(x)|^p \int_{\Omega \backslash \mathrm{supp}(\varphi)} \omega(x,y)^p \mathrm{d}y \mathrm{d}x <+\infty$$
    for all $\varphi \in \mathcal{C}_c^\infty(\Omega)$.
    
    This, coupled with the fact that $\omega$ is integrable in any compact set yields that
    \begin{equation}\label{Q2Answer2}
        \int_{\mathrm{supp}(\varphi)} |\nabla \varphi(x)|^p f_\omega^p(x) \mathrm{d}x =\int_{\mathrm{supp}(\varphi)} |\nabla \varphi(x)|^p \int_{\Omega} \omega(x,y)^p \mathrm{d}y \mathrm{d}x <+\infty
    \end{equation}
    for any $\varphi \in \mathcal{C}_c^\infty(\Omega)$. This implies that $f_\omega^p$ is integrable on any compact set with nonempty interior, and the left to right implication of \eqref{TestFuncAux3} is proven. Indeed, let $K\subset\Omega$ be a compact set and choose $\varphi_K \in \mathcal{C}_c^\infty(\Omega)$ so that $K \subset\mathrm{supp}(\varphi_K)$ and $|\nabla \varphi_K| \ge 1$ on $K$. Then by \eqref{Q2Answer2}, we have
    \begin{align}
        \int_K f_\omega^p(x) \mathrm{d}x &\le \int_{\mathrm{supp}(\varphi_K)} |\nabla \varphi_K(x)|^p f_\omega^p(x) \mathrm{d}x <+\infty.
    \end{align}

    Finally, we prove
    \begin{equation} \label{TestFuncAux5}
        \mathcal{C}_c^\infty(\Omega) \subset W^{1,\infty}_\omega(\Omega) \iff f_\omega^\infty \in L^\infty_{\mathrm{loc}}(\Omega)
    \end{equation}
    for any admissible weight function $\omega$.

    Suppose $f_\omega^\infty \in L^\infty_{\mathrm{loc}}(\Omega)$. Given $\varphi\in\mathcal{C}_c^\infty(\Omega)$, we have
    
    $$\sup_{(x,y)\in \mathrm{supp}(\varphi)\times(\Omega\backslash\mathrm{supp}(\varphi))} |\nabla\varphi(x)| \omega(x,y) \le \|\nabla \varphi\|_{L^\infty(\Omega;\mathbb{R}^N)} \sup_{x\in\mathrm{supp}(\varphi)}f_\omega^\infty(x) <+\infty$$
    by hypothesis. By Lemma \ref{TestFuncLemma}, this implies that $\varphi \in W^{1,\infty}_\omega(\Omega)$, and the right to left implication of \eqref{TestFuncAux5} is proven.
    
    On the other hand, suppose that $\mathcal{C}_c^\infty(\Omega) \subset W^{1,\infty}_\omega(\Omega)$. Given $K\subset \Omega$ a compact set, choose $\varphi_K \in \mathcal{C}_c^\infty(\Omega)$ such that $K \subset \mathrm{supp}(\varphi_K)$ and $\inf_K |\nabla \varphi_K| \ge 1$. We have, by Lemma \ref{TestFuncLemma} that

    \begin{align}
        \sup_{(x,y)\in K\times\left(\Omega \setminus \mathrm{supp}(\varphi_K)\right)} \omega(x,y) &\le \sup_{(x,y)\in K\times \left(\Omega \setminus \mathrm{supp}(\varphi_K)\right)} |\nabla \varphi_K(x)| \omega(x,y)\\
        &\le \sup_{(x,y)\in \mathrm{supp}(\varphi_K)\times \left(\Omega \setminus \mathrm{supp}(\varphi_K)\right)} |\nabla \varphi_K(x)| \omega(x,y) <+\infty.
    \end{align}
    
    Thus the left to right implication of \eqref{TestFuncAux5} follows from the above computation and from the fact that $\omega$ is also bounded in $K\times \mathrm{supp}(\varphi_K)$.

    For the $\mathrm{NLBV}_\omega$ statements, one uses Proposition \ref{C1NLTV} and reproduces the arguments above for $p=1$.
\end{proof}

\section{Triviality when the domain has infinite measure} \label{Section_trivial}

In this section we prove Theorem \ref{Theorem_Trivial}, which states that the nonlocal Sobolev space induced by a bounded away from zero weight function is the trivial space whenever the domain has infinite measure and the exponent is non infinite. Moreover, it also states that the space of functions with nonlocal bounded variation induced by a weight satisfying the same condition is also the trivial space provided the domain satisfies a certain condition. The proof will be based on the following result:

\begin{lemma} \label{NonConstantInfSeminorm}
     Let $\Omega$ be of infinite measure, $m \in \mathbb{N}$ and $g: \Omega \times \Omega \to \mathbb{R}$ be a measurable function such that
     \begin{equation}\label{inf_g}
         c \le \essinf_{(x,y) \in \Omega \times \Omega}g(x,y)
     \end{equation}
     for some constant $c>0$. Let $f:\Omega\to\mathbb{R}^m$ be measurable and not a.e.~equal to a constant. Then,
    \begin{equation}
        \int_\Omega\int_\Omega g(x,y)|f(y)-f(x)|^p \mathrm{d}x\mathrm{d}y = +\infty.
    \end{equation}
\end{lemma}

\begin{proof}
    We claim that we can find $U$ and $V$ two nonempty measurable subsets of $\Omega$ such that $U$ has infinite measure and $V$ has positive measure as well as a constant $\delta>0$ such that
    \begin{equation}\label{Remark_InfinteMeasureAux}
    \essinf_{(x,y) \in U \times V} \lvert f(y) - f(x) \rvert^p \geq\delta.
    \end{equation}
    
    Once this claim is proven, the proof of the result readily follows from \eqref{Remark_InfinteMeasureAux}, as then one can write
    \begin{align}
        \int_\Omega\int_\Omega g(x,y)|f(y)-f(x)|^p \mathrm{d}x\mathrm{d}y \ge \int_U\int_V g(x,y)|f(y)-f(x)|^p \mathrm{d}x\mathrm{d}y \ge c\delta  |U| |V| = +\infty.
    \end{align}
    
    Hence, it only remains to prove the claim. First, we argue that there exist $v_1, v_2 \in \mathbb{R}^m$, $v_1\neq v_2$ and $0< r \leq \lvert v_1-v_2\rvert/2$ such that $f^{-1}(B(v_1,r))$ and $f^{-1}(B(v_2,r))$ have positive measure. Indeed, if this were not to be the case, one would prove that the measure of $\Omega$ is accumulated on the preimage of a ball which we could make arbitrarily small, and hence $|\Omega|=|f^{-1}(\{v\})|$ for a certain $v\in\mathbb{R}^m$, which would contradict the fact that $f$ is not a.e. equal to a constant. Now, as $\Omega= f^{-1}(\mathbb{R}^m) = f^{-1}(\overline{B(v_1,r)}) \cup f^{-1}(\mathbb{R}^m \setminus B(v_1,r)) \cup f^{-1}(A)$, where $A:= \{ v \in \mathbb{R}^m: r/2 \le \lvert v-v_1 \rvert \le 3r/2\}$\footnote{The subset $A$ must be considered, as the infinite measure of $\Omega$ could be accumulated on the preimage of the set $\left\{ |v-v_1|=r \right\}$.} and $\Omega$ has infinite measure, one of these three sets must have infinite measure.

    \begin{itemize}
        \item If $f^{-1}(\overline{B(v_1,r)})$ has infinite measure, choose
    \begin{equation*}
    U:= f^{-1}\left(\overline{B(v_1,r)}\right) \mbox{ and } V:= f^{-1}\left(\overline{B(v_2,r/2)}\right).
    \end{equation*}
    
    Then by definition
    \begin{equation*}
    \lvert f(x)-v_1 \rvert \le r \mbox{ for a. e. } x \in U,
    \end{equation*}
    from which, by the triangle's inequality
    \begin{equation}
        |f(x)-v_2| \ge |v_1-v_2| - |f(x)-v_1| \ge 2r - r=r \mbox{ for a. e. } x \in U,
    \end{equation}
    and
    \begin{equation*}
    \lvert f(y)-v_2 \rvert \leq r/2 \mbox{ for a. e. } y \in V.
    \end{equation*}
    
    At this point, one uses the triangle's inequality which establishes \eqref{Remark_InfinteMeasureAux} for $\delta:=(r/2)^p$ and $U, V$ as above.

    \item If $f^{-1}(\mathbb{R}^m \setminus B(v_1,r))$ has infinite measure, choose
    \begin{equation*}
    U:=f^{-1}\left(\mathbb{R}^m \setminus B(v_1,r)\right) \mbox{ and } V:=f^{-1}\left(\overline{B(v_1,r/2)}\right),
    \end{equation*}
    and by a similar argument \eqref{Remark_InfinteMeasureAux} holds for $\delta:=(r/2)^p$ and $U, V$ as above.

    \item If $f^{-1}(A)$ has infinite measure, choose
    \begin{equation}
        U:=f^{-1}(A) \mbox{ and } V:=f^{-1}\left(\overline{B(v_1,r/4)}\right),
    \end{equation}
    and \eqref{Remark_InfinteMeasureAux} holds for $\delta:=(r/4)^p$ and $U, V$ as above.
    \end{itemize}
\end{proof}

The nonlocal Sobolev case will follow directly from Lemma \ref{NonConstantInfSeminorm}. For the $\mathrm{NLBV}_\omega$ case, the domain hypothesis will allow us to derive a density-type result in $\mathrm{NLBV}_\omega$, which coupled with Lemma \ref{NonConstantInfSeminorm} will complete the proof.

\begin{proof}[Proof of Theorem \ref{Theorem_Trivial}]
    Let $u \in W^{1,p}_\omega(\Omega)$ and recall that $0 < c \le\omega(x,y)$ for all $(x,y) \in \Omega \times \Omega$. By Remark \ref{TechRemarkFiniteMeasureNotNec}, we have that $u\in W^{1,p}_1(\Omega)$, from which, using Proposition \ref{NlWeakDerImplyWeakDer}, $u$ has standard weak derivatives in all directions in $\Omega$. By Remark \ref{1impliesomegaRemark} and Proposition \ref{WeakDerImplyNlWeakDer}, we then have that
    \begin{equation}
    \nabla_\omega u(x,y)=\omega(x,y)(\nabla u(y)-\nabla u(x)) \quad \mbox{ for almost every } (x,y) \in \Omega \times \Omega.
    \end{equation}
    
    Therefore, by applying Lemma \ref{NonConstantInfSeminorm} with $f= \nabla u$ and $g=\omega$ it follows that $\nabla u$ must be constant a.e. in $\Omega$. This means, using the fact that $\Omega$ has infinite measure and $u\in L^p(\Omega)$, that $u= 0$ a.e. in $\Omega$, which proves the Sobolev statement.

    We move onto the $\mathrm{NLBV}$ statement. Let us first see that $\mathrm{NLBV}(\Omega)=\{0\}$. This, coupled with the fact that if $\omega$ is bounded away from zero then we have the continuous embedding $\mathrm{NLBV}_\omega(\Omega) \hookrightarrow \mathrm{NLBV}(\Omega)$ by Remark \ref{TechRemarkFiniteMeasureNotNec}, will complete the proof.

    Let $u \in \mathrm{NLBV}(\Omega)$. By assumption, for any $n \in \mathbb{N}$ there exists a set $\Omega_n$ of infinite measure such that the quantity $\delta_n:=\mathrm{dist}(\Omega_n,\partial \Omega)$ is positive. For any $\epsilon >0$, let $J_\epsilon$ be the standard mollifier, which recall that we assume to be radially symmetric. If $\epsilon \in (0,\delta_n)$, then $J_\epsilon(z)=0$ for all $z\in \mathbb{R}^N$ such that $\lvert z \rvert \geq \delta_n$, which means that the function $u_{n,\epsilon}:=(J_\epsilon * u)$ is well defined in $\Omega_n$. What is more, by the well known properties of mollifiers $u_{n,\epsilon}$ has (standard) weak derivatives in all directions in $\Omega_n$.
    
    Let now $\phi \in \mathcal{C}^1_c(\Omega_n \times \Omega_n; \mathbb{R}^N)$ be such that $\|\phi\|_{L^\infty(\Omega_n\times\Omega_n;\mathbb{R}^N)}\le 1$. For the next computations, we will expand $\phi$ to the whole $\mathbb{R}^N\times\mathbb{R}^N$ by zero. We claim that, for $\epsilon\in(0,\delta_n)$ 
    \begin{equation}\label{convolution_equality_final}
        \int_{\Omega_n}u_{n,\epsilon}(x) \mathrm{div}_1 (\phi)(x) \mathrm{d}x=\int_{\Omega}u(x) \mathrm{div}_1(\phi_\epsilon)(x) \mathrm{d}x,
    \end{equation}
    where the integrals associated to the first and second $\mathrm{div}_1$ are to be taken over $\Omega_n$ and $\Omega$, respectively, and, for all $(x,y) \in \Omega \times \Omega$,
    \begin{equation}
    \phi_\epsilon(x,y):= \int_{\mathbb{R}^N} J_\epsilon(z)\phi(x-z,y-z) \mathrm{d}z.
    \end{equation}

    Indeed, by Fubini's theorem, Leibniz rule and the change of variables $\hat y = y-z$, we have
    \begin{align}
        \int_{\Omega}u(x) \mathrm{div}_1(\phi_\epsilon)(x) \mathrm{d}x &= \int_{\Omega}u(x) \left(\int_\Omega\mathrm{div}_x \left[\int_{\mathbb{R}^N} J_\epsilon(z) \left( \phi(y-z,x-z)-\phi(x-z,y-z) \right)\mathrm{d}z\right] \mathrm{d}y\right) \mathrm{d}x\\
        &= \int_{\Omega} u(x) \int_{\mathbb{R}^N} J_\epsilon(z) \left( \int_{\Omega} \mathrm{div}_x \left[ \phi(y-z,x-z) - \phi(x-z,y-z) \right] \mathrm{d}y \right) \mathrm{d}z \mathrm{d}x\\
        &= \int_{\Omega} u(x) \int_{\mathbb{R}^N} J_\epsilon(z) \left( \int_{\mathbb{R}^N} \mathrm{div}_x \left[ \phi(y-z,x-z) - \phi(x-z,y-z) \right] \mathrm{d}y \right) \mathrm{d}z \mathrm{d}x\\
        &= \int_{\Omega} u(x) \int_{\mathbb{R}^N} J_\epsilon(z) \left( \int_{\mathbb{R}^N} \mathrm{div}_x \left[ \phi(\hat y,x-z) - \phi(x-z,\hat y) \right] \mathrm{d}\hat y \right) \mathrm{d}z \mathrm{d}x\\
        &= \int_{\Omega} u(x) \int_{\mathbb{R}^N} J_\epsilon(z) \left( \int_{\Omega} \mathrm{div}_x \left[ \phi(\hat y,x-z) - \phi(x-z,\hat y) \right] \mathrm{d}\hat y \right) \mathrm{d}z \mathrm{d}x\\
        &= \int_{\Omega} u(x) \int_{\mathbb{R}^N} J_\epsilon(z) \mathrm{div}_1(\phi)(x-z) \mathrm{d}z \mathrm{d}x,\\
    \end{align}
    where in the fifth line we have used that $\phi$ vanishes outside $\Omega\times\Omega$ and in the third line we have used that, for $y \not  \in \Omega$ and $z \in \mathbb{R}^N$, either $|z|\ge \delta_n$ and thus $J_\epsilon(z)=0$, or $|z| < \delta_n$ and thus $y-z \not \in \Omega_n$, from which $\phi(y-z,\cdot)=\phi(\cdot,y-z)=0$. Moreover, by the change of variables $\hat z = x - z$, Fubini's theorem and the fact that $J_\epsilon$ is radially symmetric, we have

    \begin{align}
        \int_{\Omega} u(x) \int_{\mathbb{R}^N} J_\epsilon(z) \mathrm{div}_1(\phi)(x-z) \mathrm{d}z \mathrm{d}x &= \int_{\Omega} u(x) \int_{\mathbb{R}^N} J_\epsilon(x- \hat z) \mathrm{div}_1(\phi)(\hat z) \mathrm{d}\hat z \mathrm{d}x\\
        &= \int_{\mathbb{R}^N} \left( \int_\Omega u(x) J_\epsilon(\hat z -x ) \mathrm{d}x \right) \mathrm{div}_1(\phi)(\hat z) \mathrm{d}\hat z\\
        &= \int_{\Omega_n} \left( \int_\Omega u(x) J_\epsilon(\hat z -x ) \mathrm{d}x \right) \mathrm{div}_1(\phi)(\hat z) \mathrm{d}\hat z\\
        &= \int_{\Omega_n} \left( \int_{\mathbb{R}^N} u(x) J_\epsilon(\hat z -x ) \mathrm{d}x \right) \mathrm{div}_1(\phi)(\hat z) \mathrm{d}\hat z\\
        &=\int_{\Omega_n} u_{n,\epsilon}(\hat z) \mathrm{div}_1(\phi)(\hat z) \mathrm{d}\hat z,
    \end{align}
    where in the third line we have used that $\mathrm{div}_1(\phi)$ vanishes outside $\Omega_n$ and in the fourth that $J_\epsilon(\hat z -x)=0$ for $x \not \in \Omega$ and $\hat z \in \Omega_n$. Hence \eqref{convolution_equality_final} is proven. 
    
    With this in mind, notice that for $\epsilon \in (0,\delta_n)$ we have that $\phi_\epsilon \in \mathcal{C}^1_c(\Omega \times \Omega;\mathbb{R}^N)$ and, for any $(x,y) \in \Omega\times\Omega$
    
    \begin{align}
        \left|\phi_\epsilon(x,y)\right| &\le \int_{\mathbb{R}^N} J_\epsilon(z)|\phi(x-z,y-z)| \mathrm{d}z\le \|\phi\|_{L^\infty(\Omega_n\times\Omega_n;\mathbb{R}^N)} \le 1,
    \end{align}
    from which $\|\phi_\epsilon\|_{L^\infty(\Omega\times\Omega;\mathbb{R}^N)} \le 1$. Consequently, from \eqref{convolution_equality_final} we obtain

    \begin{equation}
        \int_{\Omega_n}u_{n,\epsilon}(x) \mathrm{div}_1 (\phi)(x) \mathrm{d}x \le \mathrm{NLTV}(u).
    \end{equation}

    Thus, by taking supremum over all $\phi$ we have

    \begin{equation}
        \mathrm{NLTV}(u_{n,\epsilon}) \le \mathrm{NLTV}(u) <+\infty.
    \end{equation}
    
    Note that in the above inequality the first and second $\mathrm{NLTV}$s are referring to the seminorms of $\mathrm{NLBV}(\Omega_n)$ and $\mathrm{NLBV}(\Omega)$, respectively. Therefore, as $u_{n,\epsilon}$ has standard weak derivatives in all directions, we can reproduce the argument of the Sobolev setting (with $\omega\equiv 1$) using Proposition \ref{C1NLTV}, which yields that $u_{n,\epsilon}=0$ a.e. in $\Omega_n$ for all $\epsilon\in(0,\delta_n)$. Since $u_{n,\epsilon} \to u$ in $L^1(\Omega_n)$ as $\epsilon \to 0$, we get that $u= 0$ a.e. in $\Omega_n$. Finally, as $n$ is arbitrary and $\cup_{n \in \mathbb{N}}\Omega_n=\Omega$, we deduce that $\mathrm{NLBV}(\Omega)=\{0\}$, from which the proof is complete.

\end{proof}

\section{Applications to variational problems}\label{section_applications}
This section is devoted to the proofs of Theorems \ref{Theorem_main_functional} and \ref{Theorem_main_existence_Sobolev}, which rely on the theory that has been developed in Subsections \ref{section_NL_Sobolev} and \ref{section_NLBV} and follow the standard scheme of the Direct Method. The reader is referred to those subsections as well as the introduction (Section \ref{IntroductionSection}) for notations. Different variants of the problems under consideration can also be treated with analogous approaches. For instance, using the definition of traces as in Propositions \ref{Proposition_trace_Sob} and \ref{Proposition_trace_NLBV} one can also impose Dirichlet boundary conditions.
\begin{proof}[Proof of Theorem \ref{Theorem_main_functional}]
    As $F$ is proper and bounded from below, $-\infty < m_* < +\infty$, where, recall that
    \begin{equation}
        m_*=\inf_{u\in\mathrm{NLBV}_\omega(\Omega)\cap L^q(\Omega)} I_\omega(u).
    \end{equation}

    Let $(u_n)_{n \in \mathbb{N}}$ be a minimizing sequence for $I_\omega$ in $\mathrm{NLBV}_\omega(\Omega)\cap L^q(\Omega)$. As $F$ is coercive in $L^q(\Omega)$, $(u_n)_{n \in \mathbb{N}}$ is bounded in $L^q(\Omega)$ and, since $L^q(\Omega)$ is reflexive, this implies that there exists $u_* \in L^q(\Omega)$ such that $u_n \rightharpoonup u_*$ weakly in $L^q(\Omega)$ up to an extraction. Since $\Omega$ is of finite measure, we have the embedding $L^q(\Omega) \hookrightarrow L^1(\Omega)$, and therefore $u_n \rightharpoonup u_*$ in $L^1(\Omega)$ also. By Lemma \ref{lemma:lsc_nlbv} we thus have that $u_* \in \mathrm{NLBV}_\omega(\Omega)$ and
    
    \begin{equation}
        \mathrm{NLTV}_\omega(u_*) \leq \liminf_{n \to \infty} \mathrm{NLTV}_\omega(u_n).
    \end{equation}

    Moreover, as $F$ is (strongly) lower semicontinuous and convex, it is weakly lower semicontinuous. This fact, combined with the previous statement and the superadditivity of the limit inferior yields that
    \begin{equation} \label{ExistenceNLBVChain}
        m_*\le I(u_*)=\mathrm{NLTV}_\omega(u_*)+F(u_*)\leq \liminf_{n \to +\infty} \mathrm{NLTV}_\omega(u_n)+\liminf_{n \to +\infty} F(u_n) \le \liminf_{n \to +\infty} I(u_n) = m_*,
    \end{equation}
    from which the existence statement is proven. If $F$ is strictly convex, so is $I_\omega$ and the uniqueness statement follows.
\end{proof}

\begin{remark}[The $q=1$ case and relaxing convexity]\label{remark_existence_NLBV}
    Notice that if $F$ is assumed to be lower semicontinuous and coercive in $L^1(\Omega)$, then the existence part of the proof of Theorem \ref{Theorem_main_functional} does not carry out due to the fact that $L^1(\Omega)$ is not reflexive. Moreover, by considering the limit case $\omega=0$ in $\Omega \times \Omega$ one then checks that the result might then become false. However, if one assumes that $\Omega$ is a bounded extension domain and takes $\omega$ such that $c \le \omega$ in $\Omega\times\Omega$ for a certain constant $c>0$, then, by Theorem \ref{Theorem_main_equality_NLBV} and the well known embeddings between $\mathrm{BV}$ and $L^p$ spaces (see, e.g., \cite[Corollary 3.49]{AmbrosioFuscoPallara}), one has the continuous embeddings

    \begin{equation}
        \mathrm{NLBV}_\omega(\Omega) \hookrightarrow \mathrm{BV}(\Omega) \hookrightarrow L^q(\Omega) \hspace{2mm}\mbox{ for all } q \in [1,N/(N-1)),
    \end{equation}
    the last one being compact, with the understanding that for $N=1$ the critical exponent becomes $+\infty$. Thus, one can then adapt the proof of Theorem \ref{Theorem_main_functional} for this alternative set of assumptions. Indeed, the chain of embeddings yields that a minimizing sequence must have a subsequence strongly converging in $L^1(\Omega)$, and as $F$ is lower semicontinuous the chain of inequalities \eqref{ExistenceNLBVChain} holds true. Moreover, notice that in this case the convexity assumption on $F$ is not needed for the existence statement.
\end{remark}

\begin{proof}[Proof of Theorem \ref{Theorem_main_existence_Sobolev}]
    As $F$ is proper and bounded from below, $-\infty < m_{*,p} < + \infty$, where, recall that
    \begin{equation}
        m_{*,p}=\inf_{u\in W_\omega^{1,p}(\Omega)\cap L^q(\Omega)} I_\omega^p(u).
    \end{equation}

    Let $(u_n)_{n \in \mathbb{N}}$ be a minimizing sequence for $I_\omega^p$ in $W^{1,p}_\omega(\Omega)\cap L^q(\Omega)$. As argued previously, the coercivity of $F$ in $L^q(\Omega)$ implies that $(u_n)_{n \in \mathbb{N}}$ converges subsequentially to a certain $u_{*,p} \in L^q(\Omega)$ weakly in $L^q(\Omega)$. Since, by hypothesis, either $q = p$, or $q>p$ and $\Omega$ is of finite measure, we have the continuous embedding $L^q(\Omega) \hookrightarrow L^p(\Omega)$, and therefore $(u_n)_{n \in \mathbb{N}}$ converges weakly to $u_{*,p}$ in $L^p(\Omega)$ also. As a consequence, $(u_n)_{n \in \mathbb{N}}$ is uniformly bounded in $W^{1,p}_\omega(\Omega)$. Propositions \ref{Proposition_Banach_Sob} and \ref{Proposition_reflexives} show that $W^{1,p}_\omega(\Omega)$ is a reflexive Banach space. Thus, up to an extraction, we find that $u_n \rightharpoonup u_{*,p}$ weakly in $W^{1,p}_\omega(\Omega)$. In particular, $\nabla_\omega u_n \rightharpoonup \nabla_\omega u_{*,p}$ weakly in $L^p(\Omega \times \Omega, \mathbb{R}^N)$. Moreover, as the $L^p(\Omega\times\Omega;\mathbb{R}^N)$ norm is convex and continuous, it is also weakly lower semicontinuous. Therefore, using the lower semicontinuity and convexity of $F$ along with the superadditivity of the limit inferior we can derive a computation similar to \eqref{ExistenceNLBVChain}, and it follows that $u_{*,p}$ is a minimizer for $I_\omega^p$ on $W^{1,p}_\omega(\Omega)$. If $F$ is strictly convex, so is $I_\omega^p$ and the uniqueness statement follows.
\end{proof}

\begin{remark}[The $q < p$ and $p=1$ cases] \label{Remark_qlessthanp}
    If $q \in (1,p)$ and for $\Omega$ of finite measure, one can modify the proof of Theorem \ref{Theorem_main_existence_Sobolev} and show existence of a minimizer for $I_\omega^p$ in the space
    
    \begin{equation}
        \{ u \in L^q(\Omega): \nabla_\omega u \in L^p(\Omega \times \Omega, \mathbb{R}^N)\}.
    \end{equation}

    The $p=1$ case can be handled as in the $\mathrm{NLBV}_\omega$ setting, see Remark \ref{remark_existence_NLBV}.
\end{remark}

\begin{remark}[The $p=+\infty$ case]\label{remark_p_infinity}
    Notice that Theorem \ref{Theorem_main_existence_Sobolev} does not contain the $p=+\infty$ case. By Propositions \ref{Proposition_Banach_Sob} and \ref{Proposition_reflexives}, one can extend the proof of Theorem \ref{Theorem_main_existence_Sobolev} to this setting by using the weak$^*$ topology of $W^{1,\infty}_\omega(\Omega)$. However, $F$ has then to be assumed to be sequentially lower semicontinuous with respect to the weak$^*$ topology of $L^{\infty}(\Omega)$, which is rather restrictive. 
\end{remark}

\bibliographystyle{abbrv}
\bibliography{Biblio}
\end{document}